\documentclass[a4paper, oneside]{amsart}
\usepackage[colorlinks]{hyperref}

\usepackage[T1]{fontenc}
\usepackage[utf8]{inputenc}

\usepackage{enumitem}
\usepackage{float}
\usepackage{mathrsfs}
\usepackage{mathbbol}
\usepackage[margin=16pt,font=small,labelfont=bf]{caption}
\usepackage{xcolor}
\usepackage{graphicx}

\title{Metastability for the Ising model on the hexagonal lattice} 

\author{Valentina Apollonio}
\address[Valentina Apollonio]{Dipartimento di Matematica e Fisica, 
              Università Roma Tre}

\author{Vanessa Jacquier}
\address[Vanessa Jacquier]{Dipartimento di Matematica ``Ulisse Dini'',
              Università degli studi di Firenze
}

\author{Francesca Romana Nardi}
\address[Francesca Romana Nardi]{Dipartimento di Matematica ``Ulisse Dini'',
              Università degli studi di Firenze and 
              Faculteit Wiskunde en Informatica, 
              Technische Universiteit Eindhoven}

\author{Alessio Troiani}
\address[Alessio Troiani]{Dipartimento di Matematica ``Tullio Levi-Civita'',
              Università degli Studi di Padova
}

 % Separate items with ;

 % Edit. Separate items with ;

\newtheorem{theorem}{Theorem}[section]
\newtheorem{corollary}[theorem]{Corollary}
\newtheorem{lemma}[theorem]{Lemma} 
\newtheorem{definition}[theorem] {Definition}

\newtheorem{condition}[theorem] {Condition}
\newtheorem{remark}[theorem]{Remark}
\newtheorem{proposition}[theorem]{Proposition}
\newtheorem{notation}[theorem]{Notation}

\newtheorem{alg}[theorem]{Algorithm}

\newcommand{\area}[1]{\lVert #1 \rVert}
\newcommand{\areamax}[1]{A^{\max}(#1)}
\newcommand{\quasiregularset}{\ensuremath{\mathcal{Q}}}

\usepackage{parskip}

\begin{document}

\begin{abstract}
We consider the Ising model on the hexagonal lattice 
evolving according to Metropolis dynamics. 
We study its metastable behavior in the limit 
of vanishing temperature when the system is immersed in
a small external magnetic field. We determine
the asymptotic properties of the transition time from the 
metastable to the stable state up to a 
multiplicative factor and study the mixing time and 
the spectral gap of the Markov process. 
We give a geometrical description of the critical configurations 
and show how not only their size but their shape varies depending 
on the thermodynamical parameters. 
Finally we provide some results concerning polyiamonds 
of maximal area and minimal perimeter.
\end{abstract}

\maketitle

\section{Introduction}
A thermodynamical system, subject to a \emph{noisy dynamics}, exhibits metastable behavior
when it remains for a long time in the vicinity of a state that is a local minimum of the energy before reaching a more stable state through a sudden transition. On a short time scale, the system behaves as if it were in equilibrium whereas, on a long time scale, it moves between regions of its state space.
This motion, linked to first order phase transitions, is triggered by the appearance of a \emph{critical} microscopic configuration of the system via a spontaneous fluctuation or some external perturbation. 

Several termodynamical systems, such as magnets immersed into an external magnetic field,
supercooled liquids or supersaturated gases, may show metastability. However this phenomenon is
not exclusive of thermodynamical systems, but it appears in a plethora of diverse fields including
biology, chemistry, computer science, economics.

Given the peculiar features of metastability outlined above, when studying the metastable behavior
of a system one is typically interested in studying the properties of the \emph{transition time} towards the stable state, the features of the critical configurations and the characterization of typical paths along which the transition takes place.
This investigation has been carried over, in the
literature, using mainly two different approaches:
pathwise (\cite{cerf2013nucleation,dehghanpour1997metropolis,kotecky1993droplet, manzo2004essential, manzo1998relaxation, manzo2001dynamical, schonmann1994slow, schonmann1998wulff, nardi1996low,  olivieri2005large, jovanovski2017metastability}) and potential
theoretic (\cite{bovier2016metastability, bovier2010homogeneous, gaudilliere2020asymptotic,bashiri2017note}).
More recently, other techniques have been used in \cite{beltran2010tunneling, beltran2012tunneling, gaudillierelandim2014} and in \cite{bianchi2016metastable}.

In the context of metastability
the stochastic Ising model, on the square and on the
cubic lattice, evolving according to Glauber dynamics has been one of the main subjects of investigation
(see, for instance, \cite{arous1996metastability, kotecky1994shapes, neves1991critical, neves1992behavior}).

In this paper we consider the Ising model on the hexagonal lattice at very low temperature, 
with isotropic interactions
and in presence of a weak external magnetic field. We let the system evolve according to Glauber dynamics.

We identify both the stable and metastable state and we study the transition from the metastable state to the stable one.
Since the system is at very low temperature, the transition probabilities of the dynamics are exponentially small. 
We compute the asymptotic expected value of the first hitting time 
of the stable state up to a multiplicative factor and
determine its probability distribution. 
Moreover we provide an estimate 
of the mixing time of the chain and its spectral gap.

These results are obtained leveraging on both
the pathwise and the potential theoretic approach
widely adopted in the literature.
However the shape of the droplets visited by
the typical trajectories from the metastable to
the stable state are heavily dependent on the geometry of the underlying lattice.

We give a 
characterization of the critical configurations triggering the transition and
show that they exhibit a \emph{gate} property. In particular, it is shown that, along with their size, also the shapes of the critical configurations are dependent on the relations between the parameters of the system. This characterization is achieved through a geometrical description of the spin configurations obtained by associating each cluster of spins with a polyiamond, that is a collection of faces of the triangular lattice.

To this end we provide some results concerning polyiamonds of minimal perimeter and maximal area, complementing those already present in the literature (e.g. \cite{FuSie,nagy2013isoperimetrically, davoli2017sharp}).
In particular we show that quasi-hexagonal polyiamonds are the unique shapes that maximize the area for fixed perimeter and minimize perimeter for fixed area for different notions of perimeter.
Recently, some of the authors of this paper studied a class of parallel dynamics (shaken dynamics in \cite{apollonio2019shaken,apollonio2019criticality}) connecting the Ising model on arbitrary graph and the Ising model on suitable bipartite graph. In particular, it has been shown how the shaken dynamics on the square lattice induces a collection of parallel dynamics on a family of Ising models on the hexagonal lattice with non-isotropic interaction where the spins in each of the two partitions are alternatively updated. This work, therefore, wants to serve also as a springboard to tackle the study of metastability for parallel dynamics on a whole family of hexagonal lattices analogously to what has been done for the homogeneous square lattice in \cite{cirillo2002note, cirillo2003metastability}, see also \cite{bet2020effect, cirillo2008competitive, cirillo2008metastability, nardi2012sharp} for other examples.

In Section~\ref{sec:model} we define the model and provide
the main results. Section~\ref{geom1} is devoted to
the description of the Ising configuration
in terms of clusters and these are linked to polyiamonds.
In Section~\ref{recurrenceproperty} we prove the theorems concerning the recurrence of the system to either the
stable or the metastable state whereas in
Section~\ref{bound} we identify the reference path.
The other theorems linked to the metastable
behavior of the system are proven in Section~\ref{sec:other_theorems}.
Finally, in Section~\ref{geom2}, we give our results concerning polyiamonds with the intent of providing 
a self contained set of tools that may be of use whenever
the volume-surface competition plays a role in determining
the properties of a statistical mechanics system living
on the hexagonal lattice.

\section{Model description and main results}\label{sec:model}
\subsection{Ising model on hexagonal lattice}
Consider the discrete hexagonal lattice $\mathbb{H}^2$ embedded in $\mathbb{R}^2$ and let $\mathbb{T}^{2}$ be its dual ($\mathbb{T}^{2}$ is, therefore, a triangular lattice). Two sites of the discrete hexagonal lattice are said to be \emph{nearest neighbors} when they share an edge of the lattice, see Figure \ref{neighbors}. 

Let $\Lambda$ be the subset of $\mathbb{H}^{2}$ obtained by
cutting a parallelogram of side length $L$ along two of the coordinate
axes of the triangular lattice. On $\Lambda$ defined as such we impose periodic
boundary conditions.  Note that $\Lambda$ contains $2L^2$ sites.
To each site $i\in\Lambda$ we associate a spin variable $\sigma(i)\in \{-1,+1\}$. We interpret $\sigma(i)=+1$ (respectively $\sigma(i)=-1$) as indicating that the spin at site $i$ is pointing upwards (respectively downwards). 
\begin{figure}[htb!]
\centering
    \includegraphics[scale=0.8]{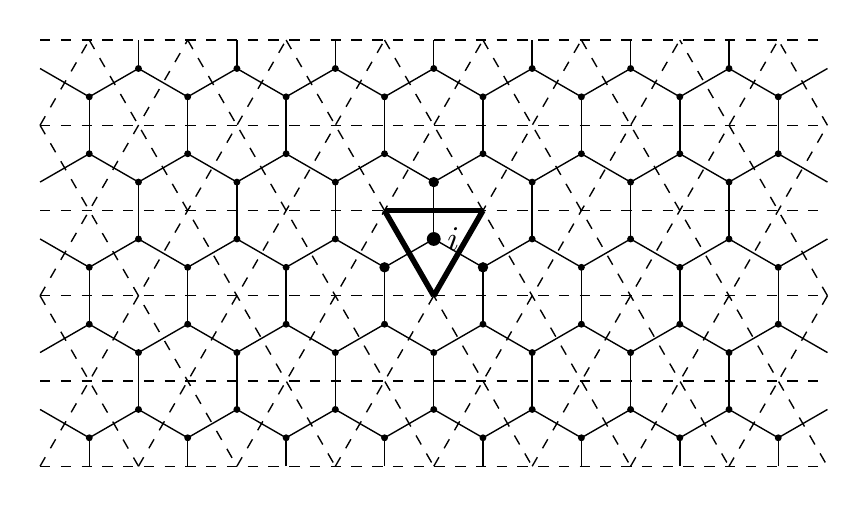}
    \caption{The solid lines show the hexagonal lattice,
            whereas the dashed lines show its dual, the triangular lattice. The solid triangle
            highlights the triangular face centered
            at site $i$. The thicker vertices are the
            nearest neighbors of site $i$ on the 
            hexagonal lattice.
    }
  \label{neighbors}
    \end{figure}
On the \emph{configuration space} $\mathcal{X}:=\{-1,+1\}^{\Lambda}$ we consider the \emph{Hamiltonian function} $H: \mathcal{X} \longrightarrow \mathbb{R}$ defined as
\begin{equation}\label{hamiltonianFunction}
H(\sigma):=-\frac{J}{2}\sum_{\substack{i,j \in \Lambda\\ d(i, j)=1}} \sigma (i) \sigma (j) -\frac{h}{2} \sum_{i \in \Lambda} \sigma (i),
\end{equation} 
where $J>0$ represents the ferromagnetic interaction between two spins, $h>0$ is the external magnetic field and $d(\cdot, \cdot)$ is the lattice distance on $\mathbb{H}^2$. 
We will consider the case $h \in (0,1)$ 
where, as it will be shown, the system exhibits a metastable behavior.
Throughout the paper we will assume that $\frac{J}{2h}-\frac{1}{2}$ is not integer.

We consider a Markov chain  $(X_t)_{t \in \mathbb{N}}$ on $\mathcal{X}$ 
defined via the so called \emph{Metropolis Algorithm}.
The transition probabilities of this dynamics are given by
\begin{equation}\label{transitionprob}
    P(\sigma, \eta)=q(\sigma,\eta) e^{-\beta[H(\eta) -H(\sigma)]_+}, \qquad \text{for all } \sigma \neq \eta,
\end{equation}
where $[\cdot]_+$ denotes the positive part and $q(\sigma,\eta)$ is a connectivity matrix
independent of $\beta$, defined, for all $\sigma \neq \eta$, as
\begin{equation}
    q(\sigma,\eta)= \left\{
    \begin{array}{ll}
    \frac{1}{|\Lambda|} & \;\;\textrm{ if } \exists \; x\in \Lambda: \sigma^{(x)}=\eta\\
    0& \;\;\textrm{ otherwise }
    \end{array}
    \right.
\end{equation}
where
\begin{equation}
    \sigma^{(x)}(z)= \left\{
    \begin{array}{ll}
        \sigma(z) & \;\;\textrm{ if } z \neq x\\
        -\sigma(x) & \;\;\textrm{ if } z = x
        \end{array}
    \right.
\end{equation}

Table \ref{Tab} shows all possible single spin flip probabilities.

\begin{table}[H]
    \centering
    \includegraphics[scale=0.8]{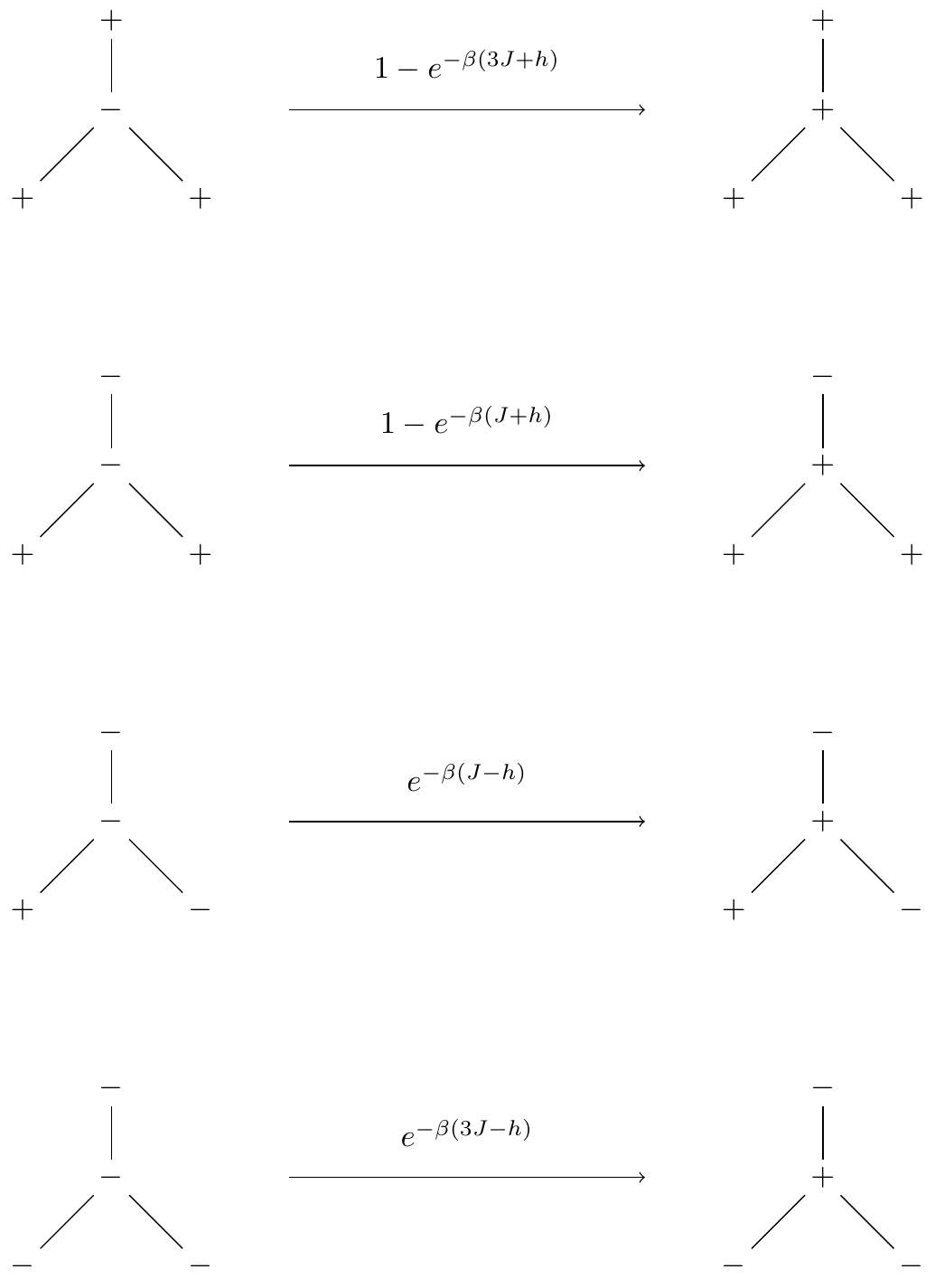}
    \caption{
        Transition probabilities when the local configuration
        on the right is obtained from the local configuration on the left by changing the value
        of the central spin for all possible values of the 
        neighboring spins. The probability that the change happens at the site at the center is uniform on all sites of $\Lambda$.
       }
    \label{Tab}
\end{table}

It is possible to check that $(X_t)_{t \in \mathbb{N}}$  is an ergodic-aperiodic Markov chain  on $\mathcal{X}$  satisfying the detailed balance condition
\begin{equation}\label{reversibility}
    \mu(\sigma)P(\sigma, \eta)=\mu(\eta)P(\eta, \sigma),
\end{equation}
with respect to the Gibbs measure 
\begin{equation}\label{def:gibbs}
    \mu(\sigma)=\frac{e^{-\beta H(\sigma)}}{\sum_{\eta \in \mathcal{X}} e^{-\beta H(\eta)}},
\end{equation}
where $\beta:=\frac{1}{T} >0$ is the inverse temperature.

Let
\begin{itemize}[label=\raisebox{0.25ex}{\tiny \textbullet}]
    \item $\underline{+1}$ be the configuration $\sigma$ such that $\sigma(i)=+1$ for every $i\in\Lambda$;
    \item $\underline{-1}$ be the configuration $\sigma$ such that $\sigma(i)=-1$ for every $i\in\Lambda$;
\end{itemize}
and assume, in the remainder,  $J\gg h>0$, for $h$ fixed. Under periodic boundary conditions, the energy of these configurations is, respectively
\begin{itemize}[label=\raisebox{0.25ex}{\tiny \textbullet}]
    \item $H(\underline{+1}) = -2{L^2}(3J+h)$,
    \item $H(\underline{-1}) = -2{L^2}(3J-h)$.
\end{itemize}

It is straightforward to check that
$\underline{+1}$ maximizes both
sums in \eqref{hamiltonianFunction} and, consequently,
we have the following

\begin{lemma}\label{statostabile} $\underline{+1}$ is the global minimum (or ground state) of the Hamiltonian \eqref{hamiltonianFunction}.
\end{lemma}

In the remainder we will show that $\underline{-1}$ is the unique \emph{metastable} state, that is the  \emph{deepest local minimum} of the Hamiltonian (see Theorem \ref{Identification}).

\subsection{Metastability: definitions and notation}\label{defandnot}
The problem of metastability is the study of the first arrival of the process 
$(X_t)_{t \in \mathbb{N}}$
to the set of the stable states, corresponding to the set of absolute minima of $H$, starting from an initial local minimum. Local minima can be ordered in terms of their increasing stability level, i.e., the height of the barrier separating them from lower energy states. More precisely, for any $\sigma \in \mathcal{X}$, let $\mathcal{I}_{\sigma}$ be the set of configurations with energy strictly lower than $H(\sigma)$, i.e.,
\begin{equation}\label{I}
\mathcal{I}_{\sigma}:=\{\eta\in \mathcal{X} \,|\, H(\eta)<H(\sigma)\}.
\end{equation}
Let $\omega=\{\omega_1,\ldots,\omega_n\}$ be a finite sequence of configurations in $\mathcal{X}$, where, for each $k$ from $1$ to $n-1$, $\omega_{k+1}$ is obtained from $\omega_k$ by a single spin flip. We call $\omega$ a \emph{path} with starting configuration $\omega_1$ and final configuration $\omega_n$ and we denote the set of all these paths as $\Theta(\omega_1,\omega_n)$. We indicate the length of $\omega$ as $|\omega|=n$. We call \emph{communication height} between two configurations $\sigma$ and $\eta$ the maximal height along the minimal path in $\Theta(\sigma,\eta)$, i.e.,
\begin{equation}\label{minmax}
\Phi(\sigma,\eta):=\min_{\omega\in\Theta(\sigma,\eta)}\max_{\zeta \in \omega} H(\zeta).
\end{equation}
By $\Phi(\omega)$ we denote the communication height along the path 
\mbox{$\omega=\{\omega_1,\ldots,\omega_n\}$}, i.e. \linebreak 
\mbox{$\Phi(\omega)=\max_{i=1,\ldots,n} H(\omega_i)$}. 
Similarly, we also define the \emph{communication height} 
between two sets $A, B \subset \mathcal{X}$ as
\begin{equation}
\Phi(A,B):=\min_{\sigma \in A,\eta \in B} \Phi(\sigma,\eta).
\end{equation}
Now we are able to formally define the stability level as
\begin{equation}
V_{\sigma}:=\Phi(\sigma,\mathcal{I}_{\sigma})-H(\sigma).
\end{equation}
If $\mathcal{I}_{\sigma}$ is empty, then we define $V_{\sigma}=\infty$. We denote by $\mathcal{X}^s$ the set of global minima of the energy, and we refer to these as ground states. To define the set of metastable states, we introduce the \emph{maximal stability level} 
\begin{equation}\label{Gamma}
    \Gamma_m:=\max_{\sigma\in \mathcal{X}\setminus \mathcal{X}^s}V_{\sigma}.
\end{equation}

The metastable states are those that attain the maximal stability level $\Gamma_m< \infty$, that is 

\begin{equation}\label{Xm}
    \mathcal{X}^m:=\{y\in \mathcal{X}| \, V_{y}=\Gamma_m \}.
\end{equation}

Since the metastable states are defined in terms of their stability level, a crucial role in our proofs is played by the set of all configurations with stability level strictly greater than $V$, that is 

\begin{equation}\label{Xv}
\mathcal{X}_V:=\{x\in \mathcal{X} \,\, | \,\, V_{x}>V\}.
\end{equation}

We frame the problem of metastability as the identification of metastable states 
and the computation of transition times from the metastable states to the stable ones. 
To study the transition between $\mathcal{X}^m$ and $\mathcal{X}^s$, 
we define the \emph{first hitting time} of 
$A\subset \mathcal{X}$ starting from $\sigma \in \mathcal{X}$
\begin{equation}\label{fht}
    \tau^{\sigma}_A:=\inf\{t>0 \,|\, X_t\in A\}.
\end{equation}
Whenever possible we shall drop the superscript denoting the
starting point $\sigma$ from the notation and we denote by $\mathbb{P}_{\sigma}(\cdot)$ and $\mathbb{E}_{\sigma}[\cdot]$ respectively the probability and the average along the trajectories of the process started at $\sigma$. 

Now we define formally the \emph{energy barrier} $\Gamma$ as
\begin{equation}
    \Gamma:=\Phi(m, s)-H(m) \qquad \text{with } m \in \mathcal{X}^m, \, s \in \mathcal{X}^s.
\end{equation}
 In what follows we consider the set of paths realizing the minimal value of the maximal energy in the paths between
 any metastable state and the set of the stable states. 
 To this end, we define the set of \emph{optimal paths}. 
\begin{definition}
    We write $(\mathcal{A} \to \mathcal{B})_{opt}$ to denote 
    the set of \emph{optimal paths}, i.e., the set of all paths 
    from $\mathcal{A}$ to $\mathcal{B}$ realizing the min-max \eqref{minmax} 
    in $\mathcal{X}$ between $\mathcal{A}$ and $\mathcal{B}$.
\end{definition}
Another basic notion is the set of \emph{saddles} defined as 
the set of all maxima in the optimal paths between two configurations.  
\begin{definition} 
The set of \emph{minimal saddles} between $\sigma, \eta \in \mathcal{X}$ is defined as
\begin{equation}
    \mathscr{S}(\sigma,\eta):=\{\zeta \in \mathcal{X} \, | \, \exists \, \omega:\sigma \to \eta, \, \omega \ni \zeta \text{ such that } \max_{\xi \in \omega} H(\xi)=H(\zeta)=\Phi(\sigma,\eta)\}.
\end{equation}
\begin{equation}
\mathscr{S}(\mathcal{A},\mathcal{B}):= \bigcup_{\substack{\sigma \in \mathcal{A}, \, \eta \in \mathcal{B}: \\ \Phi(\sigma,\eta)=\Phi(\mathcal{A},\mathcal{B})}} \mathscr{S}(\sigma,\eta).
\end{equation}
\end{definition}
We focus on the subsets of saddles that are typically visited during the last excursion from a metastable state to the set of the stable states. To this end, we introduce the \emph{gates} from metastability to stability, defined as the subsets of $\mathscr{S}$ visited by all the optimal paths. More precisely,
\begin{definition}\label{gates}
Given a pair of configurations $\sigma, \eta \in \mathcal{X}$, we say that $\mathcal{W} \equiv \mathcal{W}(\sigma, \eta)$ is a \emph{gate} for the transition from $\sigma$ to $\eta$ if $\mathcal{W}(\sigma, \eta) \subseteq \mathscr{S}(\sigma,\eta)$ and $\omega \cap \mathcal{W} \neq \emptyset$ for all $\omega \in (\sigma \to \eta)_{opt}$.  
\end{definition}
Moreover,
\begin{definition}\label{minimalgate}
A gate $\mathcal{W}$ is a \emph{minimal gate} for the transition from $\sigma$ to $\eta$ if for any $\mathcal{W}' \subset \mathcal{W}$ there exists $\omega' \in (\sigma \to \eta)_{opt}$ such that $\omega' \cap \mathcal{W}' = \emptyset$. 

For a given pair $\eta, \eta'$, there may be several disjoint minimal gates. We denote by $\mathscr{G}(\eta, \eta')$ the union of all minimal gates:
\begin{equation}
    \mathscr{G}(\eta,\eta'):= \bigcup_{\mathcal{W}: \, \text{minimal gate for } (\eta, \eta')} \mathcal{W}
\end{equation}
Obviously, $\mathscr{G}(\sigma,\sigma') \subseteq \mathscr{S}(\sigma,\sigma')$ and $\mathscr{S}(\sigma,\sigma')$ is a gate (but in general it is not minimal). The configurations $\xi \in \mathscr{S}(\eta,\eta') \setminus \mathscr{G}(\eta,\eta')$ (if any) are called \emph{dead ends}.
\end{definition}
In words, a minimal gate is a minimal (by inclusion) subset of $\mathscr{S}(\sigma,\eta)$ that is visited by all optimal paths.
The configurations in the \emph{minimal gates} have the physical meaning of \emph{critical configurations} and are central objects both from a probabilistic and from a physical point of view.

To study the function $H(\sigma)$ it is convenient
to associate to each configuration
$\sigma \in \mathcal{X}$
certain geometrical objects and, then, to study
their properties.

To this end, recall that $\mathbb{H}^2$ is the discrete hexagonal lattice embedded in $\mathbb{R}^2$ and denote
by $\mathbb{T}^2$ its dual, i.e. $\mathbb{T}^2$ is the discrete triangular lattice embedded in
$\mathbb{R}^2$.
Given a configuration $\sigma \in \mathcal{X}$, consider the set $C(\sigma) \subseteq \Lambda$ defined as the union of the closed triangular faces centered at sites $i$ with the boundary contained in $\mathbb{T}^2$ and such that $\sigma(i)=+1$ (see Figure~\ref{neighbors}) and look at the maximal connected components $C_{1}, \ldots, C_{m}, m \in \mathbb{N},$ of $C(\sigma)$.
If a maximal connected component wraps around the torus it is called a \emph{plus strip}, otherwise it is called a \emph{cluster (of pluses)}.
This construction leads to a bijection that associates to each configuration a collection of its clusters and plus strips. 
Likewise, other geometrical objects may be associated to a configuration $\sigma$ by
considering the connected components of triangular faces centered at the sites of the lattice 
with spin value minus one. 
Among these, there could be a connected component which contains two or three lines that wrap 
around the torus parallel to the coordinate axes of $\mathbb{T}^2$. 
If this is the case, the component is called a \emph{sea of minuses}. 
Similarly, if there is only one line that wraps around the torus we call it a \emph{minus strip}. 
The other connected components of triangular faces centered at minus spins are called \emph{holes}. 

Given a configuration $\sigma \in \mathcal{X}$ we denote by $\gamma(\sigma)$ its Peierls contour that is the boundary of the clusters. 
Note that Peierls contours live on the dual lattice $\mathbb{T}^2$ and are the union of 
piecewise linear curves separating spins with opposite sign in $\sigma$. 
In particular in each dual vertex there are 0, 2, 4, 6 dual bonds contained in $\gamma(\sigma)$. 

In this setting, it is immediate to see that for each configuration $\sigma$ we have
\begin{align}\label{eq:peierls_hamiltonian}
    H(\sigma)-H(\underline{-1})=J|\gamma(\sigma)|-h N^{+}(\sigma),
\end{align}
where
\begin{align}\label{eq:number_of_pluses}
    N^{+}(\sigma)=\sum_{x \in \Lambda}\frac{\sigma(x)+1}{2},
\end{align}
represents the number of plus spins.
In this way the 
energy of each configuration is associated to the
area and the length of the boundary of a
suitable collection of triangular faces.

Call $r^*$ the \emph{critical radius}:
\begin{equation}\label{raggiocritico}
    r^*:=\big\lfloor \frac{J}{2h}-\frac{1}{2} \big\rfloor, 
    %\qquad  \text{ thus } r^*=\frac{J}{2h}-\frac{1}{2}-\delta \text{ with } 0<\delta <1.
\end{equation}
and let $\delta \in (0,1)$ be the fractional part of $\frac{J}{2h}-\frac{1}{2}$, that is
\begin{equation}
    \delta = \frac{J}{2h}-\frac{1}{2} - r^*.
\end{equation}
We will show that for our model the energy barrier $\Gamma$ is equal to
\begin{equation}\label{GammaH}
    \Gamma^{Hex}:=
    \begin{cases}
        -6{r^*}^2h+6r^*J-10r^*h+7J-5h & \text{if } 0<\delta<\frac{1}{2} \\
        -6(r^*+1)^2h+6(r^*+1)J-2(r^*+1)h+3J-h & \text{if } \frac{1}{2}<\delta<1
    \end{cases}
\end{equation}

The value of $\Gamma^{Hex}$ is obtained by 
computing the energy of the \emph{critical configurations}.
We will see that these configurations consist of a cluster having a shape
that is close to a hexagon of radius $r^*$ and, in particular,
we will compute the $\emph{critical area}$ to be
\begin{align}\label{areacritica}
        A^*_1=& 6{r^*}^2+10r^*+5 & \text{if } 0<\delta<\frac{1}{2}, \notag \\
        A^*_2=& 6(r^*+1)^2+2(r^*+1)+1 & \text{if } \frac{1}{2}<\delta<1.
\end{align}
\begin{figure}[htb!] 
\centering \includegraphics[scale=0.5]{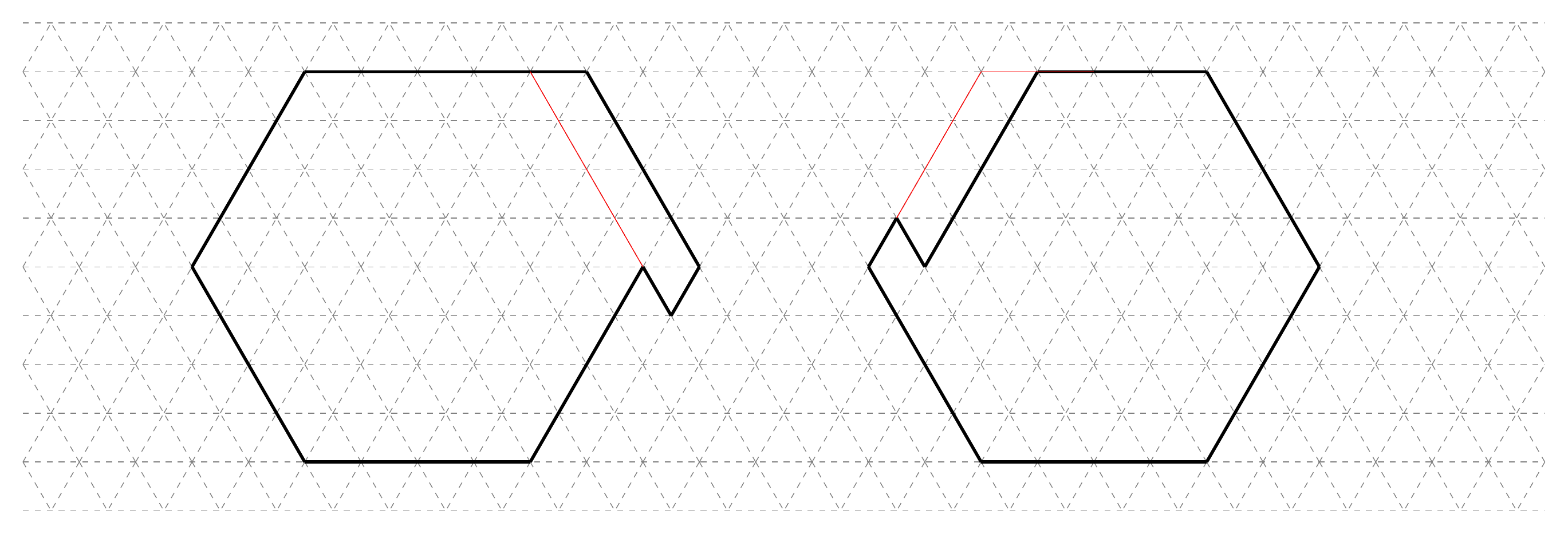} 
\caption{On the left, $S(A^*_2)$ with $A^*_2=6(r^*+1)^2+2(r^*+1)+1$ and $\frac{1}{2}<\delta<1$. On the right, $S(A^*_1)$ with $A^*_1=6{r^*}^2+10r^*+5$ and $0<\delta<\frac{1}{2}$.} \label{STANDARDcritico}
\end{figure}

\subsection{Main results}
Our results concerning the metastable behavior of the model are given
under the assumption that the torus is large compared to the size
of the critical clusters. More precisely, throughout the paper we assume
the following
\begin{condition}\label{condizionetoro}
The magnetic field $h$,  the ferromagnetic interaction $J$ and the torus $\Lambda$ are such that $0<h<1$, $J \geq 2h$, $\frac{J}{2h}-\frac{1}{2}$ is not integer, and $|\Lambda| \geq (\frac{4J}{h})^2$ is finite.
\end{condition}
Note that the assumption $\frac{J}{2h}-\frac{1}{2}$ not integer
is made so to avoid strong degeneracy of the critical configurations.
Similar assumptions are common in literature (see e.g., \cite{arous1996metastability, boviermanzo2002metastability, cirillo2008metastability, cirillo2013relaxation}).
Further observe that Theorems~\ref{Identification},~\ref{teotime'}, \ref{TMIX} and  Remark~\ref{transitiontime}
are expected to hold
also in the case of free boundary conditions.
As far as Theorem~\ref{thm:sharp_estimate} is concerned, then the estimate of the prefactor is a bit more delicate.However it is reasonable to think that an analogous result holds, at least asymptotically, as $|\Lambda \to \infty|$ (independently of $\beta$).

We say that a function $\beta \mapsto f(\beta)$ is super exponentially small (SES) if $$\lim_{\beta\to\infty}\frac{\log{f(\beta)}}{\beta}=-\infty.$$
With this notation we can state our first theorem concerning the recurrence of 
the system to either the state $\underline{-1}$ or $\underline{+1}$.

\begin{proposition}\label{teoRP} (Recurrence property). If $V^*=2J$, then \mbox{$\mathcal{X}_{V^*}=\{ \underline{-1}, \underline{+1} \}$} and for any $\epsilon>0$, the function
\begin{align}\label{recurrence2J}
    \beta \mapsto \sup_{\sigma \in \mathcal{X}} \mathbb{P}_{\sigma}(\tau_{\mathcal{X}_{V^*}}> e^{\beta(V^*+\epsilon)})
\end{align}
is SES.
\end{proposition}

Equation~\eqref{recurrence2J} implies that the system reaches with high probability
either the state $\underline{-1}$ (which is a local minimizer
of the Hamiltonian) or the  ground state in a time shorter than 
$e^{\beta (V^*+\epsilon)}$, uniformly in the starting configuration $\sigma$ for any $\epsilon >0$. In other words we can say that the dynamics speeded up by a 
time factor
of order $e^{\beta V^*}$ reaches with high probability $\{ \underline{-1}, \underline{+1} \}$. In Corollary \ref{XJ-h} we give a geometrical description of $\mathcal{X}_V$ for $V=J-h$ and we discuss the behavior of the speeded up dynamics by a time factor of order $e^{\beta (J-h)}$.

In the next theorem we identify the metastable state and we compute the maximal stability level. Recalling the
Definition of $\Gamma^{Hex}$ in Equation~\eqref{GammaH} we have

\begin{theorem}\label{Identification} (Identification of metastable state)
$\mathcal{X}^m=\{\underline{-1}\}$ and $\Gamma_m=\Gamma^{Hex}$.
\end{theorem}
In the next theorems, we give the asymptotic behavior (for $\beta \to \infty$) of the transition time for the system started at the metastable state. In particular, in Theorem \ref{thm:sharp_estimate} we estimate the expected value of the transition time and in Theorem \ref{teotime'} we give its asymptotic distribution. 
\begin{theorem}\label{thm:sharp_estimate} (Sharp estimates of $\tau_{\underline{+1}}$) For $\beta$ large enough, we have
\begin{equation}\label{valoreatteso}
     \mathbb{E}_{\underline{-1}}[\tau_{\underline{+1}}]=\frac{1}{k}e^{\beta\Gamma^{Hex}}(1+o(1)),
\end{equation}
 where
 \begin{equation}
k=
\begin{cases}
5(r^*+1) & \qquad \text{if } \delta \in (0,\frac{1}{2}), \\
10(r^*+1) & \qquad \text{if } \delta \in (\frac{1}{2},1).
%5(l-1) & \qquad \text{if } \delta \in (0,\frac{1}{2}), \\
%10(l-1) & \qquad \text{if } \delta \in (\frac{1}{2},1).
\end{cases}
\end{equation}
%l=r^*+2
\end{theorem}
\begin{theorem}\label{teotime'} (Asymptotic distribution of $\tau_{\underline{+1}}$)

Let $T_{\beta}:= \inf\{ n \geq 1 \, | \, \mathbb{P}_{\underline{-1}}(\tau_{\underline{+1}} \leq n) \geq 1-e^{-1}\}$ 
\begin{equation}\label{Ptime'}
    \lim_{\beta \to \infty} \mathbb{P}_{\underline{-1}}(\tau_{\underline{+1}}>tT_{\beta})=e^{-t}
\end{equation}
and 
\begin{equation}\label{Etime'}
    \lim_{\beta \to \infty} \frac{\mathbb{E}_{\underline{-1}}(\tau_{\underline{+1}})}{T_{\beta}}=1.
\end{equation}
\end{theorem}

The two previous Theorems imply the following result.

\begin{remark}\label{transitiontime} (Asymptotic behavior of $\tau_{\underline{+1}}$ in probability) For any $\epsilon>0$, we have
\begin{equation}
    \lim_{\beta \to \infty} \mathbb{P}_{\underline{-1}}(e^{\beta(\Gamma^{Hex}-\epsilon)}< \tau_{\underline{+1}}<e^{\beta(\Gamma^{Hex}+\epsilon)})=1.
\end{equation}
\end{remark}

To prove Theorem~\ref{teotime'} a rather detailed
knowledge of the geometry of the \emph{critical configurations}
is required. However, though the result of Remark~\ref{transitiontime} is weaker, it can be proven
independently of Theorem~\ref{teotime'} without such a detailed
study.

The following theorem gives an estimate of the mixing time and the spectral gap for our model.
\begin{theorem}\label{TMIX} (Mixing time and spectral gap) For any $0<\epsilon<1$ we have

\begin{equation}\label{lim2PCA}
 \lim_{\beta \rightarrow \infty}{\frac{1}{\beta}\log{ t^{mix}_\beta(\epsilon)}}=\Gamma^{Hex},
\end{equation}

 and there exist two constants $0<c_1<c_2<\infty$ independent of $\beta$ such that for every $\beta>0$

\begin{equation}\label{rocompresopca}
 c_1e^{-\beta(\Gamma^{Hex}+\gamma_1)} \leq \rho_{\beta} \leq c_2e^{-\beta(\Gamma^{Hex}-\gamma_2)},
\end{equation}

 where $\gamma_1,\gamma_2$ are functions of $\beta$ that vanish for $\beta\to\infty$, and $\rho_{\beta}$ is the spectral gap.
\end{theorem}

In the theorems below, we characterize the gate for the transition from $\underline{-1}$ to $\underline{+1}$.
To do this, we give an intuitive definition of the configurations denoted by $\mathcal{\tilde{S}}(A^*_i)$ and $\mathcal{\tilde{D}}(A^*_i)$ that play the role 
of $\emph{critical configurations}$.
$\mathcal{\tilde{S}}(A^*_i)$ is a configuration with a cluster such that its area is $A^*_i$ and its shape is that in Figure \ref{figselle} (a)-(c); $\mathcal{\tilde{D}}(A^*_i)$ is a configuration with a cluster such that its area is $A^*_i$ and its shape is that in Figure \ref{figselle} (b)-(d). 
We refer the reader to Notation~\ref{notation:standard_canonical_defective_configurations} for a precise definition of $\mathcal{\tilde{S}}(A^*_i)$ and to Corollary~\ref{Phimax} for the values of $A^*_i$ with $i \in \{1,2\}$. We observe that $\Gamma^{Hex}$ is equal to $H(\mathcal{\tilde{S}}(A^*_i))-H(\underline{-1})$. See Figure~\ref{STANDARDcritico} to have the two pictures of standard clusters with area $A^*_i$ according to the values of parameters. 

\begin{figure}[htb!]
\centering
    \includegraphics[width=\textwidth]{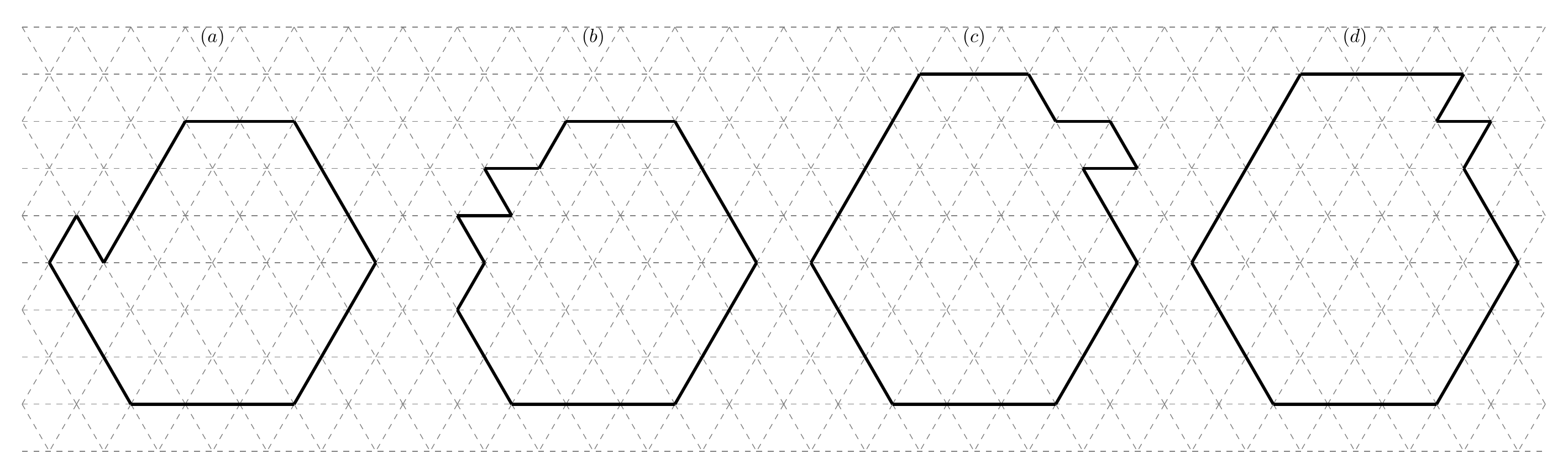}
    \caption{On the left there are two examples of two configurations $\mathcal{\tilde S}(A^*_1)$, $ \mathcal{\tilde D}(A^*_1)$ belonging to the gate for $\delta \in (0,1/2)$. On the right there are other two examples  $\mathcal{\tilde S}(A^*_2)$, $\mathcal{\tilde D}(A^*_2)$ for $\delta \in (1/2,1)$.}
  \label{figselle}
    \end{figure}

\begin{theorem}\label{selle} (Gate) Given $\delta \in (0,1)$, $A^*_i \in \{A^*_1, A^*_2 \}$ in \eqref{areacritica}. We have that any optimal path $\omega \in (\underline{-1} \to \underline{+1})_{opt}$ visits $\mathcal{\tilde S}(A^*_i) \cup  \mathcal{\tilde D}(A^*_i)$ , i.e. there exists an integer $j$ such that $\omega \ni \omega_j \equiv \mathcal{\tilde S}(A^*_i)$ or $\omega \ni \omega_j \equiv \mathcal{\tilde D}(A^*_i)$. In other words $\mathcal{\tilde S}(A^*_i) \cup \mathcal{\tilde D}(A^*_i)$ is the union of all minimal gates from $\underline{-1}$ to $\underline{+1}$. 
\end{theorem}

From a physical point of view, the configurations in the gate are those 
that must be visited, for a system at very low temperature, in order for 
the transition to the stable state to take place.
Moreover, the characterization of the gate allows to compute the sharp estimates
for the transition time in Theorem~\ref{thm:sharp_estimate} using a potential 
theoretic approach.
Using solely a pathwise approach, exponential asymptotics 
for the expected transition time could have been obtained as well
without this detailed description of the gate. 
To this purpose, using the model dependent results
of Theorem~\ref{Identification}, one could apply
\cite[Theorem 4.9]{MNOS04} (setting $\eta_0=\{\underline{-1}\}$)
to get
\begin{align}
    \lim_{\beta \to \infty}\frac{1}{\beta} \log \mathbb{E}_{\underline{-1}}[\tau_{\underline{+1}}] =\Gamma^{Hex}.
\end{align}

\section{Geometry of the model}\label{geom1}

Since the faces of a cluster live naturally on the triangular lattice it is beneficial to associate clusters to plane polyforms obtained by joining equilateral triangles along their edges (\emph{polyiamonds}). In this way it will be possible to 
characterize spin configurations that are relevant for the dynamics under consideration in terms of the area and the perimeter of the polyiamond associated to their clusters. 
Though we will consider polyiamonds to study the Ising model on the
hexagonal lattice, the properties that will derive may be of use to study other 
statistical mechanics lattice models for which
the notion of clusters may be linked to that of polyiamonds.

\begin{definition}\label{def:polyiamond}
    A \emph{polyiamond} $P \subset \mathbb{R}^2$ is a 
    finite maximally edge-connected union of faces of the lattice $\mathbb{T}^2$.
    Each face belonging to the polyiamond is called \emph{triangular unit} whereas
    the faces of $\mathbb{T}^2$ outside of $P$
    are called \emph{empty triangular units}.
\end{definition}

We remark that two faces are not connected if they
share a single point.

Note that with this construction there is a bijection between clusters 
of plus spins not wrapping around the torus and
polyiamonds.
Analogously, minus spins are associated to the empty triangular units of the lattice $\mathbb{T}^2$. 
Strictly speaking, this mapping is different from the one introduced before which is a bijection between the configuration space $\mathcal{X}$ and the set of sea of minuses, strips, clusters and holes. Both these bijections are relevant in the entire paper.

\begin{definition}
    An \emph{elementary rhombus} is a set of two triangular units that share an edge.
\end{definition}

In the remainder of the Section we
will give those definitions and
provide the key results concerning
polyiamonds that are used in
Sections~\ref{recurrenceproperty} and~\ref{bound} whereas
a more comprehensive and self-contained
discussion on polyiamonds is deferred to Section~\ref{geom2}. 
\begin{definition}
The \emph{area} $A$ of a polyiamond is the number of its triangular units. For any polyiamond $P$, we denote its area by $||P||$. Analogously for any cluster $C$, we denote by $||C||$ the number of plus spins in $C$.  
\end{definition}
\begin{definition}\label{def:perimeter_of_polyiamond}
The \emph{boundary} of a polyiamond $P$ is the collection of unit edges of the lattice $\mathbb{T}^2$ such that each edge separates a triangular unit belonging to $P$ from an empty triangular unit. The \emph{edge-perimeter} $p(P)$ of a polyamond $P$ is the cardinality of its boundary. 
\end{definition}
In other words the perimeter is given by the number of interfaces on the discrete triangular lattice $(\mathbb{T}^2)$ between the sites inside the polyiamond and those outside.
If not specified differently, we will refer to the edge-perimeter simply as perimeter. 
\begin{definition}\label{def:internal_perimeter}
The \emph{external boundary} of a polyiamond consists of the connected components of the boundary such that for each edge there exists a hexagonal-path in $\mathbb{H}^2$ which connects this edge with infinity without intersecting
the polyiamond. The \emph{internal boundary} of a polyiamond consists of the connected components of the boundary that are not external. The \emph{external perimeter}, respectively the \emph{internal perimeter}, of a polyiamond is the cardinality of the external, respectively internal, boundary.
\end{definition}
\begin{definition}\label{def:hole_in_polyiamond}
A \emph{hole} of a polyiamond $P$ is a finite maximally connected component of empty triangular units surrounded by the internal boundary of $P$.  
\end{definition}
We refer to holes consisting of a single empty triangle as \emph{elementary holes}.
\begin{definition}\label{internalangle}
Orient the external boundary counter-clockwise and internal boundary clockwise. For each pair of oriented edges, the angle defined rotating counter-clockwise the second edge on the first edge is called \emph{internal angle}. See Figure \ref{anglesfigure}.

\begin{figure}[htb!] 
\centering \includegraphics[scale=0.6]{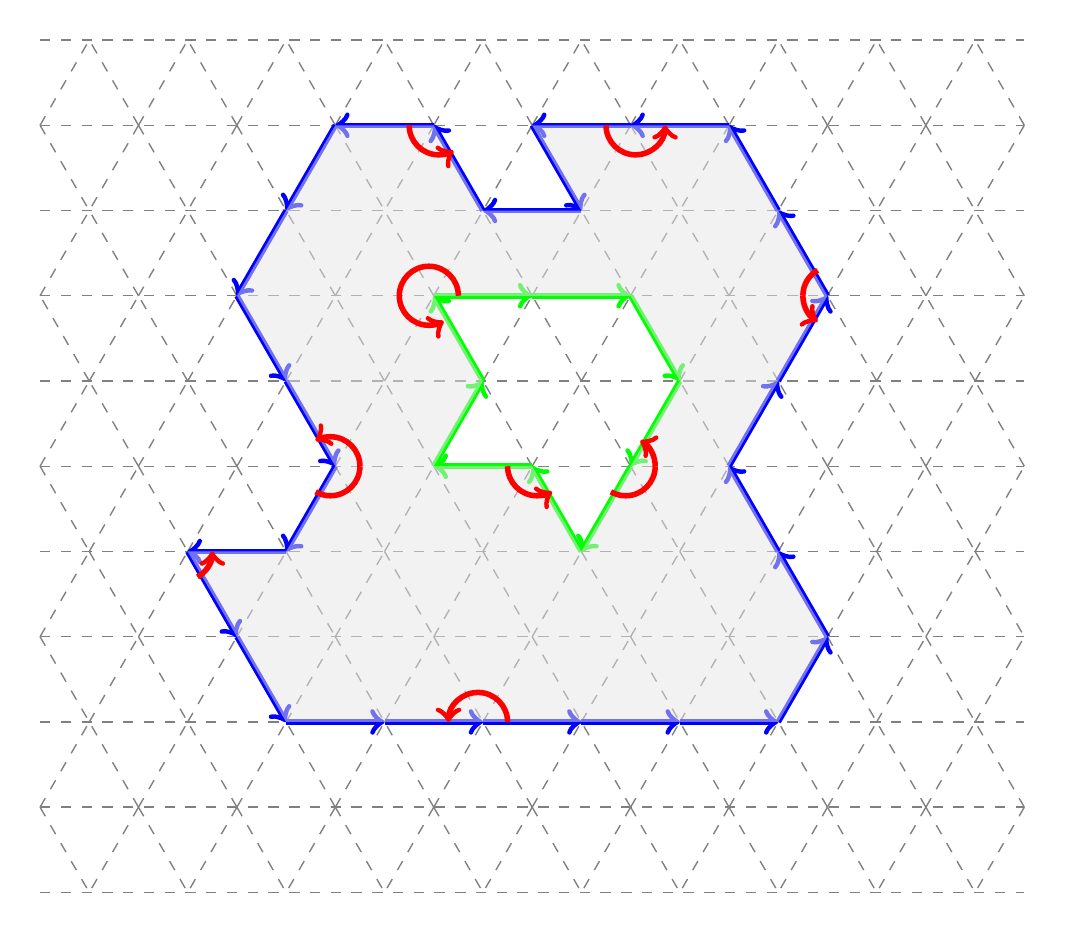} 
\caption{An example of polyiamond with in blue the external boundary oriented counter-clockwise, and in green the internal boundary oriented clockwise. In red, some internal angles of the polyiamond.} \label{anglesfigure}
\end{figure}
\end{definition}
\begin{definition}
A polyiamond is \emph{regular} if it has only internal angles of $\pi$ and $\frac{2}{3}\pi$ and it has no holes. 
\end{definition}
We note that a regular polyiamond has the shape of a hexagon.
\begin{definition}
A polyiamond is a \emph{regular hexagon} if it is a regular polyiamond with all equal sides. We denote by $E(r)$ the regular hexagon, where $r$ is its radius.
\end{definition}
\begin{definition}\label{def:bars}
A \emph{bar} $B(l)$ with larger base $l$ is a set of $||B(l)||=2l-1$ triangular units obtained as a difference between an equilateral triangle with side length $l$ and another equilateral triangle with side length $l-1$ 
(see Fig.~\ref{fig:barra}).
\end{definition}
\begin{figure}[htb!] 
    \centering \includegraphics[scale=0.5]{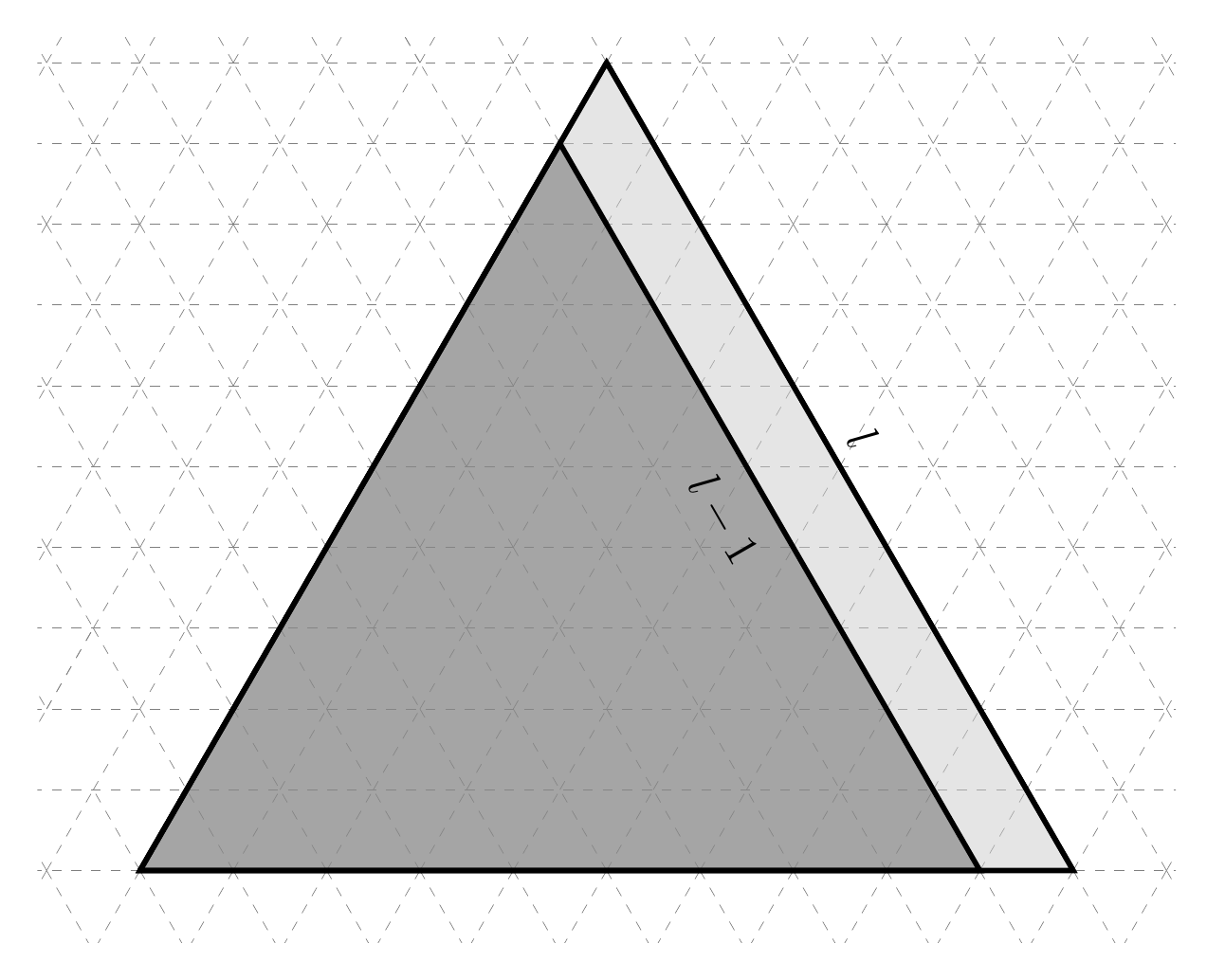} 
    \caption{The lightly shaded triangular units form a bar with larger base $l$ and
    it is obtained as difference between an equilateral triangle with
    side length $l$ and an equilateral triangle with side length $l-1$.}
    \label{fig:barra}
\end{figure}

\begin{definition}\label{def:quasi_reg_hex}
We denote by $E_{B_1}(r)$ the polyiamond obtained by attaching a bar $B(r)$ along its larger base to the a side of the regular hexagon,
see Figure \ref{bars}, so that 
$E_{B_1}(r)$ is contained in $E(r+1)$. 
Analogously, we denote by $E_{B_i}(r)$ for $i=2,\ldots,5$ the polyiamonds obtained by attaching a bar $B(r+1)$ along its larger base to a suitable side of $E_{B_{i-1}}(r)$ so that, again, $E_{B_i}(r)$ is contained in $E(r+1)$
Finally, we denote by $E_{B_6}(r)$ the polyiamond obtained by attaching a bar $B(r+2)$ along its larger base to $E_{B_5}(r)$ so to obtain $E(r+1)$.
We call $E_{B_i}(r)$ a \emph{quasi-regular hexagon}, where $r$ is the radius of the regular hexagon $E(r)$ and $i\in \{1,\ldots,6\}$ is the number of bars attached to it. 
\end{definition}
Note that $E_{B_i}(r)$ is defined up to 
translations and rotations of $z \frac{\pi}{3}$ for $z \in \mathbb{Z}$.
Moreover \mbox{$E_{B_0}(r) := E(r)$} and  $E(r+1) \equiv E_{B_6}(r)$. 
\begin{figure}[htb!] 
\centering \includegraphics[scale=0.5]{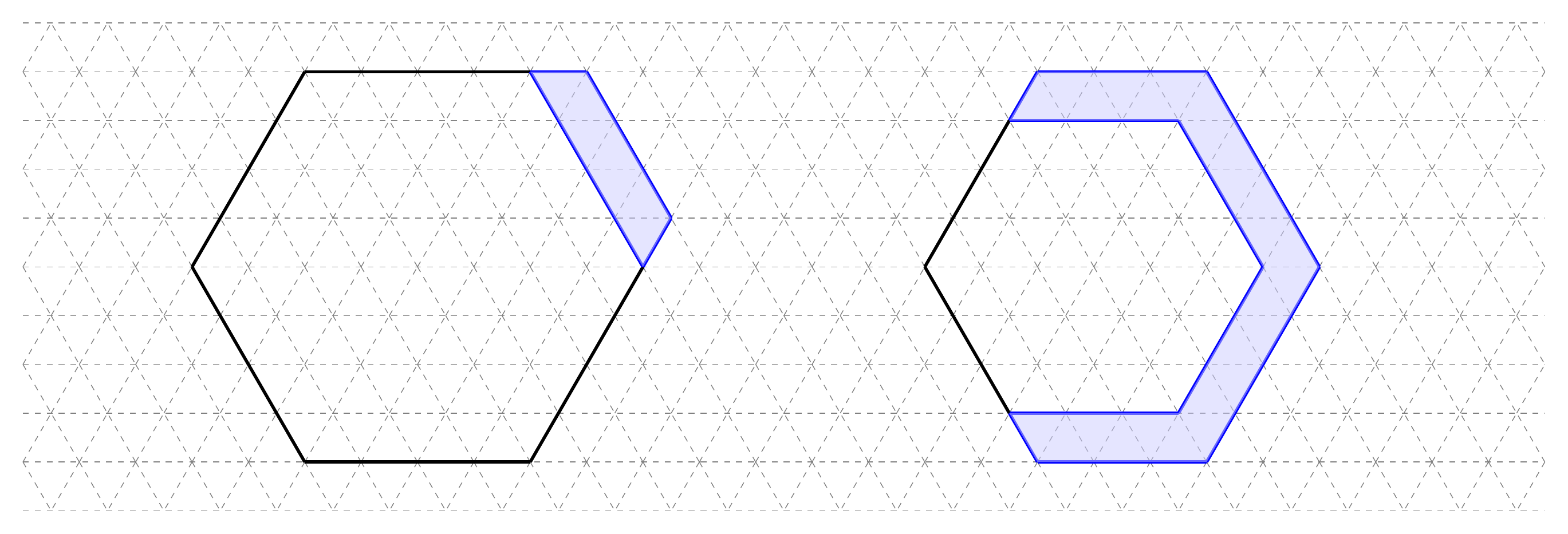} 
\caption{On the left a quasi-regular hexagon $E_{B_1}(4)$. We observe that the cardinality of $B_1$ of $E_{B_1}(3)$ is $||B_1||=2r-1=5$. On the right a quasi-regular hexagon $E_{B_4}(3)$. We observe that the cardinality of $B_1$ of $E_{B_1}(3)$ is $||B_1||=2r-1=5$, while the cardinality of $B_i$ of $E_{B_4}(3)$ is $||B_i||=2r+1=7$ with $i=2,\ldots,4$.} \label{bars}
\end{figure}
\begin{definition}
  An \emph{incomplete bar} of cardinality $k < 2l-1$ 
  is an edge-connected subset of a bar $B(l)$ attached to a  hexagon along its longest base, see Figure \ref{incompletebar}. 
\end{definition}
Observe that an incomplete bar has either the shape of a trapezoid or 
of a parallelogram with height of size $\frac{\sqrt{3}}{2}$. 
\begin{figure}[htb!] 
\centering \includegraphics[scale=0.5]{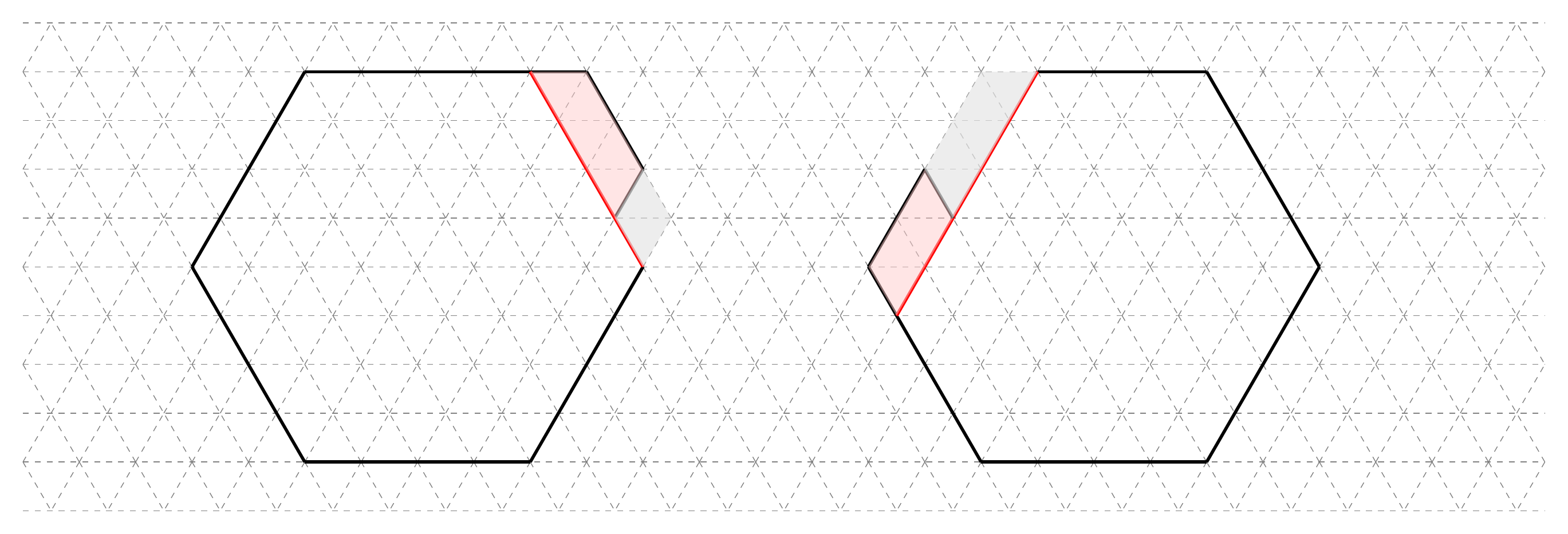} 
\caption{On the left an incomplete bar having the shape of a trapezoid and of cardinality $k=5$ attached to the regular hexagon $E(4)$. We observe that the cardinality of the bar containing the incomplete bar is $||B_1||=7>k$. On the right an incomplete bar with
the shape of a parallelogram of cardinality $k=4$, attached to the quasi-regular hexagon $E_{B_5}(3)$. We observe that the cardinality  of the bar containing the incomplete bar is $||B_6||=9>k$.
Both configurations are examples of \emph{standard polyiamonds}.} \label{incompletebar}
\end{figure}
\begin{definition}
For any $A \in \mathbb{N}$, the \emph{minimal quasi-regular hexagon} $R_A$ is the smallest quasi-regular hexagon containing at least $A$ triangular units. 
\end{definition}
Thus, $R_A$ has area $A+k$, where $k$ is the smallest number of triangular units added to $P$ to build a quasi-regular hexagon.
\begin{definition}
For any $A \in \mathbb{N}$, the \emph{maximal quasi-regular hexagon} $R'_A$ is
the largest quasi-regular hexagon containing at most $A$ triangular units. 
\end{definition}
Thus, $R'_A$ has area $A-q$, where $q$ is the smallest number of triangular units
that must be removed from $P$ to build a quasi-regular hexagon. 
\begin{definition}\label{canonical}
A \emph{canonical polyiamond} of area $A$, denoted by $\tilde{S}(A)$, is a quasi-regular hexagon $E_{B_i}(r)$, for $i \in \{0, \ldots, 5 \}$, with possibly an additional incomplete bar of cardinality $k$ and such that it is contained in $E_{B_{i+1}}(r)$ 
(see Fig.~\ref{incompletebar}). 
\end{definition}
\begin{definition}\label{standard}
Orient the external boundary clockwise and attach an incomplete bar to $E_{B_i}(r)$, for $i \in \{0, \ldots, 5 \}$, following this orientation. We call this canonical polyiamond \emph{standard polyiamond} and denote it by $S(A)$.
\end{definition}
\begin{definition}\label{def:defective_polyiamond}
     A polyiamond consisting of a quasi-regular hexagon with two triangular units attached to one of its longest sides at triangular lattice distance 2 one from the other is called \emph{defective} and is denoted by $\tilde D(A)$, where $A$ is the area (see the second shape in Figure~\ref{addtriangle1}).
\end{definition}
Note that a standard polyiamond $S(A)$ is determined solely by its area $A$.
We characterize $S(A)$ in terms of its radius $r$, the number $i$ of bars
attached to the regular hexagon $E(r)$ to obtain $E_{B_i}(r)$ and the
cardinality $k$ of the possibly incomplete bar. The values of
$r$, $i$ and $k$
for any value of $A$,
together with the minimal and the maximal quasi-regular
hexagons $R_A$ and $R'_{A}$,
can be obtained with the following:

\begin{alg}{Construction of standard polyiamond with area $A$.}\label{algorithm} Given $A$ as input, the outputs $r$, $k$, $i$, $R_A$, $R'_A$ are obtained by:
\begin{itemize}
\item[1)] Set $r= [\sqrt{\frac{A}{6}}]$.
\item[2)] Let $l$ be the difference between $A$ and $6r^2$, i.e., $l=A-6r^2$:
\begin{itemize}
\item[a)] If $l=0$, then $R_A=R'_A=E(r)$; $i=0$, $k=0$.
\item[b)] If $l-(2r-1) < 0$, then 
          $R_A=E_{B_1}(r)$ and $R'_A=E(r)$;
          $i=0$, $k=A-\area{R'_A}$.
\item[c)] If $l-(2r-1)=0$, then  
          $R_A=R'_A=E_{B_1}(r)$; 
          $i=1$, $k=0$.
\item[d)] If $l-((2r-1)+(2r+1)) < 0$, then 
          $R_A=E_{B_2}(r)$ and $R'_A=E_{B_1}(r)$; 
          $i=1$, $k=A-\area{R'_A}$.
\item[e)] If $l-((2r-1)+(2r+1))=0$, then  
          $R_A=R'_A=E_{B_2}(r)$; 
          $i=2$, $k=0$. 
\item[f)] If $l-((2r-1)+2(2r+1)) < 0$, then 
          $R_A=E_{B_3}(r)$ and $R'_A=E_{B_2}(r)$; 
          $i=2$, $k=A-\area{R'_A}$.
\item[g)] If $l-((2r-1)+2(2r+1))=0$, then  
          $R_A=R_A'=E_{B_3}(r)$; 
          $i=3$, $k=0$.
\item[h)] If $l-((2r-1)+3(2r+1)) < 0$, then 
          $R_A=E_{B_4}(r)$ and $R'_A=E_{B_3}(r)$; 
          $i=3$, $k=A-\area{R'_A}$.
\item[i)] If $l-((2r-1)+3(2r+1))=0$, then  
          $R_A=R'_A=E_{B_4}(r)$; 
          $i=4$, $k=0$.
\item[j)] If $l-((2r-1)+4(2r+1)) < 0$, then 
          $R_A=E_{B_5}(r)$ and $R'_A=E_{B_4}(r)$; 
          $i=4$, $k=A-\area{R'_A}$.
\item[k)] If $l-((2r-1)+4(2r+1))=0$, then  
          $R_A=R'_A=E_{B_5}(r)$; 
          $i=5$, $k=0$.
\item[l)] If $l-((2r-1)+4(2r+1)+(2r+3))<0$, then 
          $R_A=E_{B_6}(r)$ and $R'_A=E_{B_5}(r)$; 
          $i=5$, $k=A-\area{R'_A}$.  
\end{itemize}
\end{itemize}
\end{alg}
Once the standard polyiamonds of area $A$ have been described in
the previous terms, it is straightforward to write their perimeter
as follows:
\begin{remark}\label{proprietaS}
    The perimeter of $S(A)$ is 
    $p(A) = 6r + i + \mathbb{1}_{\{k>0\}} +\mathbb{1}_{\{k>0 \text{ even}\}}$
    with $r$, $i$ and $k$ given by the previous algorithm.
\end{remark}
\begin{notation}
We denote by $\mathcal{E}(r)$ the configuration $\sigma \in \mathcal{X}$ such that $\sigma$ has a unique cluster (of pluses) with shape $E(r)$.
We denote by $\mathcal{E}_{B_i}(r)$ the configuration $\sigma \in \mathcal{X}$ such that $\sigma$ has a unique cluster (of pluses) with shape $E_{B_i}(r)$.
\end{notation}
\begin{notation}\label{notation:standard_canonical_defective_configurations}
We denote by $\mathcal{\tilde{S}}(A)$ (respectively $\mathcal{S}(A)$, $\mathcal{\tilde{D}}(A)$) the configuration $\sigma \in \mathcal{X}$ such that $\sigma$ has a unique cluster (of pluses) with shape $\tilde{S}(A)$ (respectively $S(A)$, $\tilde{D}(A)$).
\end{notation}

Each of these geometrical definitions and properties can be extended from polyiamonds to clusters. So, for example, when we call a cluster \emph{standard cluster}, our meaning is that the cluster has the shape and the properties of a standard polyiamond. 

The next Theorem states that the set of polyiamonds of minimal perimeter and area $A$ contains the set of standard polyiamonds with area $A$. In other words, standard polyiamonds minimize the perimeter for any given number of triangular units. 
\begin{theorem}\label{thm:optimality_of_standard_polyiamonds}
For any $A \in \mathbb{N}$ the perimeter of a polyiamond $P$ of area $A$ is at least $p(S(A))$ where $p(S(A))$ is the perimeter of a standard polyiamond $S(A)$.
\end{theorem}
Considering the construction of minimal $R_A$ and maximal $R'_A$ quasi-regular hexagon given in the algorithm \ref{algorithm}, we get immediately the following:
\begin{corollary}\label{lemma617}
For any $A$ positive integer there exist four positive integers $r$, $k_1$, $k_2$ and $k_3$ such that one of the following conditions applies: 
\begin{enumerate}
    \item $A=6r^2+k_1$ with $0 \le k_1 < 2r-1$. Then the set of polyiamonds of area $A$ and minimal perimeter contains the polyiamond $S(6r^2+k_1)$. 
    \item $A=6r^2+2r-1+k_2$ with $0 \le k_2 < 2r+1$. Then the set of polyiamonds of area $A$ and minimal perimeter contains the polyiamond $S(6r^2+2r-1+k_2)$. 
    \item $A=6r^2+4r+k_2$ with $0 \le k_2 < 2r+1$. Then the set of polyiamonds of area $A$ and minimal perimeter contains the polyiamond $S(6r^2+4r+k_2)$. 
    \item $A=6r^2+6r+1+k_2$ with $0 \le k_2 < 2r+1$. Then the set of polyiamonds of area $A$ and minimal perimeter contains the polyiamond $S(6r^2+6r+1+k_2)$. 
    \item $A=6r^2+8r+2+k_2$ with $0 \le k_2 < 2r+1$. Then the set of polyiamonds of area $A$ and minimal perimeter contains the polyiamond $S(6r^2+8r+2+k_2)$. 
    \item $A=6r^2+10r+3+k_3$ with $0 \le k_3 < 2r+3$. Then the set of polyiamonds of area $A$ and minimal perimeter contains the polyiamond $S(6r^2+10r+3+k_3)$. 
\end{enumerate}
\end{corollary}
Moreover, the next lemma, proved in Section \ref{geom2}, states that for some specific values of the area $A$ there exists a unique class of polyiamonds that minimize the perimeter, namely the class of quasi-regular hexagons with area $A$. 
\begin{lemma}\label{lemma:quasi_regular_hex_unique_permiter_minimizers}
    If the area $A$ of a polyiamond is $6r^2 + 2mr + (m-2)\mathbb{1}_{\{m > 0\}}$ for $m \in \{0, 1, \ldots, 5\}$ the set of polyiamonds of area $A$ and minimal perimeter contains only the quasi-regular hexagons.
\end{lemma}
\begin{proof}[Proof of Lemma~\ref{lemma:quasi_regular_hex_unique_permiter_minimizers}]
    Note that areas of the form
    $6r^2 + 2mr + (m-2)\mathbb{1}_{\{m > 0\}}$ for $m \in \{0, 1, \ldots, 5\}$
    are compatible with the area of quasi-regular hexagons.
    Let $A$ be the area of a quasi-regular hexagon and let
    $\bar{p}$ be its edge-perimeter. The Lemma is clearly implied by Lemma~\ref{lemma:quasi_regular_hex_unique_edge_perimeter_minimizers}.
\end{proof}
\section{Recurrence proposition}\label{recurrenceproperty}
The goal of this Section is to prove Proposition~\ref{teoRP}. To achieve it we give some definitions.
Recalling \eqref{raggiocritico} and \eqref{areacritica}, we divide the set of standard clusters 
into three classes: \emph{supercritical}, \emph{critical} and \emph{subcritical} clusters. 

\begin{definition} We call a standard cluster \emph{supercritical} (respectively \emph{subcritical}) if it has the shape of $S(A)$ with area $A > A^*_i$ (respectively $A < A^*_i$) for $i \in \{1,2\}$. A \emph{critical} standard cluster is a standard cluster which has the shape of $S(A)$ with area $A=A^*_i$. 
\end{definition}
We will see that supercritical standard clusters have the tendency to grow with high probability, whereas subcritical standard clusters have the tendency to shrink with high probability. 

The following Definition distinguishes if two or more regular clusters of a configuration do interact or they do not.
\begin{definition}
Two regular clusters $Q$ and $\tilde Q$ are \emph{non-interacting} if \mbox{$d(Q,\tilde Q)> 2$}, where $d$ is the lattice distance. Otherwise, the two regular clusters are called \emph{interacting}.
\end{definition}
We recall Definition \ref{internalangle} and we extend it to clusters. 
\begin{definition}
We call a \emph{corner of a standard cluster} $C$ the pair of triangular faces of $C$ contained in the internal angle of $\frac{2}{3}\pi$. 
\end{definition}
\subsection{Recurrence property to the set $\mathcal{X}_{J-h}$}
In this Section we identify the configurations in $\mathcal{X}_{J-h}$, 
that is those configurations having a ``small'' stability level, and 
partition this set into five subsets. This partition will turn out
to be convenient in the next Section.
\begin{lemma}\label{noX}
A configuration $\sigma$ that satisfies at least one 
of the following conditions is not in $\mathcal{X}_0$:
\begin{enumerate}
    \item $\sigma$ has either a cluster or a plus strip $C$ with an internal angle of $\frac{5}{3}\pi$;
    \item $\sigma$ has either a cluster or a plus strip $C$ with an internal angle of $\frac{1}{3}\pi$.
\end{enumerate}
Moreover, a configuration $\sigma$ is not in $\mathcal{X}_{J-h}$ if the cluster or the plus strip $C$ has an internal angle of $\frac{4}{3}\pi$.
\end{lemma}

\begin{proof}[Proof of Lemma~\ref{noX}] 
Suppose that $C$ has an internal angle $\alpha=\frac{5}{3}\pi$.
Let $j$ be the site at distance one to a site in $C$ such that $\sigma(j)=-1$ and that belongs to the closed triangular face intersecting the boundary of $C$ in two or more edges, see Figure \ref{anglesconf}, case~(a). Since by Table~\ref{Tab}
\begin{equation}
    H(\sigma^{(j)})-H(\sigma)=-(J+h)<0,
\end{equation}
$\sigma^{(j)}$ belongs to $\mathcal{I}_{\sigma}$. 
Thus the stability level is equal to $V_{\sigma}=0$ and it follows that $\sigma \not \in \mathcal{X}_0$.

\begin{figure}[htb!] 
\centering \includegraphics[scale=0.7]{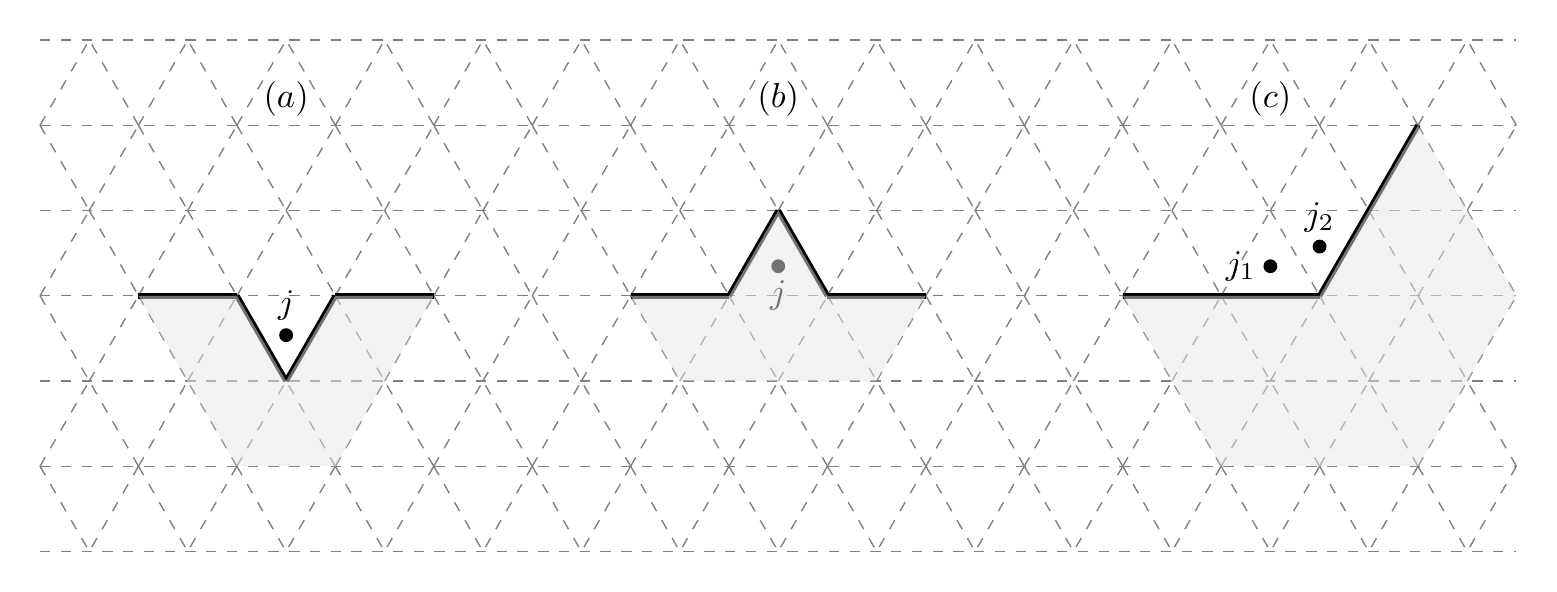} 
\caption{On the left hand side (center) we depict the site $j$ as in case 1. (case 2.). On the right hand side we depict the two sites $j_1,j_2$ when $\sigma$ has an internal angle of $\frac{4}{3}\pi$.} \label{anglesconf}
\end{figure}

Suppose now that $C$ has an internal angle $\alpha=\frac{1}{3}\pi$.
Let $j$ be a site such that $\sigma(j)=+1$ and that belongs to the closed triangular face of $C$ intersecting its boundary in two edges, see Figure~\ref{anglesconf}, case~(b). Since by Table~\ref{Tab}
\begin{equation}
     H(\sigma^{(j)})-H(\sigma)=-(J-h)<0,
\end{equation}
$\sigma^{(j)}$ belongs to $\mathcal{I}_{\sigma}$. Thus the stability level is equal to $V_{\sigma}=0$ and it follows that $\sigma \not \in \mathcal{X}_0$.

Next we prove that if a configuration $\sigma$ has a cluster or a plus strip $C$ with an internal angle of $\frac{4}{3}\pi$, then $\sigma \not \in \mathcal{X}_{J-h}$. Suppose that $C$ has an internal angle $\alpha=\frac{4}{3}\pi$. We construct a path that starts from $\sigma\equiv \omega_0$ and we define $\omega_1$ as follows. Let $j_1$, $j_2$ be two sites such that $\sigma(j_1)=\sigma(j_2)=-1$, $d(j_1,j_2)=1$ and let each of them belong to one closed triangular face intersecting the boundary of $C$ in one edge, see 
Figure \ref{anglesconf}, case~(c). Pick one of the two sites, for example $j_1$ and flips its spin, i.e., $\omega_1:=\omega_0^{(j_1)}$. Then flip the spin of $j_2$, i.e. $\omega_2:=\omega_1^{(j_2)}$.
The stability level is bounded above by
\begin{equation}
    V_{\sigma} \le H(\omega_1)-H(\omega_0)=J-h,
\end{equation}
indeed the configuration $\omega_2$ is in $\mathcal{I}_{\sigma}$ since
\begin{equation}
    H(\omega_2)-H(\omega_0)= (H(\omega_2)-H(\omega_1))+(H(\omega_1)-H(\omega_0))=-(J+h)+(J-h)=-2h<0.
\end{equation}
Thus $\sigma \not \in \mathcal{X}_{J-h}$.
\end{proof}
\begin{corollary}\label{corXJh}
A configuration $\sigma \in \mathcal{X}_{J-h}$ if it only has non-interacting clusters or plus strips with internal angles of $\frac{2}{3}\pi$ or $\pi$ only.
\end{corollary}
\begin{proof}[Proof of Corollary~\ref{corXJh}] Suppose $\sigma \in \mathcal{X}_{J-h}$. By Lemma \ref{noX}, it follows that 
the clusters (or the plus strips) of $\sigma$ do not have angles of $\frac{5}{3}\pi$, $\frac{1}{3}\pi$ and $\frac{4}{3}\pi$. So, either $\sigma\equiv \underline{+1}$ or the clusters (or the plus strips) of $\sigma$ have only internal angles of $\pi$ and $\frac{2}{3}\pi$.

We observe that if two clusters $C_1$ and $C_2$ are interacting, there exists a triangular face with minus spin that shares a side with the external boundary of $C_1$ and a side with the external boundary of $C_2$. This case can be treated as the case of cluster with an internal angle of $\frac{5}{3}\pi$, therefore by Lemma~\ref{noX} these configurations do not belong to $\mathcal{X}_{J-h}$.
    \end{proof}
\paragraph{Partition of $\mathcal{X}_{J-h}$.} We partition the set $\mathcal{X}_{J-h}\setminus\{\underline{-1},\underline{+1}\}$ into four subsets $Z, R, U, Y$. 
$Z$ is the set of configurations consisting of a single quasi-regular hexagonal cluster, see Figure~\ref{configurationZR}. More precisely, $Z=Z_1 \cup Z_2$, where:
\begin{itemize}[label=\raisebox{0.25ex}{\tiny \textbullet}]
    \item $Z_1$ is the collection of configurations such that there exists only one cluster with shape $E_{B_m}(r) \subset \Lambda$ with $r \leq r^*$ and $m \in \{0,1,2,3,4,5\}$;
    \item $Z_2$ is the collection of configurations such that there exists only one cluster with shape $E_{B_m}(r) \subset \Lambda$ with $r \geq r^*+1$ and $m \in \{0,1,2,3,4,5\}$.
    \end{itemize}
   We define the set $R$ to be the set of configurations consisting of a single regular cluster see Figure~\ref{configurationZR}. Formally, $R=R_1 \cup R_2$, where:
\begin{itemize}[label=\raisebox{0.25ex}{\tiny \textbullet}]
    \item $R_1$ is the collection of configurations such that there exists only one cluster with hexagonal shape $E \subset \Lambda$ such that it contains the greatest quasi-regular hexagon with radius $r \leq r^*$; 
    \item $R_2$ is the collection of configurations such that there exists only one cluster with hexagonal shape $E \subset \Lambda$ such that it contains the greatest quasi-regular hexagon with radius $r \geq r^*+1$.
    \end{itemize}
     \begin{figure}[htb!]
    \centering
    \includegraphics[scale=0.8]{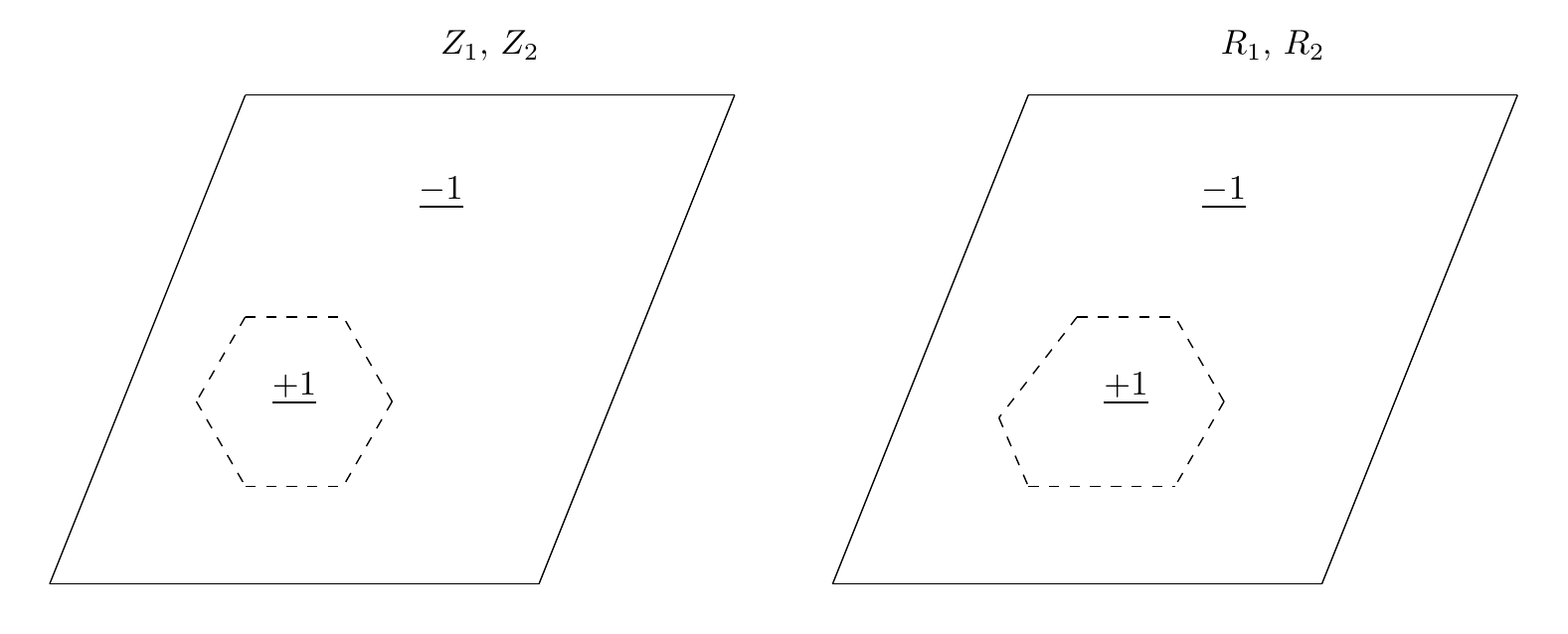}
    \caption{On the left an example of configuration in $Z$, on the right an example of configuration in $R$.}
    \label{configurationZR}
    \end{figure}
    The set $U$ contains all configurations with more than one 
    hexagonal cluster of the types in $Z_1$, $Z_2$, $R_1$, $R_2$, see Figure~\ref{configurationUY}. More precisely, we have $U=U_1 \cup U_2$, where: 
    \begin{itemize}[label=\raisebox{0.25ex}{\tiny \textbullet}]
    \item $U_1$ is the collection of configurations such that there exists a family of non-interacting clusters with hexagonal shape such that it contains the greatest quasi-regular hexagon with radius $r \leq r^*$; 
    \item $U_2$ is the collection of configurations such that there exists a family of clusters with at least one having hexagonal shape containing the greatest quasi-regular hexagon with radius $r \geq r^*+1$.
    \end{itemize}
In other words $U_1$ contains a collection of clusters of the same type of those in $Z_1$ or $R_1$, and $U_2$ contains a collection of clusters where at least one is of the same type of those in $Z_2$ or $R_2$.

The set $Y$ contains all possible 
(plus or minus)
strips with only 
$\pi$ internal angles on their boundary and, 
possibly, some hexagonal clusters, see Figure~\ref{configurationUY}. 
\begin{figure}[htb!]
    \centering
    \includegraphics[scale=0.8]{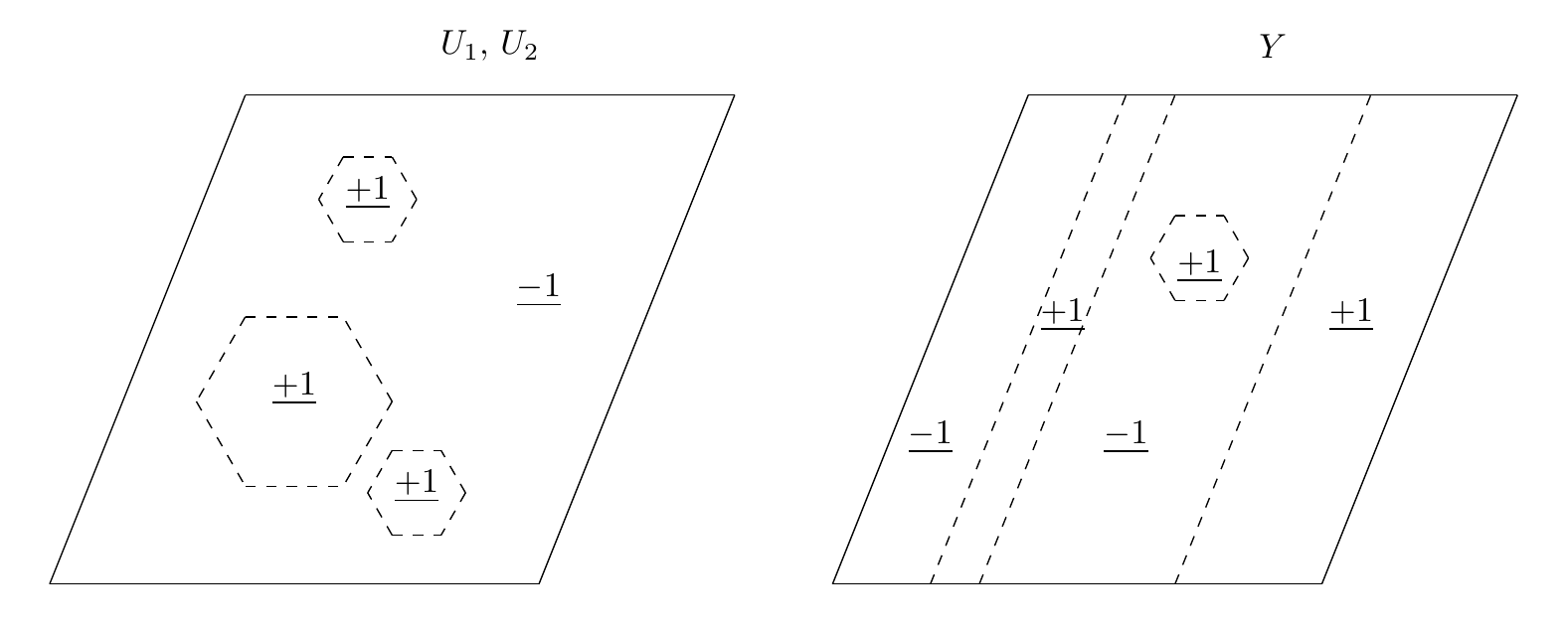}
    \caption{On the left an example of configuration in $U$, on the right an example of configuration in $Y$.}
    \label{configurationUY}
\end{figure}

\begin{corollary}\label{XJ-h}
We have 
$\mathcal{X}_{J-h}=
Z\cup 
R \cup 
U \cup 
Y \cup 
\{ \underline{+1}, \underline{-1}\}$ and for any $\epsilon>0$ and sufficiently large $\beta$,
\begin{align}\label{recurrenceJ-h}
    \sup_{\sigma \in \mathcal{X}} \mathbb{P}_{\sigma}(\tau_{\mathcal{X}_{J-h}}> e^{\beta(J-h+\epsilon)})=SES.
\end{align}
\end{corollary}
Equation \eqref{recurrenceJ-h} implies that the system visits with high probability a state with a stability level greater than $J-h$ in a time shorter than $e^{\beta (J-h +\epsilon)}$. 
In other words, the states in $\mathcal{X}_{J-h}$ are the relevant ones for a dynamics
speeded up by a factor $e^{\beta (J-h)}$.
\begin{proof}[Proof of Corollary~\ref{XJ-h}]
    By the partition described above, Definition \eqref{Xv} and Corollary \ref{corXJh} it follows that $\mathcal{X}_{J-h}=Z\cup R \cup U \cup Y \cup \{ \underline{+1}, \underline{-1}\}$. By \cite[Theorem 3.1]{MNOS04} for $V=J-h$ and by the partition of $\mathcal{X}_{J-h}$, we get \eqref{recurrenceJ-h}.
\end{proof}

\subsection{Proof of Proposition~\ref{teoRP}}
\begin{lemma}[Estimate of stability levels]\label{EST} For every $\sigma\in\mathcal{X}_{J-h}\setminus\{\underline{-1},\underline{+1}\},$ there exists \mbox{$V^*=2J$} such that $V_\sigma\leq V^*$. 
\end{lemma}
 \begin{proof}[Proof of Lemma~] 
We begin by considering the set $Z$.
    
    \paragraph{Case $Z_1$.} For any configuration $\sigma\in Z_1$ we construct a path $\overline{\omega}\in \Theta(\sigma,I_{\sigma} \cap (Z_1 \cup \{\underline{-1}\}))$ that
    dismantles the bar on one of the shortest sides of the quasi-regular hexagon starting from one of its corners. Starting from $\sigma\equiv \omega_0\in Z_1$, we will define $\omega_1$ as follows. Consider a corner in one of the shortest sides of the cluster in $\mathcal{E}_{B_m}(r)$ and let $j$ be a site belonging to this corner. Flip the spin in $j$, i.e., $\omega_1:=\omega_0^{(j)}$. Define $\omega_2:=\omega_1^{(j_1)}$, where $j_1$ is the other site belonging to the same corner. %In this way $\sigma(j_1)$ switches its sign.  
    From the values in Table~\ref{Tab}, $H(\omega_{1})-H(\omega_{0})=J+h$ and $H(\omega_{2})-H(\omega_{1})=-(J-h)$. By iterating this procedure along the considered side, a bar of the cluster is erased and we obtain the configuration $\eta \equiv \omega_k$ such that $\eta=\mathcal{E}_{B_{m-1}}(r)$ for $m \neq 0$, otherwise $\eta=\mathcal{E}_{B_5}(r-1)$ for $m=0$. Note that the length of the path is equal to the cardinality $k$ of the bar. 
    
    In order to determine where the maximum is attained, we rewrite for $n=2, \ldots, k$
    \begin{align}
       H(\omega_n)-H(\omega_0)=
       \begin{cases}
       \sum_{t=2, \, t \, even}^{n}(H(\omega_{t})-H(\omega_{t-2})) & \text{if even n}, \\
       \sum_{t=2, \, t \, even}^{n-1}(H(\omega_{t})-H(\omega_{t-2}))+H(\omega_{n})-H(\omega_{n-1}) & \text{if odd n}. \\
       \end{cases}
    \end{align}
   From the values in Table~\ref{Tab}, we obtain for every $s=2,\cdots,k-1$
    \begin{equation}\label{reversibility1}
            H(\omega_{s})-H(\omega_{s-2})=[H(\omega_{s})-H(\omega_{s-1})]+[H(\omega_{s-1})-H(\omega_{s-2})]=2h.
    \end{equation}
    The previous can also be obtained by using \eqref{eq:peierls_hamiltonian} and observing that flipping two spins, as described above, corresponds to decreasing $N^{+}(\sigma)$ and leads to a cluster with a Peierls contour of the same length. So, we have
    \begin{align}\label{differenza_hamiltoniana_n_0}
        H(\omega_n)-H(\omega_0)=
        \begin{cases} 
        2h\frac{n-1}{2}+(J+h)=J+nh & \text{if odd } n=1,\cdots,k-2,\\
        2h \frac{n}{2}=nh & \text{if even } n=2,\cdots,k-1, \\
        -J+nh & \text{if } n=k.\\
        \end{cases}
    \end{align}
    Since in each case the result is an increasing function of $n$, comparing the three maxima, we see that the absolute maximum is attained in $\omega_{k-2}$. By definitions \ref{def:bars} and \ref{def:quasi_reg_hex}, we have
    \begin{itemize}[label=\raisebox{0.25ex}{\tiny \textbullet}]
        \item $k=2r-1$, if the initial configuration is $\mathcal{E}_{B_1}(r)$;
        \item $k=2r+1$, if the initial configuration is $\mathcal{E}_{B_m}(r)$ for $m=2,3,4,5$;
        \item $k=2r+3$, if the initial configuration is $\mathcal{E}(r+1)$.
    \end{itemize}
    So, we have
    \begin{align}\label{phiomegabar}
        \Phi(\overline{\omega})-H(\omega_0) = H(\omega_{k-2})-H(\omega_0)=J+(k-2)h.
    \end{align}
Thus $\Phi(\overline{\omega})$ depends only on the cardinality $k$, that is an increasing function of the radius $r$ of the quasi-regular hexagon. 
The cardinality of the longest bar among those of the quasi-regular hexagon in a configuration in $Z_1$ is $2r^*+1$ (obtained removing ${B_5}$ from $E_{B_5}(r^*)$). Note that the maximum is not obtained for $k=2r^*+3$, since $\mathcal{E}(r^*+1) \not \in Z_1$. Let us check that $\omega_k \in I_{\sigma} \cap (Z_1 \cup \{\underline{-1}\})$. 
Since $k \leq 2r^*+1$ with $r^*=\lfloor J/2h-1/2 \rfloor$ and by \eqref{differenza_hamiltoniana_n_0}, we get
    \begin{equation}\label{sottoI}
        H(\omega_0)-H(\omega_k)=J-kh \geq J-(2r^*+1)h>0.
    \end{equation}
    Finally, by equations \eqref{sottoI} and \eqref{phiomegabar}, we have
    \begin{equation}
        V_{\sigma} \leq \Phi(\overline{\omega}) -H(\sigma)=J+(k-2)h.
    \end{equation}
Thus, we find $V^*_{Z_1}=\max_{\sigma \in Z_1} V_{\sigma}$ by choosing $k-2=(2r^*+1)-2$ and recalling $r^*=\lfloor J/2h-1/2 \rfloor$, we have 
    \begin{equation}
        V^*_{Z_1} \leq 2J-2h.
    \end{equation}

    \paragraph{Case $Z_2$.} For any configuration $\sigma\in Z_2$ we construct a path $\overline{\omega}\in \Theta(\sigma,I_{\sigma} \cap (Z_2 \cup \{\underline{+1}\}))$. Starting from $\sigma\equiv \omega_0\in Z_2$, let us define $\omega_1$. Consider a corner in one of the longest sides of the cluster in $\mathcal{E}_{B_m}(r)$ and let $j$ be a site belonging to this corner. Let $j_1$ be the site at distance one from $j$ such that $\sigma(j_1)=-1$. We define $\omega_1:=\omega_0^{(j_1)}$, i.e., $\sigma(j_1)$ switches the sign. We consider $j_2$ the site at distance one from $j_1$ such that $\sigma(j_2)=-1$ and $d(j_2,j')=2$ where $j' \neq j$ is another site of the initial cluster. We define $\omega_2:=\omega_1^{(j_2)}$, $\omega_3:=\omega_2^{(j_3)}$, where $j_3$ is the site at distance one from $j_2$ such that $\sigma(j_3)=-1$ and $d(j_3,j')=1$ where $j' \neq j$ is another site of the initial cluster. By iterating this procedure along the considered side, a bar is added to the initial cluster. We obtain the configuration $\eta \equiv \omega_k$ such that $\eta= \mathcal{E}_{B_{m+1}}(r)$ for $m\neq 5$, otherwise $\eta=\mathcal{E}(r+1)$ for $m=5$. Note that the length of the path is equal to the cardinality $k$ of the bar.

    In order to determine where the maximum is attained, we rewrite for $n=2, \ldots, k$
    \begin{align}
       H(\omega_n)-H(\omega_0)=
       \begin{cases}
        \sum_{t=2, \, t \, even}^{n}(H(\omega_{t})-H(\omega_{t-2})) & \text{if even n}, \\
       \sum_{t=2, \, t \, even}^{n-1}(H(\omega_{t})-H(\omega_{t-2}))+H(\omega_{n})-H(\omega_{n-1}) & \text{if odd n}.
       \end{cases}
    \end{align}
 From the values in Table~\ref{Tab} we obtain $H(\omega_{1})-H(\omega_{0})=J-h$, $H(\omega_{2})-H(\omega_{0})=2J-2h$, and for every $s=3,\cdots,k$
        \begin{equation}\label{reversibility2}
            H(\omega_{s})-H(\omega_{s-2})=[H(\omega_{s})-H(\omega_{s-1})]+[H(\omega_{s-1})-H(\omega_{s-2})]=-2h.
        \end{equation}
        The previous can also be obtained by using \eqref{eq:peierls_hamiltonian} and observing that flipping two spins, as described above, corresponds to increase $N^{+}(\sigma)$ and leads to a cluster with a Peierls contour of the same length. So, we have
        \begin{align}\label{valoridiff}
        H(\omega_n)-H(\omega_0)=
        \begin{cases}
        J-h, & \text{if } n=1,\\
        2J-2h, & \text{if } n=2,\\
        -2h\frac{n-2}{2}+(2J-2h)=2J-nh, & \text{if even } n=4,\cdots,k-1,\\
        -2h\frac{n-1}{2}+(J-h)=J-nh, & \text{if odd } n=3,\cdots,k.\\
        \end{cases}
        \end{align}
        As a function of $n$, the absolute maximum is attained in $\omega_{2}$. So, we have
        \begin{align}\label{phiomegabarra}
        \Phi(\overline{\omega})-H(\omega_0)=H(\omega_2)-H(\omega_0)=2J-2h.
        \end{align}
    Finally, let us check that $\omega_k \in I_{\sigma} \cap (Z_2 \cup \{\underline{+1}\})$. If $\sigma \in Z_2 \setminus \mathcal{E}(r^*+1)$, then the cardinality of the smallest bar among those of the quasi-regular hexagon in a configuration in $Z_2$ is \mbox{$k_{\text{min}}=2(r^*+1)+1$}. Since $r^*=\lfloor J/2h-1/2 \rfloor$ and by \eqref{valoridiff}, we have
    \begin{align}
        H(\omega_0)-H(\omega_k) & =kh-J\geq k_{\text{min}}h-J>0,
    \end{align}
    thus
        \begin{align}\label{Vsup}
        V_{\sigma} \leq \Phi(\overline{\omega})-H(\sigma)=2J-2h.
        \end{align}
     Now we consider $\mathcal{E}(r^*+1)$, we note that $H(\mathcal{E}(r^*+1))<H(\mathcal{E}_{B_1}(r^*+1))$. 
     Thus we need to consider a new path $\overline{\omega}$ that consists of the previously defined $(\mathcal{E}(r^*+1),\ldots,\mathcal{E}_{B_1}(r^*+1))$ connected with an additional part depending on the value of $\delta$, where $\delta \in (0,1)$ is such that $r^*= J/2h-1/2 -\delta$. If $0<\delta<\frac{1}{2}$ then we add the bar $B_2$ as we have done above for $B_1$ obtaining $\overline{\omega}=(\mathcal{E}(r^*+1),\ldots,\mathcal{E}_{B_1}(r^*+1), \ldots,\mathcal{E}_{B_2}(r^*+1) )$. If $\frac{1}{2}<\delta<1$ in the same manner we add the bars $B_2,B_3,B_4,B_5,B_6$ obtaining  $\overline{\omega}=(\mathcal{E}(r^*+1),\ldots,\mathcal{E}_{B_1}(r^*+1), \ldots,\mathcal{E}_{B_6}(r^*+1) \equiv \mathcal{E}(r^*+2) )$. In both cases the last configurations of the new paths belong to $\mathcal{I}_{\mathcal{E}(r^*+1)}$, indeed 
     \begin{align}
         H(\mathcal{E}(r^*+1))>H(\mathcal{E}_{B_2}(r^*+1)), & \qquad \text{if } \delta \in (0,\frac{1}{2}), \notag \\
         H(\mathcal{E}(r^*+1))>H(\mathcal{E}(r^*+2)), & \qquad \text{if } \delta \in (\frac{1}{2},1). \notag
     \end{align}
     Thus, using equations \eqref{valoridiff}, \eqref{phiomegabarra} and \eqref{Vsup}, we obtain
         \begin{align}
             V_{\sigma} \leq 2J-2h+ H(\mathcal{E}_{B_1}(r^*+1))-H(\mathcal{E}(r^*+1))=2J-2h+2h\delta<2J-h, \; 
             \text{for } \delta \in \Big(0,\frac{1}{2}\Big), \notag\\
             V_{\sigma} \leq 2J-2h+ H(\mathcal{E}_{B_5}(r^*+1))-H(\mathcal{E}(r^*+1))=2J-10h+10h\delta<2J, \; 
             \text{for } \delta \in \Big(\frac{1}{2},1\Big). \notag
         \end{align}
    Thus we find
    \begin{equation}
        V^*_{Z_2}=\max_{\sigma \in Z_2} V_{\sigma} < 2J.
    \end{equation}

    In conclusion, we have $V^*_Z=\max\{V^*_{Z_1},V^*_{Z_2}\} < 2J$.

    \paragraph{Case $R_1$.} For any configuration $\sigma\in R_1$ we construct a path $\overline{\omega}\in \Theta(\sigma,I_{\sigma} \cap (R_1 \cup Z_1)))$. Starting from $\sigma\equiv \omega_0\in R_1$, let us define $\omega_1$. Consider the corner in one of the shortest sides of the cluster and let $j$ be a site belonging to it. We define $\omega_1:=\omega_0^{(j)}$, i.e., $\sigma(j)$ switches the sign. Consider $j'$ the other site belonging to the corner and define $\omega_2:=\omega_1^{(j')}$, in this way $\sigma(j')$ switches the sign. By iterating this procedure along the shortest side, a bar of the cluster is erased and we obtain the configuration $\eta \equiv \omega_l$, where $l$ is the cardinality of the considered bar. We observe that the greatest value of $l$ is always smaller than the cardinality $k$ of the greatest bar of the quasi-regular hexagon contained in the cluster, that is $l < k$. Analogously to the case $Z_1$, $\omega_l \in I_{\sigma}$. Then $V_{\sigma} < 2J-2h$. Therefore,
    \begin{equation}
    V^*_{R_1}= \max_{\sigma \in R_1} V_{\sigma} < 2J-2h.
    \end{equation}
    \paragraph{Case $R_2$.} For any configuration $\sigma\in R_2$ we construct a path $\overline{\omega}\in \Theta(\sigma,I_{\sigma} \cap (R_2 \cup Z_2 \cup \{\underline{+1} \}))$. Starting from $\sigma\equiv \omega_0\in R_2$, let us define $\omega_1$. Consider the corner in one of the shortest sides of the cluster and let $j$ be a site belonging to it. If the cardinality of the bar $l$ of the shortest side is smaller than $2(r^*+1)-1$, we define $\omega_1:=\omega_0^{(j)}$, i.e. $\sigma(j)$ switches the sign. Consider the other site $j'$ belonging to the corner and define $\omega_2:=\omega_1^{(j')}$, in this way $\sigma(j')$ switches the sign. By iterating this procedure along the shortest side, a bar of the cluster is erased and we obtain the configuration $\eta \equiv \omega_l$, where $l$ is the cardinality of the considered bar. Since $l< 2(r^*+1)-1$, we observe that the greatest value of $l$ is always smaller then the cardinality $k$ of the greatest bar of the quasi-regular hexagon contained in the cluster, that is $l < k$. Analogously to the case $Z_1$, $\omega_l \in I_{\sigma}$. Thus $V_{\sigma} < 2J-2h$.
    
    If the cardinality of the bar $l$ of the shortest side is bigger than $2(r^*+1)-1$, consider the site $j_1$ at distance one from $j$ and such that $\sigma(j_1)=-1$. We define $\omega_1:=\omega_0^{(j_1)}$, i.e. $\sigma(j_1)$ switches the sign. Consider $j_2$ the site at distance one from $j_1$ such that $\sigma(j_2)=-1$ and $d(j_2,j')=2$ where $j' \neq j$ is another site of the initial cluster. We define $\omega_2:=\omega_1^{(j_2)}$, $\omega_3:=\omega_2^{(j_3)}$, where $j_3$ is the site at distance one from $j_2$ such that $\sigma(j_3)=-1$ and $d(j_3,j')=1$ where $j' \neq j$ is another site of the initial cluster. By iterating this procedure along the considered side, a bar is added to the initial cluster. Analogously to the case $Z_2$, $\omega_l \in I_{\sigma}$ since $l > 2(r^*+1)-1$. Thus $V_{\sigma} \leq 2J-2h$ and
    \begin{equation}
    V^*_{R_2} \leq 2J-2h.
    \end{equation}
    In conclusion, we have $V^*_R=\max\{V^*_{R_1},V^*_{R_2}\} < 2J$. 

    \paragraph{Case $U_1$.} For every configuration $\sigma$ in $U_1$, all clusters are non-interacting and are of the same type of those in $Z_1$ or $R_1$. If $\sigma$ contains a cluster that is not a quasi-regular hexagon, then we take our path to be the path that cuts a bar, analogously to what has been done for $R_1$. We get a configuration in $\mathcal{I}_\sigma \cap U_1$. Otherwise, if all clusters are quasi-regular hexagons, then we take our path to be the path that cuts a bar of the cluster, analogously to what has been done for $Z_1$. We get a configuration in $\mathcal{I}_\sigma \cap (U_1 \cup Z_1)$. So, we have
    \begin{equation}
        V^*_{U_1}=\max\{V^*_{R_1},V^*_{Z_1}\} < 2J-2h.
    \end{equation}

    \paragraph{Case $U_2$.} For every configuration $\sigma$ in $U_2$, there exists at least a cluster of the same type of those in $Z_2$ or $R_2$. If $\sigma$ contains a cluster of the type of those in $R_2$, i.e. $\sigma$ contains a cluster that is not a quasi-regular hexagon, we take our path to be the path that cuts a bar as it has been done for $R_2$. We get a configuration in $\mathcal{I}_\sigma \cap U_2$. Otherwise, if the cluster is like those in $Z_2$, i.e. the cluster is a quasi-regular hexagon, then we take the path that adds a bar to the quasi-regular hexagon, alike the cases encountered when considering $Z_2$. We get a configuration in $\mathcal{I}_\sigma \cap (U_2 \cup \{\underline{+1}\})$. So, we have
    \begin{equation}
        V^*_{U_2}=\max \{V^*_{R_2},V^*_{Z_2}\} < 2J.
    \end{equation}
    We conclude that 
    \begin{equation*}
    V^*_U=\max\{V^*_{U_1},V^*_{U_2}\}=V^*_Z.
    \end{equation*}

\paragraph{Case $Y$.} We analyze four different kinds of configurations.  
\begin{enumerate}
    \item If $\sigma_1 \equiv \omega_0$ is a configuration in $Y$ such that it has a minus strip that contains at least a hexagon of the type of those in $Z_2$ or $R_2$, then we take our path to be
    as that in the case $U_2$ and we obtain a configuration in $\mathcal{I}_\sigma \cap ( Y \cup \{\underline{+1}\})$. So, we have 
    \begin{equation}
        V_{\sigma_1}=V^*_{U_2} < 2J
    \end{equation}
    \item If $\sigma_2 \equiv \omega_0$ is a configuration in $Y$ such that it contains a plus strip with width greater than one, then we take our path as follows. Pick a site $j$ with 
    $\sigma(j) = - 1$ at distance one from the strip.
    We define $\omega_1=\omega_0^{(j)}$, i.e., $\sigma(j)$ switches the sign. Then pick a site $j_1$ at distance one from $j$ such that $\sigma(j_1)=-1$ and define $\omega_2=\omega_1^{(j_1)}$. Consider a site $j_2$ nearest neighbor of $j_1$ such that $\sigma(j_2)=-1$ and define $\omega_3=\omega_2^{(j_2)}$. By iterating these last two steps, we obtain a configuration in $\mathcal{I}_\sigma \cap (Y \cup \{\underline{+1}\})$. In order to determine where the maximum is attained we write for $n=2,\cdots,L-1$ 
    \begin{align}
       H(\omega_n)-H(\omega_0)=
       \begin{cases}
       \sum_{t=2, \, t \, even}^{n}(H(\omega_{t})-H(\omega_{t-2})) & \text{if even n}, \\
       \sum_{t=2, \, t \, even}^{n}(H(\omega_{t})-H(\omega_{t-2}))+ H(\omega_{n})-H(\omega_{n-1}) & \text{if odd n}.
       \end{cases}
    \end{align}
 From the values in Table~\ref{Tab} we obtain
 \begin{align}
      & H(\omega_{1})-H(\omega_{0})=J-h, \\
      & H(\omega_{2})-H(\omega_{0})=2J-2h, \\
      & H(\omega_{L})-H(\omega_{L-1})=-(J+h),
 \end{align}
and for every $t=3,\cdots,L-1$
        \begin{equation}\label{reversibility3}
            H(\omega_{t})-H(\omega_{t-2})=[H(\omega_{t})-H(\omega_{t-1})]+[H(\omega_{t-1})-H(\omega_{t-2})]=-2h.
        \end{equation}
        The previous can also be obtained by using \eqref{eq:peierls_hamiltonian} and observing that flipping two spins, as described above, corresponds to increasing $N^{+}(\sigma)$ and leads to a cluster with a Peierls contour of the same length. 
        So, we have
    \begin{align}\label{differenza_hamiltoniana_strisce}
    \scriptstyle
    H(\omega_n)-H(\omega_0) =
    \begin{cases}
        J-h, & \text{if } n=1,\\
        2J-2h, & \text{if } n=2,\\
        -2h\frac{n-2}{2}+(2J-2h)=2J-nh, & \text{if even } n=4,\cdots,L-1,\\
        -2h\frac{n-1}{2}+(J-h)=J-nh, & \text{if odd } n=3,\cdots,L-1,\\
        -2h\frac{n-4}{2}+(2J-2h)-(2J+2h)=-Lh, & \text{if } n=L.\\
    \end{cases}
    \end{align}
The result of the last equations is obtained for $n=2$ and it is attained in $\omega_{2}$. Using \eqref{differenza_hamiltoniana_strisce}, we prove that $\omega_L \in I_{\sigma} \cap (Y \cup \{\underline{+1}\})$
    \begin{align}
     H(\omega_0)-H(\omega_L)=Lh>0.
    \end{align}
Thus 
        \begin{align}
        V_{\sigma_2}\leq H(\omega_2)-H(\omega_0)=2J-2h.
        \end{align}
    \item If $\sigma_3 \equiv \omega_0$ is a configuration in $Y$ such that it contains a plus strip with width one, then we take our path as follows. Pick a site $j$ in the strip and define $\omega_1:=\omega_0^{(j)}$, i.e., $\sigma(j)$ switches the sign. The difference of the energy given by the Table~\ref{Tab} is $H(\omega_1)-H(\omega_0)=J+h$. Considering $j_2,j_3,\ldots$ the nearest sites in the strip, we define $\omega_2:=\omega_1^{j_2}$, $\omega_3:=\omega_2^{(j_3)}$ and so on until we obtain a configuration in $\mathcal{I}_\sigma \cap (Y \cup Z_1 \cup U_1 \cup \{\underline{-1}\})$. Finally, let us check that $\omega_L \in I_{\sigma} \cap (Y \cup \{\underline{+1}\})$. Using Table~\ref{Tab}, we get
    \begin{align}
        H(\omega_L)- H(\omega_0) & = [H(\omega_L)-H(\omega_{L-1})]+ \sum_{t=2}^{L-1}(H(\omega_t)-H(\omega_{t-1})+[H(\omega_1)-H(\omega_0)]  \notag \\
       & = -(3J-h)-(L-2)(J-h)+(J+h)=-L(J-h)<0.
    \end{align}
    Thus
    \begin{equation}
        V_{\sigma_3} \leq J+h.
    \end{equation}
\end{enumerate}
    We conclude that 
    \begin{equation*}
    V^*_Y=\max\{V_{\sigma_1},V_{\sigma_2}, V_{\sigma_3}\} < 2J.
    \end{equation*}
\end{proof}
  
\begin{proof}[Proof of Proposition~\ref{teoRP}] 
By applying \cite[Theorem 3.1]{MNOS04} for $V^*=2J$ and Lemma \ref{EST}, we get \eqref{recurrence2J}.
\end{proof}

\section{Identification of maximal stability level}\label{bound}
\subsection{Reference path}\label{referencepath}
In this Section we define our \emph{reference path}, that is a sequence of configurations from $\underline{-1}$ to $\underline{+1}$, that are increasing clusters \emph{as close as possible to quasi-regular hexagonal shape}. 
We first describe intuitively this path: starting from $\underline{-1}$ we flip the spin at the origin and then consecutively filling (flipping the spins from minus to plus) a standard cluster, see Figure \ref{standardfigure}. Let $\omega^* \in \Theta(\underline{-1},\underline{+1})$, we will construct a path in which at each step we flip one spin from minus to plus. Starting from the origin, we add clockwise six triangular units to obtain the first regular hexagon with radius $r=1$, that is $\mathcal{E}(1)$. Then for each $r=1,\ldots,m$ we construct the bar on the top of the hexagon $\mathcal{E}(r)$, adding consecutive triangular units until we obtain $\mathcal{E}_{B_1}(r)$. Next we fill the bar on the top right adding consecutive triangular units until we get $\mathcal{E}_{B_2}(r)$. We go on in the same manner adding bars clockwise, until we get $\mathcal{E}_{B_3}(r)$,\ldots,$\mathcal{E}_{B_6}(r)\equiv \mathcal{E}(r+1)$ (see Figure \ref{standardfigure}). We iterate this procedure until the hexagon is large enough to wrap around the torus along one direction, giving rise to two triangles of minuses with side length $\frac{L}{2}$. Now the reference path fills six triangular units ``covering'' all the $\frac{5}{3}\pi$ angles.
As a result, it is possible to identify six bars of length two, each adjacent to one of the filled triangular units. For each initial triangle of minuses choose only one of the previous bars and fill it. We obtain two bars of length three, each intersecting the previous bars in the external boundary. Iterate this procedure filling at each step the bars of length four, five, \ldots, $\frac{L}{2}$. 

\begin{figure}[htb!] 
\centering \includegraphics[scale=0.6]{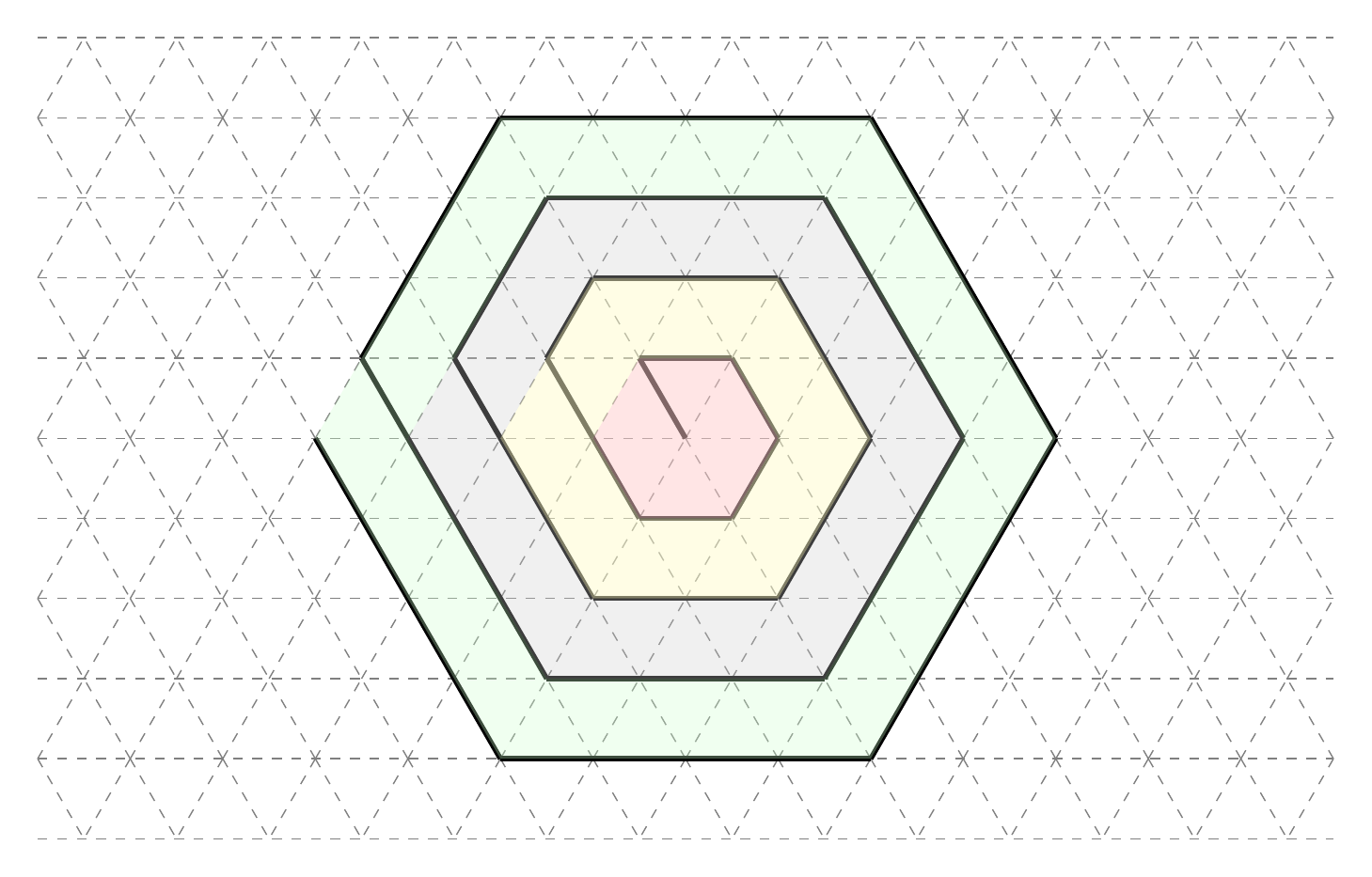} 
\caption{The construction of a standard polyiamond.} \label{standardfigure}
\end{figure}

All configurations in this path contain a standard cluster with radius $r < \frac{L}{2}$. Note that the standard cluster $\mathcal{E}(\frac{L}{2})$ wraps around the torus.
\begin{proposition}\label{CHbetweenQRHcritical}
The maximum of the energy in $\omega^*$ between two consecutive quasi-regular hexagons $\Phi_{\omega^*}(\mathcal{E}_{B_i}(r),\mathcal{E}_{B_{i+1}}(r))$ for every $i=0,\ldots,5$ is achieved in the standard polyiamond obtained adding to $\mathcal{E}_{B_i}(r)$ one elementary rhombus with two pluses along the longest side consecutive to $B_i$ clockwise.
\end{proposition}

\begin{proposition}\label{CHbetweenQRHcritical2}
The maximum of the energy in $\omega^*$ between two consecutive quasi-regular hexagons $\Phi_{\omega^*}(\mathcal{E}_{B_i}(r),\mathcal{E}_{B_{i-1}}(r))$ for every $i=1,\ldots,6$ is achieved in the standard polyiamond obtained removing counter-clockwise from $\mathcal{E}_{B_i}(r)$ a number of triangular units equals to $||B_i||-2$.
\end{proposition}

\begin{proof}[Proof of Proposition~\ref{CHbetweenQRHcritical}]
Let $E_{B_i}(r)$ and $E_{B_{i+1}}(r)$ be two quasi-regular hexagons for some $r \in \mathbb{N}$ and some $i=0,\ldots,5$. Let $A^{(n)}$ be the area obtained adding $n$ triangular units to the area of the quasi-regular hexagon $E_{B_i}(r)$, where $n=0,\ldots,||B_{i+1}||$. Note that $A^{(n)}$ is the area of the standard polyiamond $S(A^{(n)})$. We observe that $S(A^{(0)}) \equiv E_{B_i}(r)$ and $S(A_{||B_{i+1}||}) \equiv E_{B_{i+1}}(r)$. By Table~\ref{Tab}, we have
\begin{align}
H(\mathcal{S}(A^{(n)}))-H(\mathcal{S}(A^{(n-1)}))=
\begin{cases}
    J-h, & \text{if } n=1,\\
    J-h, & \text{if } n \text{ is even},\\
    -(J+h), & \text{if } n \neq 1 \text{ is odd}.\\
    \end{cases}
\end{align}

It follows that
\begin{align}
H(\mathcal{S}(A^{(n)}))-H(\mathcal{E}_{B_i}(r))=
\begin{cases}
    J-nh, & \text{if } n \text{ is odd},\\
    2J-nh, & \text{if } n \text{ is even}.\\
    \end{cases}
\end{align}
Since the r.h.s. of the last equation decreases with 
$n$ in both the odd and the even case, it is immediate to check that the maximum is attained for $n=2$ namely in $\mathcal{S}(A^{(2)})$.

\end{proof}

\begin{proof}[Proof of Proposition~\ref{CHbetweenQRHcritical2}]
Let $E_{B_i}(r)$ and $E_{B_{i-1}}(r)$ be two quasi-regular hexagons for some $r \in \mathbb{N}$ and some $i=1,\ldots,6$. Let $A^{(n)}$ be the area obtained adding $n$ triangular units to the area of the quasi-regular hexagon $E_{B_{i-1}}(r)$, where $n=0,\ldots,||B_{i}||$. Note that $S(A^{(n)})$ can be obtained either by removing $||B_i||-n$ triangular units from $E_{B_i}(r)$ or by adding $n$ triangular units to $E_{B_{i-1}}(r)$. 
We recall that \emph{removing} a triangular units means to flip a plus spin into a minus spin. 
By Table~\ref{Tab}, we have
\begin{align}
H(\mathcal{S}(A^{(n-1)}))-H(\mathcal{S}(A^{(n)}))=
\begin{cases}
    J+h, & \text{if } n \neq ||B_i|| \text{ is odd},\\
    -(J-h), & \text{if } n \text{ is even},\\
    -(J-h), & \text{if } n=||B_i||.\\
    \end{cases}
\end{align}

It follows that
\begin{align}
H(\mathcal{S}(A^{(n)}))-H(\mathcal{E}_{B_i}(r))=
\begin{cases}
    J+nh, & \text{if } n \neq ||B_i|| \text{ is odd},\\
    nh, & \text{if } n \text{ is even},\\
    -(J-nh), & \text{if } n=||B_i||.
    \end{cases}
\end{align}
Since the r.h.s. of the last equation increases with $n$ in all three cases and since $||B_i||$ is odd by Definition \ref{def:quasi_reg_hex}, the maximum is attained removing $||B_i||-2$ triangular units from $E_{B_i}(r)$. So we obtain $\mathcal{S}(A^{(2)})$.
\end{proof}

Recalling \eqref{raggiocritico}, from now on the strategy of the proof is to divide the reference path $\omega^*$ into three regions depending on $r$: 
\begin{itemize}[label=\raisebox{0.25ex}{\tiny \textbullet}]
    \item the region $r \leq r^*$ will be considered in Proposition \ref{rminore};
    \item the region $r=r^*+1$ will be considered in Proposition \ref{ruguale};
    \item the region $r \geq r^*+2$ will be considered in Proposition \ref{rmaggiore}.
\end{itemize}
\begin{proposition}\label{rminore}
If $r \leq r^*$, then the communication height between two consecutive regular hexagons $\Phi_{\omega^*}(\mathcal{E}(r), \mathcal{E}(r+1))$ along the path  $\omega^*$ is achieved in a configuration with a standard cluster such that the number of its triangular units is $\tilde A=6r^2+10r+5$, that is $\Phi_{\omega^*}(\mathcal{E}(r), \mathcal{E}(r+1))=\Phi_{\omega^*}(\mathcal{E}_{B_5}(r), \mathcal{E}(r+1))=H(\mathcal{S}(\tilde A))-H(\underline{-1})$.
Moreover, $\Phi_{\omega^*}(\underline{-1}, \mathcal{E}(r^*+1))=\Phi_{\omega^*}(\mathcal{E}(r^*), \mathcal{E}(r^*+1))=H(\mathcal{S}(A_1))-H(\underline{-1})$ is achieved in a configuration with a standard cluster $\mathcal{S}(A_1)$ such that $A_1:=6{r^*}^2+10r^*+5$.
\end{proposition}
\begin{proposition}\label{rmaggiore}
If $r \geq r^*+2$, then the communication height between two consecutive regular hexagons $\Phi_{\omega^*}(\mathcal{E}(r), \mathcal{E}(r+1))$ along the path  $\omega^*$ is achieved in a configuration with a standard cluster such that the number of its triangular units is $\tilde A=6r^2+2$, that is $\Phi_{\omega^*}(\mathcal{E}(r), \mathcal{E}(r+1))=\Phi_{\omega^*}(\mathcal{E}(r), \mathcal{E}_{B_1}(r))=H(\mathcal{S}(\tilde A))-H(\underline{-1})$.
Moreover, $\Phi_{\omega^*}(\mathcal{E}(r^*+2), \underline{+1})= \Phi_{\omega^*}(\mathcal{E}(r^*+2),\mathcal{E}(r^*+3))=H(\mathcal{S}(A_2))-H(\underline{-1})$ is achieved in a configuration with a standard cluster $\mathcal{S}(A_2)$ such that $A_2:=6(r^*+2)^2+2$.
\end{proposition}
\begin{proposition}\label{ruguale} 
If $r=r^*+1$, then $r= \lfloor J/2h+1/2 \rfloor$ and the communication height $\Phi_{\omega^*}(\mathcal{E}(r^*+1),\mathcal{E}(r^*+2))$ along the path  $\omega^*$ is achieved in a configuration with a standard cluster $\mathcal{S}(A_3)$ with $A_3:=6(r^*+1)^2+2(r^*+1)+1$.
\end{proposition}
\begin{proof}[Proof of Proposition~\ref{rminore}] 
Let $m$ be the number of site pairs $(i,j)$ such that $d(i,j)=1$ with $i,j \in \Lambda$, and let $p(S(A))$ be the perimeter of the standard hexagon $S(A)$. By Definition \ref{eq:peierls_hamiltonian}, we have:
\begin{align}
    H(\mathcal{S}(A))-H(\underline{-1})=Jp(S(A))-hA.
\end{align}
So, by Remark \ref{proprietaS} and Definition \ref{eq:peierls_hamiltonian}, it follows that:
\begin{equation}\label{eq:standard_energy_peierls}
    H(\mathcal{S}(A))-H(\underline{-1})=
    \begin{cases}
    -6r^2h+6rJ+2J-2h & \mbox{ for } A=6r^2+2\\
    -6r^2h+6rJ-2rh+3J-h & \mbox{ for } A=6r^2+2r+1 \\
    -6r^2h+6rJ-4rh+4J-2h & \mbox{ for } A=6r^2+4r+2\\
    -6r^2h+6rJ-6rh+5J-3h & \mbox{ for } A=6r^2+6r+3 \\
    -6r^2h+6rJ-8rh+6J-4h & \mbox{ for } A=6r^2+8r+4\\
    -6r^2h+6rJ-10rh+7J-5h & \mbox{ for } A=6r^2+10r+5
    \end{cases}
\end{equation}
We compare $\Phi_{\omega^*}(\mathcal{E}(r),\mathcal{E}_{B_1}(r))=\Phi_{\omega^*}(\mathcal{S}(6r^2),\mathcal{S}(6r^2+2r-1))$ with $\Phi_{\omega^*}(\mathcal{E}_{B_1}(r),\mathcal{E}_{B_2}(r))=\Phi_{\omega^*}(\mathcal{S}(6r^2+2r-1),\mathcal{S}(6r^2+4r))$. 
By Proposition \ref{CHbetweenQRHcritical}, we have: $\Phi_{\omega^*}(\mathcal{E}(r),\mathcal{E}_{B_1}(r))=H(\mathcal{S}(6r^2+2))-H(\underline{-1})$ and $\Phi_{\omega^*}(\mathcal{E}_{B_1}(r),\mathcal{E}_{B_2}(r))=H(\mathcal{S}(6r^2+2r+1))-H(\underline{-1})$. 
By \eqref{eq:standard_energy_peierls} it follows
\begin{equation}\label{dis1}
    -6r^2h+6rJ+2J-2h \leq -6r^2h+6rJ-2rh+3J-h.
\end{equation}
This inequality holds if $r \leq \frac{J}{2h}+\frac{1}{2}$. Since we assume $r \leq r^*$, the inequality \eqref{dis1} is satisfied.

\noindent We compare 
$\Phi_{\omega^*}(\mathcal{E}_{B_1}(r), \mathcal{E}_{B_2}(r))$
$=\Phi_{\omega^*}(\mathcal{S}(6r^2+2r-1), \mathcal{S}(6r^2+4r))$ 
with
$\Phi_{\omega^*}(\mathcal{E}_{B_2}(r), \mathcal{E}_{B_3}(r))$ $=\Phi_{\omega^*}(\mathcal{S}(6r^2+4r), \mathcal{S}(6r^2+6r+1))$. 
By Proposition \ref{CHbetweenQRHcritical}, we have: $\Phi_{\omega^*}(\mathcal{E}_{B_1}(r),\mathcal{E}_{B_2}(r))=H(\mathcal{S}(6r^2+2r+1))-H(\underline{-1})$ and $\Phi_{\omega^*}(\mathcal{E}_{B_2}(r),\mathcal{E}_{B_3}(r))=H(\mathcal{S}(6r^2+4r+2))-H(\underline{-1})$. 
By \eqref{eq:standard_energy_peierls} it follows
\begin{equation}\label{dis2}
     -6r^2h+6rJ-2rh+3J-h \leq -6r^2h+6rJ-4rh+4J-2h.
\end{equation}
This inequality holds if $r \leq \frac{J}{2h}-\frac{1}{2}$. Since we assume $r \leq r^*$, the inequality \eqref{dis2} is satisfied.

\noindent We compare $\Phi_{\omega^*}(\mathcal{E}_{B_2}(r),\mathcal{E}_{B_3}(r))=\Phi_{\omega^*}(\mathcal{S}(6r^2+4r),\mathcal{S}(6r^2+6r+1))$ with $\Phi_{\omega^*}(\mathcal{E}_{B_3}(r),\mathcal{E}_{B_4}(r))$ = $\Phi_{\omega^*}(\mathcal{S}(6r^2+6r+1),\mathcal{S}(6r^2+8r+2))$.
By Proposition \ref{CHbetweenQRHcritical}, we have: $\Phi_{\omega^*}(\mathcal{E}_{B_2}(r),\mathcal{E}_{B_3}(r))=H(\mathcal{S}(6r^2+4r+2))-H(\underline{-1})$ and $\Phi_{\omega^*}(\mathcal{E}_{B_3}(r),\mathcal{E}_{B_4}(r))=H(\mathcal{S}(6r^2+6r+3))-H(\underline{-1})$. 
By \eqref{eq:standard_energy_peierls} it follows
\begin{equation}\label{dis3}
     -6r^2h+6rJ-4rh+4J-2h \leq -6r^2h+6rJ-6rh+5J-3h.
\end{equation}
This inequality holds if $r \leq \frac{J}{2h}-\frac{1}{2}$. Since we assume $r \leq r^*$, the inequality \eqref{dis3} is satisfied. 

\noindent We compare $\Phi_{\omega^*}(\mathcal{E}_{B_3}(r),\mathcal{E}_{B_4}(r))=\Phi_{\omega^*}(\mathcal{S}(6r^2+6r+1),\mathcal{S}(6r^2+8r+2))$ with\\ $\Phi_{\omega^*}(\mathcal{E}_{B_4}(r),\mathcal{E}_{B_5}(r))=\Phi_{\omega^*}(\mathcal{S}(6r^2+8r+2),\mathcal{S}(6r^2+10r+3))$.
By Proposition \ref{CHbetweenQRHcritical}, we have: $\Phi_{\omega^*}(\mathcal{E}_{B_3}(r),\mathcal{E}_{B_4}(r))=H(\mathcal{S}(6r^2+6r+3))-H(\underline{-1})$ and $\Phi_{\omega^*}(\mathcal{E}_{B_4}(r),\mathcal{E}_{B_5}(r))=H(\mathcal{S}(6r^2+8r+4))-H(\underline{-1})$.
By \eqref{eq:standard_energy_peierls} it follows
\begin{equation}\label{dis4}
     -6r^2h+6rJ-6rh+5J-3h \leq -6r^2h+6rJ-8rh+6J-4h.
\end{equation}
This inequality holds if $r \leq \frac{J}{2h}-\frac{1}{2}$. Since we assume $r \leq r^*$, the inequality \eqref{dis4} is satisfied.

\noindent We compare $\Phi_{\omega^*}(\mathcal{E}_{B_4}(r),\mathcal{E}_{B_5}(r))=\Phi_{\omega^*}(\mathcal{S}(6r^2+8r+2),\mathcal{S}(6r^2+10r+3))$ with\\ $\Phi_{\omega^*}(\mathcal{E}_{B_5}(r),\mathcal{E}(r+1))=\Phi_{\omega^*}(\mathcal{S}(6r^2+10r+3),\mathcal{S}(6r^2+12r+6))$. 
By Proposition \ref{CHbetweenQRHcritical}, we have: $\Phi_{\omega^*}(\mathcal{E}_{B_4}(r),\mathcal{E}_{B_5}(r))=H(\mathcal{S}(6r^2+8r+4))-H(\underline{-1})$ and $\Phi_{\omega^*}(\mathcal{E}_{B_5}(r),\mathcal{E}(r+1))=H(\mathcal{S}(6r^2+10r+5))-H(\underline{-1})$. 
By \eqref{eq:standard_energy_peierls} it follows
\begin{equation}\label{dis5}
     -6r^2h+6rJ-8rh+6J-4h \leq -6r^2h+6rJ-10rh+7J-5h.
\end{equation}
This inequality holds if $r \leq \frac{J}{2h}-\frac{1}{2}$. Since we assume $r \leq r^*$, the inequality \eqref{dis5} is satisfied. 

Thus the communication height between two consecutive regular hexagons along the path  $\omega^*$ is achieved in $\mathcal{S}(6r^2+10r+5)$, that is $\Phi_{\omega^*}(\mathcal{E}(r), \mathcal{E}(r+1))=\Phi_{\omega^*}(\mathcal{E}_{B_5}(r), \mathcal{E}(r+1))=H(\mathcal{S}(\tilde A))-H(\underline{-1})$. The maximum of the function $H(\mathcal{S}(\tilde A))-H(\underline{-1})=-6r^2h+6rJ-10rh+7J-5h$ is obtained in $r=\frac{J}{2h}-\frac{5}{6}$. However $r \in \mathbb{N}$ and $r \leq r^*$, therefore the maximum is attained in $r^*$ and $\Phi_{\omega^*}(\underline{-1}, \mathcal{E}(r^*+1))=\Phi_{\omega^*}(\mathcal{E}(r^*), \mathcal{E}(r^*+1))=H(\mathcal{S}(A_1))-H(\underline{-1})$ where $\mathcal{S}(A_1)$ is a configuration with a standard cluster such that $A_1:=6{r^*}^2+10r^*+5$.
\end{proof}

\begin{proof}[Proof of Proposition~\ref{rmaggiore}] Using Remark \ref{proprietaS} and Proposition \ref{CHbetweenQRHcritical}, we compare\\ $\Phi_{\omega^*}(\mathcal{E}(r),\mathcal{E}_{B_1}(r))=\Phi_{\omega^*}(\mathcal{S}(6r^2),\mathcal{S}(6r^2+2r-1))=H(\mathcal{S}(6r^2+2))-H(\underline{-1})$ with\\ $\Phi_{\omega^*}(\mathcal{E}_{B_1}(r),\mathcal{E}_{B_2}(r))=\Phi_{\omega^*}(\mathcal{S}(6r^2+2r-1),\mathcal{S}(6r^2+4r))=H(\mathcal{S}(6r^2+2r+1))-H(\underline{-1})$:
\begin{equation}\label{dis1sup}
    -6r^2h+6rJ+2J-2h \geq -6r^2h+6rJ-2rh+3J-h.
\end{equation}
This inequality holds if $r \geq \frac{J}{2h}+\frac{1}{2}$. Since we assume $r \geq r^*+2$, the inequality \eqref{dis1sup} is satisfied. 

Now, we compare $\Phi_{\omega^*}(\mathcal{E}_{B_1}(r),\mathcal{E}_{B_2}(r))=\Phi_{\omega^*}(\mathcal{S}(6r^2+2r-1),\mathcal{S}(6r^2+4r))=H(\mathcal{S}(6r^2+2r+1))-H(\underline{-1})$ with $\Phi_{\omega^*}(\mathcal{E}_{B_2}(r),\mathcal{E}_{B_3}(r))=\Phi_{\omega^*}(\mathcal{S}(6r^2+4r),\mathcal{S}(6r^2+6r+1))=H(\mathcal{S}(6r^2+4r+2))-H(\underline{-1})$:
\begin{equation}\label{dis2sup}
     -6r^2h+6rJ-2rh+3J-h \geq -6r^2h+6rJ-4rh+4J-2h.
\end{equation}
This inequality holds if $r \geq \frac{J}{2h}-\frac{1}{2}$. Since we assume $r \geq r^*+2$, the inequality \eqref{dis2sup} is satisfied.

Now, we compare $\Phi_{\omega^*}(\mathcal{E}_{B_2}(r),\mathcal{E}_{B_3}(r))=\Phi_{\omega^*}(\mathcal{S}(6r^2+4r),\mathcal{S}(6r^2+6r+1))=H(\mathcal{S}(6r^2+4r+2))-H(\underline{-1})$ with $\Phi_{\omega^*}(\mathcal{E}_{B_3}(r),\mathcal{E}_{B_4}(r))=\Phi_{\omega^*}(\mathcal{S}(6r^2+6r+1),\mathcal{S}(6r^2+8r+2))=H(\mathcal{S}(6r^2+6r+3))-H(\underline{-1})$:
\begin{equation}\label{dis3sup}
     -6r^2h+6rJ-4rh+4J-2h \geq -6r^2h+6rJ-6rh+5J-3h.
\end{equation}
This inequality holds if $r \geq \frac{J}{2h}-\frac{1}{2}$. Since we assume $r \geq r^*+2$, the inequality \eqref{dis3sup} is satisfied.

Now, we compare $\Phi_{\omega^*}(\mathcal{E}_{B_3}(r),\mathcal{E}_{B_4}(r))=\Phi_{\omega^*}(\mathcal{S}(6r^2+6r+1),\mathcal{S}(6r^2+8r+2))=H(\mathcal{S}(6r^2+6r+3))-H(\underline{-1})$ with $\Phi_{\omega^*}(\mathcal{E}_{B_4}(r),\mathcal{E}_{B_5}(r))=\Phi_{\omega^*}(\mathcal{S}(6r^2+8r+2),\mathcal{S}(6r^2+10r+3))=H(\mathcal{S}(6r^2+8r+4))-H(\underline{-1})$:
\begin{equation}\label{dis4sup}
     -6r^2h+6rJ-6rh+5J-3h \geq -6r^2h+6rJ-8rh+6J-4h.
\end{equation}
This inequality holds if $r \geq \frac{J}{2h}-\frac{1}{2}$. Since we assume $r \geq r^*+2$, the inequality \eqref{dis4sup} is satisfied. 

Now, we compare $\Phi_{\omega^*}(\mathcal{E}_{B_4}(r),\mathcal{E}_{B_5}(r))=\Phi_{\omega^*}(\mathcal{S}(6r^2+8r+2),\mathcal{S}(6r^2+10r+3))=H(\mathcal{S}(6r^2+8r+4))-H(\underline{-1})$with $\Phi_{\omega^*}(\mathcal{E}_{B_5}(r),\mathcal{E}(r+1))=\Phi_{\omega^*}(\mathcal{S}(6r^2+10r+3),\mathcal{S}(6r^2+12r+6))=H(\mathcal{S}(6r^2+10r+5))-H(\underline{-1})$:
\begin{equation}\label{dis5sup}
     -6r^2h+6rJ-8rh+6J-4h \geq -6r^2h+6rJ-10rh+7J-5h.
\end{equation}
This inequality holds if $r \geq \frac{J}{2h}-\frac{1}{2}$. Since we assume $r \geq r^*+2$, the inequality \eqref{dis5sup} is satisfied.

Thus the communication height between two consecutive regular hexagons along the path  $\omega^*$ is achieved in $\mathcal{S}(6r^2+2)$, that is $\Phi_{\omega^*}(\mathcal{E}(r), \mathcal{E}(r+1))=\Phi_{\omega^*}(\mathcal{E}(r), \mathcal{E}_{B_1}(r))=H(\mathcal{S}(\tilde A))-H(\underline{-1})$. The maximum of the function $H(\mathcal{S}(\tilde A))-H(\underline{-1})=-6r^2h+6rJ+2J-2h$ is obtained in $r=\frac{J}{2h}$, but $r \in \mathbb{N}$ and $r \geq r^*+2$, so $\Phi_{\omega^*}(\mathcal{E}(r^*+2),\underline{+1})=\Phi_{\omega^*}(\mathcal{E}(r^*+1), \mathcal{E}(r^*+2))=H(\mathcal{S}(A_2))-H(\underline{-1})$ where $\mathcal{S}(A_2)$ is a configuration with a standard cluster such that $A_2:=6(r^*+2)^2+2$.
\end{proof}
\begin{proof}[Proof of Proposition~\ref{ruguale}] 
We analyze $\Phi_{\omega^*}(\mathcal{E}(r^*+1),\mathcal{E}(r^*+2))$ using Definition \ref{hamiltonianFunction}, Remark \ref{proprietaS} and Proposition \ref{CHbetweenQRHcritical}. 
With the same arguments of the proofs of Propositions \ref{rminore} and \ref{rmaggiore}, we consider the Equation \eqref{dis1sup} which holds if and only if $r \geq \frac{J}{2h}+\frac{1}{2}$. In this case, we consider $r=r^*+1=\lfloor \frac{J}{2h}+\frac{1}{2} \rfloor$, so we have 
\begin{equation}
    \Phi_{\omega^*}(\mathcal{S}(6r^2),\mathcal{S}(6r^2+2r-1)) <\Phi_{\omega^*}(\mathcal{S}(6r^2+2r-1), \mathcal{S}(6r^2+4r)).
\end{equation}
Equations \eqref{dis2sup} - \eqref{dis5sup} hold if and only if $r \geq \frac{J}{2h}-\frac{1}{2}$
\begin{align*}
\Phi_{\omega^*}(\mathcal{S}(6r^2+2r-1), \mathcal{S}(6r^2+4r))> & \Phi_{\omega^*}(\mathcal{S}(6r^2+4r),\mathcal{S}(6r^2+6r+1)) \notag \\ 
> & \Phi_{\omega^*}(\mathcal{S}(6r^2+6r+1),\mathcal{S}(6r^2+8r+2)) \notag \\ 
> & \Phi_{\omega^*}(\mathcal{S}(6r^2+8r+2),\mathcal{S}(6r^2+10r+3)) \notag \\
> & \Phi_{\omega^*}(\mathcal{S}(6r^2+10r+3),\mathcal{S}(6r^2+12r+6)). 
\end{align*}
Then the communication height along the path $\omega^*$ between two consecutive regular hexagons with radius $r^*+1$ is $\Phi_{\omega^{*}}(\mathcal{E}(r^*+1),\mathcal{E}(r^*+2))=\Phi_{\omega^*}(\mathcal{S}(6(r^*+1)^2+2(r^*+1)-1), \mathcal{S}(6(r^*+1)^2+4(r^*+1)))$ and, by Proposition \ref{CHbetweenQRHcritical}, it is achieved in $\mathcal{S}(A_3)$, such that $A_3:=6(r^*+1)^2+2(r^*+1)+1$.
\end{proof}

\begin{corollary}\label{Phimax}
Let $\delta \in (0,1)$ such that $\frac{J}{2h}-\frac{1}{2}-\delta$ is integer. The maximum $\Phi_{\omega^*}(\underline{-1}, \underline{+1})$ along the path $\omega^*$ is achieved in a configuration with a standard cluster with area $A^*_i$ for $i \in \{1,2\}$
(see Figures \ref{cadPhi1} and \ref{cadPhi2}), where
\begin{enumerate}
        \item $A^*_1=A_1=6{r^*}^2+10r^*+5$, if $0<\delta<\frac{1}{2}$;
        \item $A^*_2=A_3=6(r^*+1)^2+2(r^*+1)+1$, if $\frac{1}{2}<\delta<1$.
    \end{enumerate}
\end{corollary}

\begin{figure}[htb!] 
\centering \includegraphics[scale=0.35]{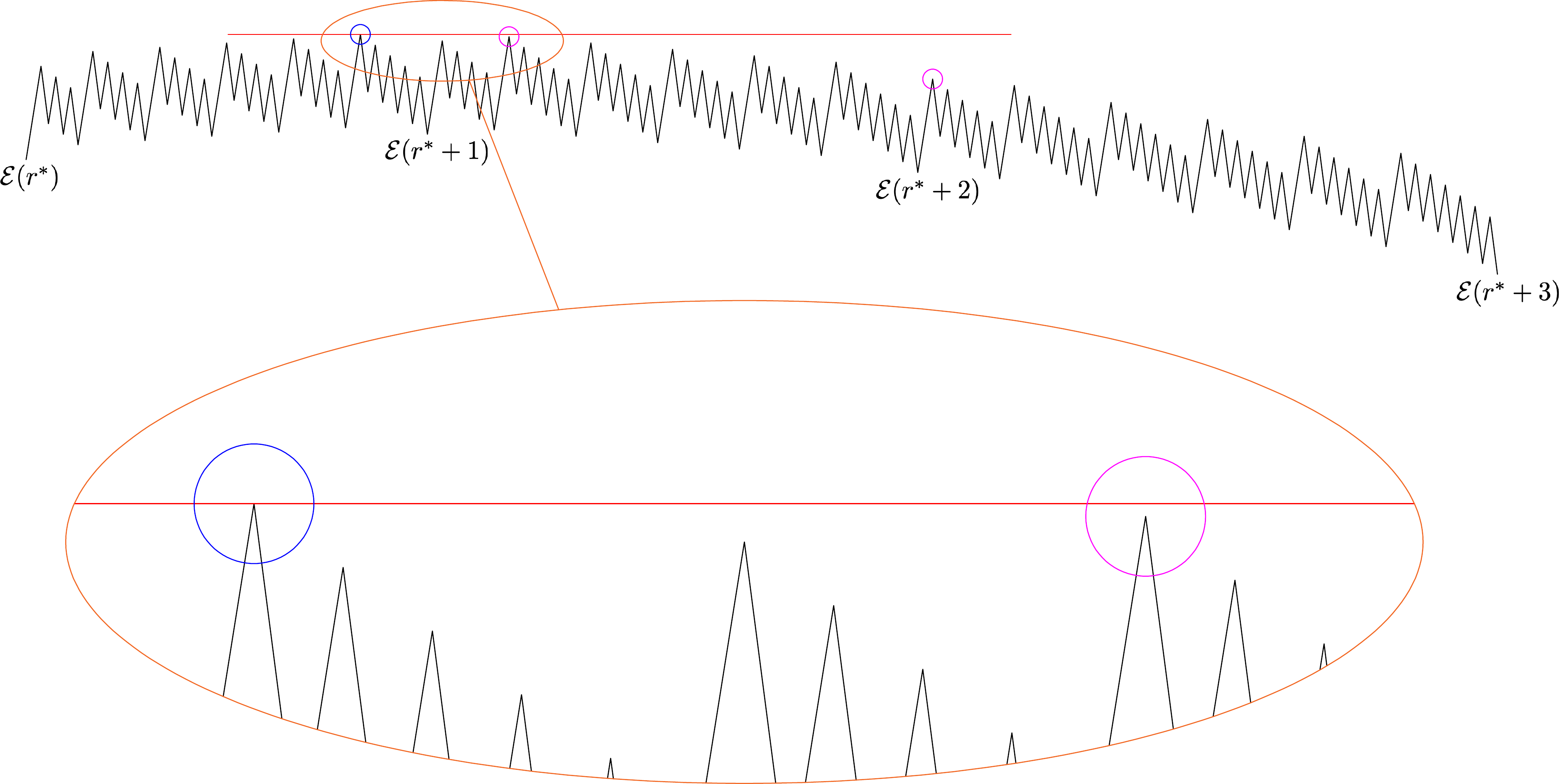} 
\caption{An example of the energy landscape between $\mathcal{E}(r^*)$ and $\mathcal{E}(r^*+3)$ for the values of the external magnetic field $h=5/7$ and the ferromagnetic interaction $J=7$, thus $\delta \in (0,1/2)$. We zoom in some part of the energy landscape, in order to compare the saddles and we highlight the maximal saddle in blue.} \label{cadPhi1}
\end{figure}
\begin{figure}[htb!] 
\centering \includegraphics[scale=0.35]{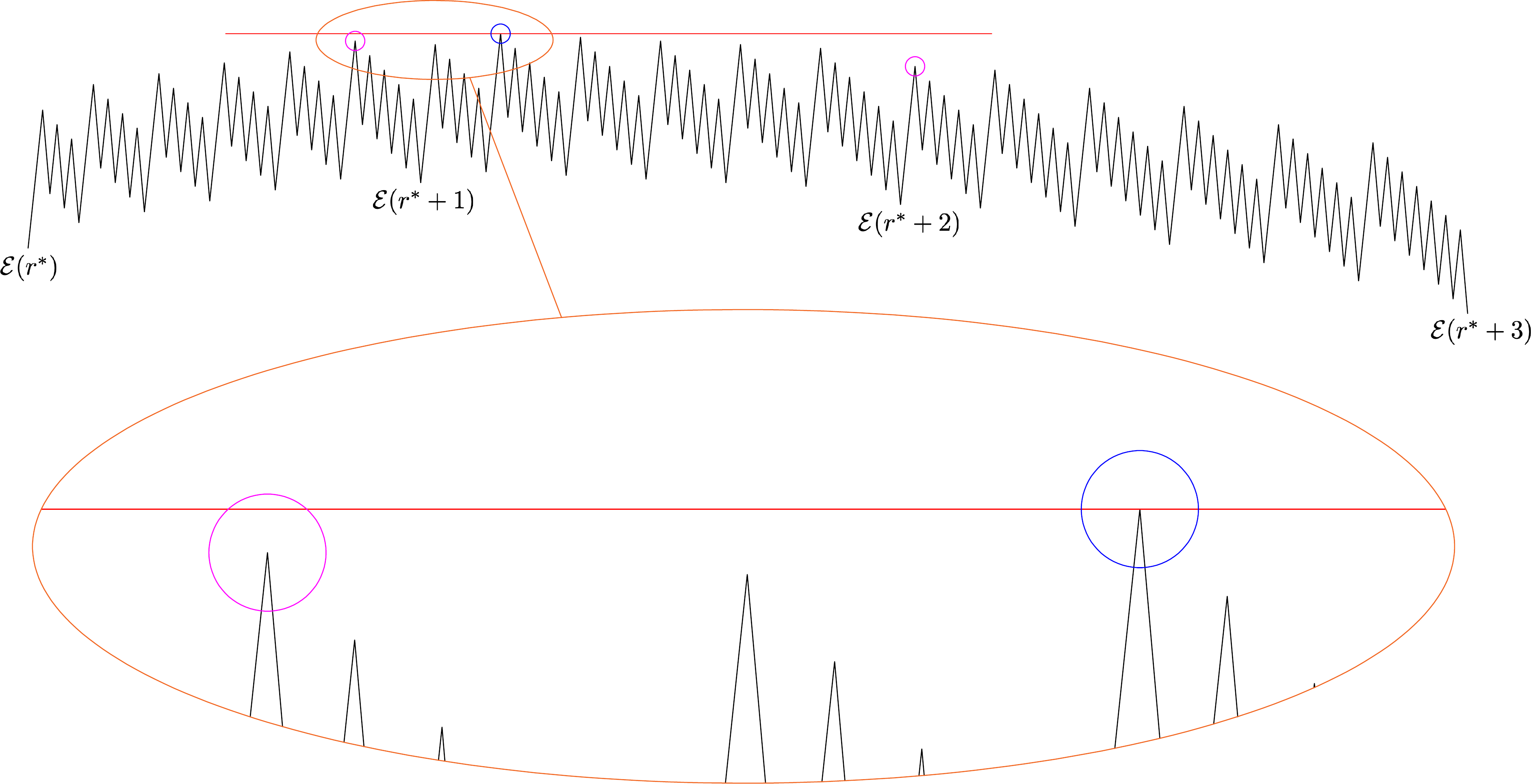} 
\caption{An example of the energy landscape between $\mathcal{E}(r^*)$ and $\mathcal{E}(r^*+3)$ for the values of the external magnetic field $h=1$, the ferromagnetic interaction $J=21/2$, thus $\delta \in (1/2,1)$. We zoom in some part of the energy landscape, in order to compare the saddles and we highlight the maximal saddle in blue.} \label{cadPhi2}
\end{figure}
\begin{proof}[Proof of Corollary~\ref{Phimax}] We compare $\Phi_{\omega^*}(\underline{-1},\mathcal{E}(r^*+1))$, $\Phi_{\omega^*}(\mathcal{E}(r^*+1), \mathcal{E}(r^*+2))$ and $\Phi_{\omega^*}(\mathcal{E}(r^*+2), \underline{+1})$. By Proposition \ref{rminore}, we have 
\begin{align}\label{Phi1}
\Phi_{\omega^*}(\underline{-1},\mathcal{E}(r^*+1)) & =  H(\mathcal{S}(A_1))-H(\underline{-1}) \notag \\
& =-6{r^*}^2h+6r^*J-10r^*h+7J-5h.
\end{align}
By Proposition \ref{ruguale}, we have 
\begin{align}\label{Phi2}
\Phi_{\omega^*}(\mathcal{E}(r^*+1), \mathcal{E}(r^*+2)) & = H(\mathcal{S}(A_3))-H(\underline{-1}) \notag \\
& =-6(r^*+1)^2h+6(r^*+1)J-2(r^*+1)h+3J-h.
\end{align}
By Proposition \ref{rmaggiore}, we have 
\begin{align}\label{Phi3} 
\Phi_{\omega^*}(\mathcal{E}(r^*+2), \underline{+1}) & = H(\mathcal{S}(A_2))-H(\underline{-1}) \notag \\
& =-6(r^*+2)^2h+6(r^*+2)J+2J-2h.
\end{align}

Recalling \eqref{raggiocritico} and comparing equations \eqref{Phi1}, \eqref{Phi2}, \eqref{Phi3}, we obtain
\begin{align}
\Phi_{\omega^*}(\underline{-1},\mathcal{E}(r^*+1)) & >\Phi_{\omega^*}(\mathcal{E}(r^*+2), \underline{+1}), \\
\Phi_{\omega^*}(\mathcal{E}(r^*+1), \mathcal{E}(r^*+2)) & >\Phi_{\omega^*}(\mathcal{E}(r^*+2), \underline{+1}).
\end{align}
Thus $\Phi_{\omega^*}(\mathcal{E}(r^*+2), \underline{+1})$ can not be the maximum. As it can be seen in Figures \ref{cadPhi1} and \ref{cadPhi2} and by standard computations, we have
\begin{align}
    \Phi_{\omega^*}(\underline{-1},\mathcal{E}(r^*+1))>  \Phi_{\omega^*}(\mathcal{E}(r^*+1), \mathcal{E}(r^*+2)), & \qquad \text{if } 0<\delta<\frac{1}{2} \\
    \Phi_{\omega^*}(\mathcal{E}(r^*+1), \mathcal{E}(r^*+2))>  \Phi_{\omega^*}(\underline{-1},\mathcal{E}(r^*+1)), & \qquad \text{if } \frac{1}{2}<\delta<1
\end{align}
So, if $0<\delta<\frac{1}{2}$, then the maximum $\Phi_{\omega^*}(\underline{-1}, \underline{+1})=\Phi_{\omega^*}(\underline{-1},\mathcal{E}(r^*+1))$ is achieved in a configuration with the standard cluster $\mathcal{S}(6{r^*}^2+10r^*+5)$. If $\frac{1}{2}<\delta<1$, then the maximum $\Phi_{\omega^*}(\underline{-1}, \underline{+1})=\Phi_{\omega^*}(\mathcal{E}(r^*+1), \mathcal{E}(r^*_2+1))$ is achieved in a configuration with the standard cluster $\mathcal{S}(6(r^*+1)^2+2(r^*+1)+1)$. 
Moreover, we observe that if $\delta=\frac{1}{2}$, then the maximum $\Phi_{\omega^*}(\underline{-1}, \underline{+1})=\Phi_{\omega^*}(\underline{-1},\mathcal{E}(r^*+1))=\Phi_{\omega^*}(\mathcal{E}(r^*+1), \mathcal{E}(r^*+2))$ is achieved in two configurations with the standard clusters $\mathcal{S}(6{r^*}^2+10r^*+5)$ and $\mathcal{S}(6(r^*+1)^2+2(r^*+1)+1)$.
Indeed, in this case, both these configurations
share the same energy.
\end{proof}

\subsection{Lower bound of maximal stability level}
Given $\sigma \in \mathcal{X}$, we recall \eqref{eq:number_of_pluses} for the number of plus spins in the configuration $\sigma$. We denote with $\mathcal{F}(Q)$ the set of ground states in $Q \subseteq \mathcal{X}$, that is
\begin{equation}\label{F}
\mathcal{F}(Q):=\{y \in Q | \min_{x \in Q} H(x) = H(y)\}.
\end{equation}
For $n$ integer, $0 \leq n \leq |\Lambda|$, we introduce the  following set
\begin{equation}\label{Vn}
    \mathcal{V}_n:=\{\sigma \in \mathcal{X} | N^+(\sigma)=n\},
\end{equation}
namely $\mathcal{V}_n$ is the set of configurations with a number of plus spins fixed at the value $n$. The number of plus corresponds to the area of the cluster.
\begin{lemma}\label{lowerbound} Assume that Condition \ref{condizionetoro} is satisfied. We have 
\begin{enumerate}
    \item Let $\sigma \in \mathcal{V}_{A_i^*}$ for $i \in \{1,2\}$, we have that the set $N^+(\sigma)$ is not a nearest neighbor connected subset of $\Lambda$ winding around the torus $\Lambda$;
    \item $\mathcal{V}_{A_i^*} \supset \mathcal{S}(A_i^*)$, for $i \in \{1,2\}$;
    \item $H(\mathcal{F}(\mathcal{V}_{A_i^*}))=H(\underline{-1})+\Gamma^{Hex}$, for $i \in \{1,2\}$.
\end{enumerate} 
\end{lemma}

\begin{proof}[Proof of Lemma~\ref{lowerbound}]
\begin{enumerate}
    \item Recalling the two cases of critical area in Corollary \ref{Phimax} and observing that $r^*<\frac{J}{2h}$, we have
    \begin{itemize}
        \item[a.]
        \begin{align}
            |N^+(\sigma)| & =6{r^*}^2+10r^*+5<6 (r^*+1)^2 < 6 \Big(\frac{J}{2h}+1 \Big)^2 \notag \\
            & \leq 6 \Big(\frac{J}{h}\Big)^2 < \Big(\frac{4J}{h}\Big)^2
        \end{align}
        where in the third inequality we have used that, by Condition \ref{condizionetoro}, $J \geq 2h$. The item follows since, by Condition \ref{condizionetoro}, we have that $|\Lambda| \geq \Big(\frac{4J}{h}\Big)^2$.
        \item[b.]
        \begin{align}
            |N^+(\sigma)| & =6(r^*+1)^2+2(r^*+1)+1<6 (r^*+2)^2 < 6 \Big(\frac{J}{2h}+2 \Big)^2 \notag \\
            & \leq 6 \Big(\frac{3J}{2h}\Big)^2 < \Big(\frac{4J}{h}\Big)^2
        \end{align}
        where in the third inequality we have used that, by Condition \ref{condizionetoro}, $J \geq 2h$. The item follows since, by Condition \ref{condizionetoro}, we have that $|\Lambda| \geq \Big(\frac{4J}{h}\Big)^2$.
    \end{itemize}
    \item By Equation \eqref{Vn} and Definition \ref{standard} of standard polyamond $S(A_i^*)$ and the corresponding cluster in $\mathcal{S}(A_i^*)$, we have $\mathcal{V}_{A_i^*} \supset \mathcal{S}(A_i^*)$, for $i \in \{1,2\}$.
    \item Let $m$ be the number of site pairs of $\Lambda$ at lattice distance one, and let $p(C)$ be the perimeter of the cluster $C$ in a configuration $\sigma(C) \in \mathcal{V}_{A^*_i}$. By Definition \ref{hamiltonianFunction} and \eqref{eq:peierls_hamiltonian} we have:
    \begin{align}
    \min_{\sigma(C) \in \mathcal{V}_{A^*_i}}H(\sigma(C)) & 
    = \min_{\sigma(C) \in \mathcal{V}_{A^*_i}} \big[H(\underline{-1})+Jp(C)-hA^*_i \big] \notag \\
   & = H(\underline{-1})-hA^*_i + J \min_{x \in \mathcal{V}_{A^*_i}} p(C) \notag \\
   & = H(\underline{-1})-hA^*_i + J p(S(A^*_i)),
    \end{align}
    where in the last equality we use Theorem \ref{thm:optimality_of_standard_polyiamonds}. Using Remark \ref{proprietaS} and the values of $A^*_i$ given in Corollary \ref{Phimax}, we compute $p(S(A^*_i))$ for $i=1,2$, obtaining 
    \begin{equation}
    -hA^*_i + J p(S(A^*_i))=
    \begin{cases}
        -6{r^*}^2h+6r^*J-10r^*h+7J-5h & \text{for } i=1\\
        -6(r^*+1)^2h+6(r^*+1)J-2(r^*+1)h+3J-h & \text{for } i=2
    \end{cases} \notag
    \end{equation}
    
    Recalling \eqref{raggiocritico}, we observe that these values correspond to the Definition \ref{GammaH}. Therefore
    \begin{equation}
        H(\mathcal{F}(\mathcal{V}_{A_i^*}))=H(\underline{-1})+\Gamma^{Hex}.
    \end{equation}
\end{enumerate}
\end{proof}

\begin{lemma}\label{PhiGammaHex}
Assume that Condition \ref{condizionetoro} is satisfied. We have that $\Phi_{\omega^*}(\underline{-1},\underline{+1})-H(\underline{-1})=\Gamma^{Hex}$.
\end{lemma}
\begin{proof}[Proof of Lemma~\ref{PhiGammaHex}]
By Corollary \ref{Phimax}, we have two values of $A^*_i$ depending on the two parameters $J,h$. We analyze two different cases:
\begin{itemize}[label=\raisebox{0.25ex}{\tiny \textbullet}]
    \item $A^*_1=6{r^*}^2+10r^*+5$. By Corollary \ref{Phimax}, we have 
    \begin{equation}
        \Phi_{\omega^*}(\underline{-1},\underline{+1})-H(\underline{-1})=
        H(\mathcal{S}(A^*_1))-H(\underline{-1})=-6{r^*}^2h+6r^*J-10r^*h+7J-5h
        \end{equation}
        and this value corresponds to that of Definition~\ref{GammaH}.
    \item $A^*_2=6(r^*+1)^2+2(r^*+1)+1$. By Corollary \ref{Phimax}, we have 
    \begin{align}
        \Phi_{\omega^*}(\underline{-1},\underline{+1})-H(\underline{-1}) & =
        H(\mathcal{S}(A^*_2))-H(\underline{-1}) \notag \\
        & =-6(r^*+1)^2h+6(r^*+1)J-2(r^*+1)h+3J-h
    \end{align}
    and this value corresponds to that of Definition~\ref{GammaH}.
\end{itemize}
\end{proof}

\section{Proofs of other theorems}\label{sec:other_theorems}
\begin{proof}[Proof of Theorem~\ref{Identification}]
In Section \ref{referencepath} we computed the value of $\Gamma$ to be $\Gamma^{Hex}$, see Definition \eqref{GammaH}. There, we also proved that
\begin{equation}
 \Phi(\underline{-1},\underline{+1})-H(\underline{-1})=\Gamma^{Hex}.
\end{equation}
Thus, the first assumption of \cite[Theorem 2.4]{CNrelax13} is satisfied for the choice of $A=\{\underline{-1}\}$ and $a=\Gamma^{Hex}$. The second assumption of \cite[Theorem 2.4]{CNrelax13} is satisfied thanks to Lemma \ref{EST}, since either $\mathcal{X}\setminus \{\underline{-1}, \underline{+1} \}=\emptyset$ or
\begin{equation}
 V_\sigma<\Gamma^{Hex} \text { for all } \sigma\in \mathcal{X}\setminus \{\underline{-1}, \underline{+1} \}.
\end{equation}
Finally, by applying \cite[Theorem 2.4]{CNrelax13}, we conclude that $\Gamma_m=\Gamma^{Hex}$ and $\mathcal{X}^m=\{\underline{-1}\}$.
\end{proof}

\begin{proof}[Proof of Theorem~\ref{teotime'}]
By Proposition~\ref{teoRP}, the assumptions of \cite[Theorem 4.15]{MNOS04} are verified taking $\eta_0=\{\underline{-1}\}$ and $T'_{\beta}=e^{\beta(V^*+\epsilon)}$. Then \eqref{Ptime'} and \eqref{Etime'} follow from \cite[Theorem 4.15]{MNOS04}. 
\end{proof}
\begin{proof}[Proof of Remark~\ref{transitiontime}] Apply \cite[Theorem 4.1]{MNOS04} with $\eta_0=\{\underline{-1}\}$ and $\Gamma=\Gamma^{Hex}$.
\end{proof}
\begin{proof}[Proof of Theorem~\ref{TMIX}]
By \cite[Proposition 3.24 and Example 3]{nardi2016hitting}, since $\Gamma_m=\Gamma^{Hex}$ thanks to Theorem \ref{Identification}, we get the result. 
\end{proof}
\begin{proof}[Proof of Theorem~\ref{selle}]
Recalling the two cases in Corollary \ref{Phimax}, we analyze for $i\in \{1,2\}$ the elements of $\mathcal{V}_{A^*_i}$ with minimal perimeter (otherwise the configuration has,
at least, energy $H(\underline{-1}) + \Gamma^{Hex} + 2h$). By Lemma \ref{lemma:lower_bound_edge_perimeter_non_standard_polyiamonds}, we have that every optimal path from $\underline{-1}$ to $\underline{+1}$ intersects the configurations with cluster of area $A^*_i-2$ and minimal perimeter
(otherwise the configuration has,
at least, energy $H(\underline{-1}) + \Gamma^{Hex} + 2h$) consisting of 
the quasi-regular hexagons $E_{B_5}(r^*)$ if $i=1$ and $E_{B_1}(r^*+1)$ if $i=2$. Adding two triangular units to these quasi-regular hexagons, we obtain an element $\sigma(C^*)$ of $\mathcal{V}_{A^*_i}$, see Figure \ref{addtriangle1}. $\sigma(C^*)$ is a configuration with cluster $C^*$ which is composed by a quasi-regular hexagonal cluster and two triangular units attached to it according to one of the following cases:
\begin{enumerate}[label=(\arabic*)]
    \item the two triangular units form an elementary rhombus which is attached to one of the longest sides of the quasi-regular hexagonal cluster, $\mathcal{\tilde S}(A^*_i)$;
    \item the two triangular units are attached to one of the longest sides of the quasi-regular hexagonal cluster at triangular lattice distance 2, $\mathcal{\tilde D}(A^*_i)$;
    \item the two triangular units are attached to the same side of the quasi-regular hexagonal cluster at triangular lattice distance grater than 2;
    \item the two triangular units are attached to two different sides of the quasi-regular hexagonal cluster;
    \item the two triangular units form an elementary rhombus which is attached to one of the sides, other than the longest, of the quasi-regular hexagonal cluster;
    \item the two triangular units are attached at triangular lattice distance 2 to the same side, other than the longest, of the quasi-regular hexagonal cluster.
\end{enumerate}
We note that in all the above cases $C^*$ has minimal perimeter, since $C^*$ has the same perimeter of a standard hexagon with the same area. In the remainder we analyze the previous cases from a dynamical point of view. 
We will show that all optimal paths must go through a
configuration as in case (1) or (2), whereas optimal paths
visiting configurations as those of cases (3), (4), (5), and (6)
(\emph{dead ends}) must go back to configurations as those of the first two cases before reaching $\underline{+1}$.

Let $\omega_0$ be the configuration that contains the quasi-regular hexagon of area $A^*_i-2$. $\omega_1$ is obtained from $\omega_0$ flipping a minus spin
adjacent to the cluster and $\omega_2$ is obtained from $\omega_1$ flipping another minus spin to reach a configuration as those of the
cases described above. By Table~\ref{Tab}, we have
\begin{equation}
    H(\omega_2)-H(\omega_0)=[H(\omega_2)-H(\omega_1)]+[H(\omega_1)-H(\omega_0)]=2J-2h.
\end{equation}
We observe that $H(\omega_2)-H(\underline{-1})=H(\sigma(C^*))-H(\underline{-1})$ is equal to $\Gamma^{Hex}$ by Corollary \ref{Phimax}, Lemma \ref{lowerbound} and Lemma \ref{PhiGammaHex}. Next, we consider a configuration $\omega_3$ obtained from $\omega_2$ flipping another minus spin. We observe that the perimeter of the cluster decreases only if we add a plus spin in a site belonging to an internal angle of $\frac{5}{3}\pi$ (of the cluster). This spin flip can be done only in the cases (1) and (2) above, see Figure \ref{addtriangle1}.
\begin{figure}[htb!]
\centering
    \includegraphics[width=\textwidth]{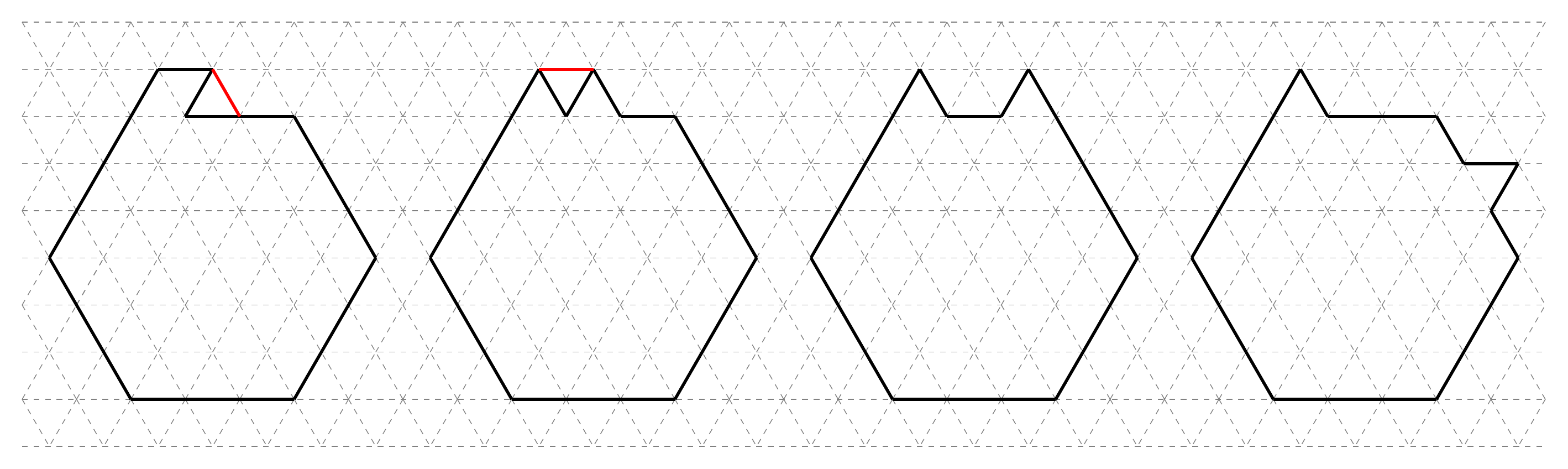}
    \caption{From left to right we depict configurations in cases (1),(2),(3),(4)  where two triangular units are added to the quasi-regular hexagon. In red the third triangular unit that can be added only in cases (1) and (2) decreasing the energy.}
  \label{addtriangle1}
    \end{figure}
In these two cases, we obtain $\omega_3$ adding the third triangular unit to cover the internal angle of $\frac{5}{3}\pi$. Therefore, the energy of the system is lowered by $J+h$, see Table~\ref{Tab}:
\begin{equation}
    H(\omega_3)-H(\omega_2)=-(J+h).
\end{equation}
Thus, we have
\begin{equation}
    H(\omega_3)-H(\underline{-1})=[H(\omega_3)-H(\omega_2)]+ [H(\omega_2)-H(\underline{-1})]=-(J+h)+\Gamma^{Hex}<\Gamma^{Hex}.
\end{equation}
In cases (3) and (4) the cluster does not have an angle of $\frac{5}{3}\pi$. This implies that when we add another triangular unit, the energy of the system can only increase. Thus, we have
\begin{equation}
    H(\omega_3)-H(\omega_2) \ge J-h,
\end{equation}
and
\begin{equation}\label{nosaddle}
    H(\omega_3)-H(\underline{-1})=[H(\omega_3)-H(\omega_2)]+ [H(\omega_2)-H(\underline{-1})] \ge J-h+\Gamma^{Hex}>\Gamma^{Hex}.
\end{equation}
To rule out cases (5) and (6), we show that the two triangular units have to be attached along one of the longest sides of the quasi-regular hexagon. In particular, recalling Corollary \ref{Phimax} and Definition \ref{areacritica}, if $\delta \in (0,\frac{1}{2})$ we 
must attach the two triangular units along the longest side of $E_{B_5}(r^*)$. 
Indeed, recalling Definition \ref{def:quasi_reg_hex}, the longest side of $E_{B_5}(r^*)$ has length $r^*+2$ and it has
the same length of the larger base of the bar $B_6$, that has cardinality $2r^*+3$. Any other side $s$ of $E_{B_5}(r^*)$ has length $r^*+1$, so the bar $B$ with the larger base $r^*+1$ has cardinality $l=2r^*-1$. 
Suppose by contradiction that there exits a path 
$\tilde \omega:=(\omega_0, \ldots, \omega_2, \ldots, \omega_l, \ldots, \omega_{l+2}, \ldots, \omega_{2r^*+3})$ 
that intersects the configurations described in case (5) and (6) with $\Phi_{\tilde \omega} \leq H(\underline{-1})+\Gamma^{Hex}$. 
Let $\omega_0:=\mathcal{E}_{B_5}(r^*)$, $\omega_2$ be the configuration with the two triangular units attached along $s$, 
$\omega_3,\ldots,\omega_l$ be the configurations obtained filling the new bar $B$ and let $\omega_{l+2}$ be the configuration obtained from 
$\omega_l$ by attaching two triangular units. 
Recalling \eqref{valoridiff} and \eqref{raggiocritico}, we have the following contradiction 
\begin{align}
    H(\omega_{l+2})-H(\underline{-1}) & =[H(\omega_{l+2})-H(\omega_l)]+[H(\omega_l)-H(\omega_2)]+[H(\omega_2)-H(\underline{-1})] \notag \\ 
    & =[2J-2h]+[-J-(l-2)h]+\Gamma^{Hex}=J-lh+\Gamma^{Hex}= \notag \\
    & =J-(2r^*-1)h+\Gamma^{Hex}=2h(\delta+1)+\Gamma^{Hex}>\Gamma^{Hex}.
    \end{align}
Analogously, if $\delta \in (\frac{1}{2},1)$ 
we must attach the two triangular units along one of the longest sides of $E_{B_1}(r^*+1)$. 
Indeed, recalling Definition \ref{def:quasi_reg_hex}, there exist two longest sides of $E_{B_1}(r^*+1)$ with length $r^*+2$ and each of these sides has the same length of the larger base of a bar with cardinality $2(r^*+1)+1$. The other sides $s$ of $E_{B_1}(r^*+1)$ have length $r^*+1$ and the 
corresponding bars have cardinality $2(r^*+1)-1$. 
Suppose by contradiction that there exits a path $\tilde \omega:=(\omega_0, \ldots, \omega_2, \ldots, \omega_l, \ldots, \omega_{l+2}, \ldots, \omega_{2r^*+3})$ that intersects the configurations described in case (5) and (6) with $\Phi_{\tilde \omega} \leq H(\underline{-1})+\Gamma^{Hex}$. Let $\omega_0:=\mathcal{E}_{B_1}(r^*+1)$, $\omega_2$ be the configuration with the two triangular units attached along $s$, $\omega_3,\ldots,\omega_l$ be the configurations obtained filling the new bar $B$ and let $\omega_{l+2}$ be the configuration obtained from $\omega_l$ by attaching two triangular units. Recalling \eqref{valoridiff} and \eqref{raggiocritico}, we have the following contradiction
\begin{align}
    H(\omega_{l+2})-H(\underline{-1}) & =[H(\omega_{l+2})-H(\omega_l)]+[H(\omega_l)-H(\omega_2)]+[H(\omega_2)-H(\underline{-1})] \notag \\ 
    & =[2J-2h]+[-J-(l-2)h]+\Gamma^{Hex}=J-lh+\Gamma^{Hex}= \notag \\
    & =J-(2(r^*+1)-1)h+\Gamma^{Hex}=2h\delta+\Gamma^{Hex}>\Gamma^{Hex}.
\end{align}
Therefore, the two triangular units have to be attached along one of the longest sides of the quasi-regular hexagon as in case (1) and (2). 
Recalling Definition \ref{gates} and \ref{minimalgate}, let $\mathcal{W}(\underline{-1},\underline{+1})$ be a minimal gate for the transition from $\underline{-1}$ to $\underline{+1}$. By the previous analysis, we have that 
all configurations as those in the cases (3), (4), (5), (6) 
are not in $\mathcal{W}(\underline{-1},\underline{+1})$. 
Moreover, observing that configurations as those of case (1) correspond to the configurations $\mathcal{\tilde S}(A^*_i)$ and configurations as those of case (2) correspond to the configurations $\mathcal{\tilde D}(A^*_i)$ (see Figure \ref{figselle}), we conclude $\mathcal{W}(\underline{-1},\underline{+1})=\mathcal{\tilde S}(A^*_i) \cup \mathcal{\tilde D}(A^*_i)$. 
\end{proof}

\subsection{Proof of Theorem~\ref{thm:sharp_estimate}}
To prove Theorem~\ref{thm:sharp_estimate}, we estimate the \emph{capacity} between $\underline{-1}$ and $\underline{+1}$ which is linked to the mean hitting time
of the stable state through the following 
formula (see \cite[Corollary 7.11]{bovier2016metastability}):
\begin{equation}
    \mathbb{E}_{\underline{-1}} [\tau_{\underline{+1}}] =
    \frac{1}{\text{CAP}(\underline{-1},\underline{+1})} \sum_{\sigma \in \mathcal{X}}\mu(\sigma)h_{\underline{-1},\underline{+1}}(\sigma).
\end{equation}
For a detailed discussion of the strategy outlined below to
estimate the capacity refer to \cite{bovier2016metastability}.
\paragraph{Capacity as the minimum of the Dirichlet form.}
Let $h: {\mathcal{X}} \to \mathbb{R}$ and consider
the following \emph{Dirichlet form}
\begin{align}\label{dirichlet}
\mathfrak{E}(h) & =\frac{1}{2}\sum_{\sigma,\eta\in\mathcal{X}} \mu(\sigma) p(\sigma, \eta) {[h(\sigma)-h(\eta)]}^2 \\
& =\frac{1}{2}\sum_{\sigma,\eta\in\mathcal{X}} \frac{e^{-\beta H(\sigma)}}{Z} \frac{e^{-\beta [H(\eta)-H(\sigma)]_+}}{|\Lambda|}{[h(\sigma)-h(\eta)]}^2
\end{align}
where $Z$ is the partition function $Z:=\sum_{\eta \in \mathcal{X}} e^{-\beta H(\eta)}$ .

Given two non-empty disjoint sets $A, \, B$ the \textit{capacity} of the pair $A, \, B$ is defined by
\begin{equation}\label{capacity}
   \text{CAP}(A,B):= \min_{\substack{h:\mathcal{X}\to[0,1] \\ {h_{|A}=1,h_{|B}=0}}} \mathfrak{E}(h).
\end{equation}
From this definition it follows immediately that the capacity is a \textit{symmetric} function of the sets $A$ and $B$.\\
The right hand side of (\ref{capacity}) has a unique minimizer $h_{A,B}^*$ called \textit{equilibrium potential} of the pair $A, \, B$ given  by
\begin{equation}
h_{A,B}^*(\eta)=\mathbb{P}_\eta(\tau_A<\tau_B),
\end{equation}
for any $\eta\in A\cup B$.

Hence, inserting a general test function $h$ in the Dirichlet form, one obtains an \emph{upper bound} for the capacity. Obviously, the closer $h$ is to the equilibrium potential, the sharper is the bound.

\paragraph{Capacity as the maximum of the expectation of a flow dependent variable.}
A remarkable property of capacity is that it can be characterized also by  another variational principle, useful 
to obtain \emph{a lower bound}. Think to $\mathcal{X}$ as the vertex set of a graph $(\mathcal{X},\mathcal{L})$ whose edge set $\mathcal{L}$ consists of all pairs $(\sigma,\eta)$ with $\sigma,\eta \in \mathcal{X}$ for which $P(\sigma,\eta)>0$ (see also \cite{nardi2012sharp, bovier2010homogeneous, bianchi2009sharp} for further details and applications
to several models). %\mathcal{L} dopo lo abbiamo chiamato E
\begin{definition}
\label{unit-flow}
Given two non-empty disjoint sets $A,B\subset \mathcal{X}$, a loop-free non-negative unit
flow, $f$, from $A$ to $B$ is a function $f\colon\,E \rightarrow [0,\infty)$ such
that:
\begin{itemize}
\item[(a)] $(f(e)>0 \Longrightarrow f(-e)=0)$ $\forall\,e\in E$.
\item[(b)] $f$ satisfies Kirchoff's law:
\begin{equation}\label{ab.1}
\sum_{\sigma'\in\mathcal{X}} f(\sigma,\sigma') = \sum_{\sigma''\in\mathcal{X}} f(\sigma'',\sigma),
\qquad \forall\,\sigma\in \mathcal{X}\backslash (A\cup B).
\end{equation}
\item[(c)] $f$ is normalized:
\begin{equation}\label{normalization}
\sum_{\sigma\in A} \sum_{\sigma'\in\mathcal{X}} f (\sigma,\sigma')
= 1 = \sum_{\sigma\in\mathcal{X}} \sum_{\sigma \in B} f(\sigma'',\sigma ).
\end{equation}
\item[(d)] Any path from $A$ to $B$ along edges $e$ such that $f(e)>0$ is self-avoiding, that is, visits each configuration at most once.
\end{itemize}
The space of all loop-free non-negative unit flows from $A$ to $B$ is denoted by
$\mathbb{U}_{A,B}$.
\end{definition}

A loop-free non-negative unit flow $f$ is naturally associated with a probability measure $\mathbb{P}^f$ on self-avoiding paths. To see this, define $F(\sigma) = \sum_{\sigma'\in\mathcal{X}}
f(\sigma,\sigma')$, $\sigma\in\mathcal{X} \setminus B$. Then $\mathbb{P}^f$ is the Markov chain $(\sigma_n)_{n\in\mathbb{N}_0}$
with initial distribution $\mathbb{P}^f (\sigma_0) = F(\sigma_0)\mathbb{1}_A(\sigma_0)$, transition probabilities
\begin{equation}
\label{ab.3}
q^f (\sigma,\sigma' )= \frac{f(\sigma,\sigma')}{F(\sigma)}, \qquad \sigma \in \mathcal{X} \setminus B,
\end{equation}
such that the chain is stopped upon arrival in $B$. In terms of this probability measure, we have the following proposition:
\begin{proposition}[Berman–Konsowa principle: flow version] \label{BermanKonsowa}
Let $A,B \subset \mathcal{X}$ be two non-empty disjoint sets. Then, with the notation introduced
above and denoting by $\gamma$ a self-avoiding path from $A$ to $B$,
\begin{equation}\label{ab.5}
\text{CAP}(A,B)=\sup_{f\in \mathbb{U}_{A,B}}\mathbb{E}^f\left(\left[
\sum_{e\in\gamma}\frac{f(e_l,e_r)}{\mu(e_l)\,p(e_l,e_r)}\right]^{-1}\right),
\end{equation}
where $e_l$ and $e_r$ are the endpoints
of edge $e$ and the expectation is taken with respect to the measure $\mathbb{P}^f$ on self-avoiding paths.
\end{proposition}

Thanks to this variational principle, any flow provides a \emph{computable lower bound} for the capacity.

\paragraph{Upper bound.}
Consider the following sets:
\begin{itemize}[label=\raisebox{0.25ex}{\tiny \textbullet}]
    \item $\mathcal{X}^* \subset \mathcal{X}$ defined as 
          the subgraph obtained by removing all vertices $\sigma$ 
          with $\Phi(\underline{-1}, \sigma) > \Gamma^{Hex}+H(\underline{-1})$ or
          with $\Phi(\underline{+1}, \sigma) > \Gamma^{Hex}+H(\underline{-1})$
          together with all edges incident to these vertices (note that, in particular,
          $\mathcal{X}^*$ does not contain vertices with 
          $H(\sigma) > \Gamma^{Hex}+H(\underline{-1})$);
          and all edges incident to these vertices;
    \item $A:=\{\sigma \in \mathcal{X} \, | \, 
                 \Phi(\underline{-1},\sigma) < \Gamma^{Hex}+H(\underline{-1}) \} 
             \subset \mathcal{X}^{*}$;
    \item $B:=\{\sigma \in \mathcal{X} \, | \,
                  \Phi(\sigma,\underline{+1}) < \Gamma^{Hex}+H(\underline{-1}) \}
              \subset \mathcal{X}^{*}$;
    \item $\mathcal{G}_i \subset \mathcal{X}^*, i=1, \ldots, I$ a collection of sets
          such that, for all $i$, $\Phi(\underline{-1}, \sigma) = \Phi(\underline{+1}, \sigma)$ and $\sigma \sim \eta$ and $\sigma \sim \eta'$, for all $\eta \in A$, $\eta' \in B$
          for all $\sigma \in \mathcal{G}_i$;
     \item $\mathcal{N}_j^A \subset \mathcal{X}^*, j=1, \ldots, J_A$ a collection of sets
          such that, for all $j$ and for all $\sigma \in \mathcal{N}_j^A$, 
          $\Phi(\underline{-1}, \sigma) 
            = \Phi(\underline{+1}, \sigma)$
          and any path $\omega: \sigma \to \underline{+1}$ must be such that $\omega \cap A \neq \emptyset$;
    \item $\mathcal{N}_j^B \subset \mathcal{X}^*, j=1, \ldots, J_B$ a collection of sets
          such that, for all $j$ and for all $\sigma \in \mathcal{N}_j^B$, 
          $\Phi(\underline{+1}, \sigma) 
            = \Phi(\underline{-1}, \sigma)$
          and any path $\omega: \sigma \to \underline{-1}$ must be such that $\omega \cap B \neq \emptyset$.
\end{itemize}

In words, sets $A$ and $B$ are the cycles with stability level $\Gamma^{Hex}$
around $\underline{-1}$ and $\underline{+1}$ respectively.
Further sets $\mathcal{G}_{i}$ consists of those configurations belonging to
some minimal gate whereas $\mathcal{N}_j^A$ and $\mathcal{N}_j^B$ are
\emph{dead ends}, that is, to achieve a transition between $\underline{-1}$ and $\underline{+1}$ starting from one of these sets the dynamics
must go back to $A$ or $B$.

Note that sets $\mathcal{G}_i$, $\mathcal{N}_j^A$ and $\mathcal{N}_j^B$ are
not connected via \emph{allowed moves}.

By standard arguments (see \cite[Chapter 16]{bovier2016metastability}) it is straightforward to check that
the sum over $\sigma, \eta \in \mathcal{X}$ in \eqref{dirichlet} can be substituted by 
a sum over $\mathcal{X}^*$ at the price of a factor $[1+o(1)]$.

To find an upper bound for the capacity we choose a test function $h^U$ 
defined in the following way:
\begin{equation}
    h^U(\sigma):=
    \begin{cases}
    1 & \qquad \sigma \in A \cup \{ \bigcup_{j=1}^{J_A}\mathcal{N}_j^A \}  \\
    %&\\
    c_i & \qquad \sigma \in \mathcal{G}_i, \, i=1,\ldots, I \\
    %&\\
    0 & \qquad \sigma \in B \cup \{ \bigcup_{j=1}^{J_B}\mathcal{N}_j^B \}
\end{cases}
\end{equation}

It is convenient to call 
$\mathcal{X}_A := A \cup
    \{\bigcup_{j=1}^{J_A}\mathcal{N}_j^A\}$, 
$\mathcal{X}_B := B \cup 
    \{ \bigcup_{j=1}^{J_B}\mathcal{N}_j^B \}$.

Consider the case $\delta \in (0, \frac{1}{2})$. 
By Theorem~\ref{selle}, it follows that $I = 2$
and $\mathcal{G}_1 \equiv \mathcal{\tilde{S}}(A^*_1)$
and $\mathcal{G}_2 \equiv \mathcal{\tilde{D}}(A^*_1)$. Define $(\mathcal{C}^*)^-$ as the set of configurations where the cluster has the shape of $E_{B_5}(r^*)$ with attached a triangular unit along the longest side and $(\mathcal{C}^*)^+$ as the set of configurations where the cluster has the shape of $E_{B_5}(r^*)$ with an incomplete bar of cardinality $3$ along the longest side. 
Note that $(\mathcal{C}^*)^- \subset A$ and $(\mathcal{C}^*)^+ \subset B$.
\begin{align}
    \text{CAP}(A,B) & \leq 
        (1 + o(1)) \Big[
            \min_{c_1, c_2 \in [0,1]} 
            \min_{
                \substack{h:\mathcal{X}^*\to [0,1] \\ 
                {h_{|\mathcal{X}_A}=1, \, h_{|\mathcal{X}_B}=0}, \\ {h_{|\mathcal{G}_1}=c_1}, \, h_{|\mathcal{G}_2}=c_2}
            }  
            \frac{1}{2}\sum_{\sigma, \eta \in \mathcal{X}^*} 
                \mu(\sigma)P(\sigma,\eta) [h(\sigma)-h(\eta)]^2
        \Big] \notag \\
    & = (1 + o(1))\Big[
        \min_{c_1,c_2 \in [0,1]} 
        \sum_{\substack{\sigma \in \mathcal{X}_A, \\ \eta \in \mathcal{G}_i, \, i=1,2 \\ \sigma \sim \eta}} \frac{e^{-\beta H(\sigma)}}{Z|\Lambda|} (1-c_i)^2 +
        \sum_{\substack{\sigma \in \mathcal{X}_B, \\ \eta \in \mathcal{G}_i, \, i=1,2\\ \sigma \sim \eta}} \frac{e^{-\beta H(\sigma)}}{Z|\Lambda|} c_i^2 \Big] \notag \\
    & = (1 + o(1))\Big[ \frac{e^{-\beta \Gamma^{Hex}}}{Z|\Lambda|} \min_{c_1, c_2 \in [0,1]} \sum_{\substack{\sigma \in (\mathcal{C}^*)^-, \\ \eta \in \mathcal{G}_i, \, i=1,2}}  (1-c_i)^2 + \sum_{\substack{\sigma \in (\mathcal{C}^*)^+, \\ \eta \in \mathcal{G}_i, \, i=1,2 }} c_i^2 \Big] \notag \\
    & = (1 + o(1))\Big[ \frac{e^{-\beta \Gamma^{Hex}}}{Z|\Lambda|} \min_{c_1, c_2 \in [0,1]} \sum_{\substack{\eta \in \mathcal{G}_i, \\ i=1,2}}  | (\mathcal{C}^*)^- \sim \eta| (1-c_i)^2 + \sum_{\substack{\eta \in \mathcal{G}_i, \\ i=1,2}} |(\mathcal{C}^*)^+ \sim \eta| c_i^2 \Big]
\end{align}

where we used the definition of the Gibbs measure~\eqref{def:gibbs} 
and the expression of the transition probability~\eqref{transitionprob} and we observed that
$\sigma \in \mathcal{X}_A$ and $\eta \in \mathcal{G}_i$
are neighbors only if $\sigma \in (\mathcal{C}^*)^-$
and 
$\sigma \in \mathcal{X}_B$ and $\eta \in \mathcal{G}_i$
are neighbors only if $\sigma \in (\mathcal{C}^*)^+$.

For all $\eta \in \mathcal{G}_1$ we have that $| (\mathcal{C}^*)^- \sim \eta|=1$, whereas for all $\eta \in \mathcal{G}_2$ we have that $| (\mathcal{C}^*)^- \sim \eta|=2$.
Moreover, we have $|(\mathcal{C}^*)^+ \sim \eta|=1$ for all $\eta \in \mathcal{G}_1 \cup \mathcal{G}_2$. Therefore, we obtain 
\begin{align}
    \text{CAP}(A,B) & \leq (1 + o(1))\Big[\frac{e^{-\beta \Gamma^{Hex}}}{Z|\Lambda|} \min_{c_1,c_2 \in [0,1]} \sum_{\eta \in \mathcal{G}_1} (1-c_1)^2+ c_1^2 + \sum_{\eta \in \mathcal{G}_2} 2(1-c_2)^2+c_2^2 \Big] \notag \\
    & \leq  (1 + o(1))\Big[\frac{e^{-\beta \Gamma^{Hex}}}{Z|\Lambda|} \min_{c_1,c_2 \in [0,1]} |\mathcal{G}_1| (2c_1^2-2c_1+1) + |\mathcal{G}_2| (3c_2^2-4c_2+2) \Big]
\end{align}
The functions $f(c_1):=2c_1^2-2c_1+1$ and $f(c_2):=3c_2^2-4c_2+2$ are minimized respectively at $c_1=\frac{1}{2}$ and $c_2=\frac{2}{3}$. 
Moreover, $|\mathcal{G}_1|=6(l-1)|\Lambda|$ and $|\mathcal{G}_2|=3(l-1)|\Lambda|$ where $l$ is the length of the longest side of $E_{B_5}(r^*)$, i.e., $l=r^*+2$.
Finally, we have

\begin{align}\label{eq:upper_bound_capacity_small_delta}
    \text{CAP}(A,B) & \leq 
    (1+o(1)) \Big[ \frac{e^{-\beta \Gamma^{Hex}}}{Z|\Lambda|} (6(l-1)|\Lambda| \frac{1}{2} + 3(l-1)|\Lambda| \frac{2}{3}) \Big] \notag \\
    & =(1+o(1)) \Big[5(l-1)\frac{e^{-\beta \Gamma^{Hex}}}{Z} \Big].
\end{align}

The case $\delta \in (\frac{1}{2}, 1)$ can be treated likewise
with the following modifications:
the critical area $A^*_1$ becomes $A^*_2$; 
$E_{B_5}(r^*)$ is replaced by $E_{B_1}(r^*+1)$;
$|\mathcal{G}_1|=12(l-1)|\Lambda|$ and $|\mathcal{G}_2|=6(l-1)|\Lambda|$ where $l$ is the length of the longest side of $E_{B_1}(r^*+1)$, i.e., again $l=r^*+2$.
We these changes we get
\begin{align}\label{eq:upper_bound_capacity_big_delta}
    \text{CAP}(A,B) \leq (1+o(1)) \Big[ 10(l-1)\frac{e^{-\beta \Gamma^{Hex}}}{Z} \Big].
\end{align}

\paragraph{Lower bound.} 
By the symmetry of the capacity, $\text{CAP}(\underline{-1},\underline{+1})=\text{CAP}(\underline{+1},\underline{-1})$. We consider a loop-free unitary flow from $\underline{+1}$ to $\underline{-1}$ to estimate $\text{CAP}(\underline{+1},\underline{-1})$. 

Let $(\mathcal{C}^*)^+$, respectively $(\mathcal{C}^*)^-$, be the set of configurations that contain a cluster with the shape of a quasi-regular hexagon $E_{B_5}(r^*)$ with an incomplete bar of cardinality 3, respectively 1, attached along its longest side. 
If $\delta \in (0,1/2)$, choose the unitary flow $f$ as follows.
Distribute the mass of the test flow equally among a suitable subset of optimal paths. 
Consider a configuration $\sigma \in (\mathcal{C}^*)^+$ and observe that the triangular units in
the incomplete bar of $\sigma$ are of two types: the first type 
consists of those triangular units at lattice distance one from $E_{B_5}(r^*)$ (the two extreme of the 
incomplete bar)
whereas the second type consists of those triangular units at lattice distance two from $E_{B_5}(r^*)$
(the central triangular unit of the incomplete bar).
Each $\sigma \in (\mathcal{C}^*)^+$ can be 
univocally identified by two coordinates $x,y$: $x$ is the coordinate of the center of $E(r^*+1)$ and $y$ is the coordinate of the 
first triangular unit of the incomplete bar.
Further, for each $\sigma \in (\mathcal{C}^*)^+$, there exists an optimal path from $\sigma$ to 
$\underline{+1}$. In this way it is possible to find
a bijection between the set of configurations $(\mathcal{C}^*)^+$ and
the set of paths from $\sigma$ to $\underline{+1}$ for 
$\sigma \in (\mathcal{C}^*)^+$. 
Call $\gamma_{x,y}$ the time reversal of the deterministic path 
from the configuration $\sigma$ (identified by $x,y$) to $\underline{+1}$.
Note that from $(\mathcal{C}^*)^+$ the path $\gamma_{x,y}$ 
can be extended
towards $\underline{-1}$ by flipping one of the spins of the incomplete bar with probability $1/3$. If one of the two spins of the first type is flipped, creating an elementary rhombus, then the path is extended deterministically by
flipping the remaining spins from plus to minus starting from the spin belonging to the second type. Otherwise, if the flipped spin
is that of the second type, then the path can be extended by flipping one of the spins of the first type with probability $1/2$, and then
proceeding deterministically to flip the remaining spins 
from plus to minus. \\
Let $K$ be the number of negative spins in configurations
in $(\mathcal{C}^*)^+$, and let $\nu_0:=\frac{1}{3|\Lambda|(l-1)}$ be the number of possible configurations in $(\mathcal{C}^*)^+$.
Consider the following unitary flow from $\underline{+1}$ to $\underline{-1}$, see Figure \ref{fig:flow}.
\begin{align}\label{flow}
&f(\sigma',\sigma'')
&=
\begin{cases}
\nu_0 
& \text{if}\,\,\sigma'=\gamma_{x,y}(k),\,\sigma''=\gamma_{x}(k+1) \\
& \text{for some }\, x,y\in \Lambda, \, 0\leq k\leq K-1 \\
\frac{\nu_0}{3}& 
\text{if}\,\,\sigma'=\gamma_{x,y}(K),\,\sigma''=\gamma_{x,y}(K+1)\in\mathcal{\tilde{D}}(A_1^*)  \\
&\text{for some }\,x,y\in \Lambda \\
\frac{\nu_0}{3}
& \text{if}\,\,\sigma'=\gamma_{x,y}(K),\,\sigma''=\gamma_{x,y}(K+1)\in\mathcal{\tilde{S}}(A_1^*)  \\
&\text{for some }\,x,y\in \Lambda \\
\frac{\nu_0}{6}& 
\text{if}\,\,\sigma'=\gamma_{x,y}(K+1)\in\mathcal{\tilde{D}}(A_1^*) ,\,\sigma''=\gamma^i_{x,y}(K+2) \\
&\text{for some }\,x,y\in \Lambda \\
\frac{\nu_0}{3}& 
\text{if}\,\,\sigma'=\gamma_{x,y}(K+1)\in\mathcal{\tilde{S}}(A_1^*) ,\,\sigma''=\gamma^i_{x,y}(K+2) \\
&\text{for some }\,x,y\in \Lambda \\
\frac{\nu_0}{2}& 
\text{if}\,\,\sigma'=\gamma_{x,y}(K+2)\in\mathcal{\tilde{D}}(A_1^*) ,\,\sigma''=\gamma_{x,y}(K+3) \\
&\text{for some }\,x,y\in \Lambda \\
\nu_0 
& \text{if}\,\,\sigma'=\gamma_{x,y}(k),\,\sigma''=\gamma_{x,y}(k+1)\\
& \text{for some }\,x,y\in \Lambda,\,K+3\leq k\leq |\Lambda|-1 \\
0 &\text{otherwise}.
\end{cases}
\end{align}
where $\gamma_{x,y}(k)$ is the $k$-th configuration visited by the path $\gamma_{x,y}$.

    \begin{figure}[!htb]
        \centering
        \includegraphics[scale=1.5]{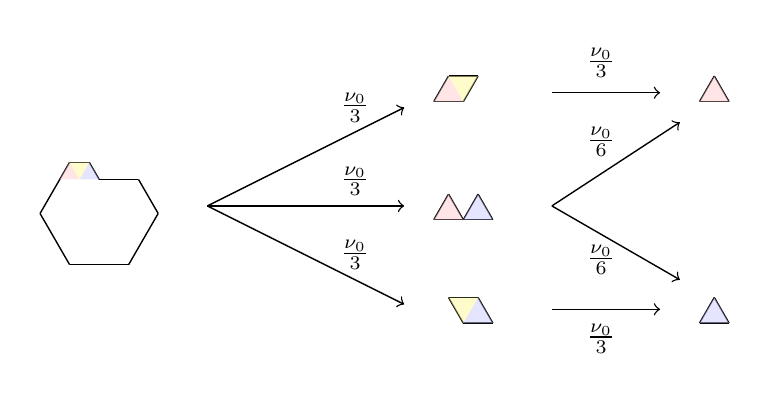}
        \caption{The chosen flow from $(\mathcal{C}^*)^+$ to $(\mathcal{C}^*)^-$. In particular, starting from the top, the second flow crosses a configuration in $\mathcal{\tilde{D}}(A_1^*)$, the first and the last one cross a configuration in $\mathcal{\tilde{S}}(A_1^*)$. The color red and the blue represent the left and the right triangular units of the first type respectively.}
        \label{fig:flow}
    \end{figure}

The flow described above can be used to assign to each path 
$\gamma=(\gamma(0),...,\gamma(|\Lambda|))$ with 
$\gamma(0)=\underline{+1}$ and $\gamma(|\Lambda|)=\underline{-1}$ 
a probability $\mathbb{P}(\gamma)$ defined as
\begin{equation}\label{eq:path_flow_probability}
    \mathbb{P}(\gamma):=\mathbb{P}^f(\mathbb{X}=\gamma)=
    \frac{\prod_{i=1}^{|\Lambda|}f(\gamma(i-1),\gamma(i))}{\prod_{i=1}^{|\Lambda|}F(\gamma(i-1))}.
\end{equation}
Non null probability paths
$\underline{+1}\to\underline{-1}$
can be partitioned into two sets
$\mathcal{I}_{\tilde S}$ and $\mathcal{I}_{\tilde D}$.
$\mathcal{I}_{\tilde S}$ contains those paths
such that $\gamma(K+1)$ is in $\mathcal{\tilde{S}}(A_1^*)$
whereas $\mathcal{I}_{\tilde S}$ contains those paths such that
$\gamma(K+1)$ is in $\mathcal{\tilde{D}}(A_1^*)$.

By Proposition \ref{BermanKonsowa} and by the choice of the flow \eqref{flow}, we have
\begin{align}\label{lb2}
& \text{CAP}(\underline{+1},\underline{-1}) \geq 
 \sum_{\gamma\in  \mathcal{I}_{\tilde S} \cup \mathcal{I}_{\tilde D}} \mathbb{P}(\gamma) \left[ \sum_{k=0}^{|\Lambda|-1}
\frac{f(\gamma(k),\gamma(k+1))}{\mu(\gamma(k))p(\gamma(k),\gamma(k+1))}
\right]^{-1} \notag \\
& \geq \sum_{\gamma\in  \mathcal{I}_{\tilde S}} \mathbb{P}(\gamma) 
\frac{e^{-\beta\Gamma^{Hex}}}{Z|\Lambda|}
\left[ f(\gamma(K),\gamma(K+1))+f(\gamma(K+1),\gamma(K+2))
\right]^{-1}(1+o(1)) \notag \\
& + \sum_{\gamma\in  \mathcal{I}_{\tilde D}} \mathbb{P}(\gamma) 
\frac{e^{-\beta\Gamma^{Hex}}}{Z|\Lambda|}
\left[ f(\gamma(K),\gamma(K+1))+f(\gamma(K+1),\gamma(K+2))
\right]^{-1}(1+o(1)) \notag \\
& \geq \sum_{\gamma\in  \mathcal{I}_{\tilde S}} \mathbb{P}(\gamma) 
\frac{e^{-\beta\Gamma^{Hex}}}{Z|\Lambda|}
\frac{3}{2\nu_0} (1+o(1)) + 
\sum_{\gamma\in  \mathcal{I}_{\tilde D}} \mathbb{P}(\gamma) 
\frac{e^{-\beta\Gamma^{Hex}}}{Z|\Lambda|}
\frac{2}{\nu_0} (1+o(1)) (1+o(1)).
\end{align}
Observe that there are $|\Lambda|/2$ hexagons in $\Lambda$, $6$ choices for the longest side of $E_{B_5}(r^*)$, and $l-1$ positions for the incomplete bar of cardinality three. Moreover, from $(\mathcal{C}^*)^+$ there are two possible choices to obtain $\mathcal{\tilde{S}}(A_1^*)$ by removing one of two spins of the first type. From $(\mathcal{C}^*)^+$ there is only one choice to obtain $\mathcal{\tilde{D}}(A_1^*)$ by removing the spin of the second type, but to the next step there are two choices to obtain a configuration with only one triangular unit. Thus, the number of the paths in both $\mathcal{I}_{\tilde S}$
and $\mathcal{I}_{\tilde D}$ is equal to $6|\Lambda|(l-1)$. 
Furthermore, by \eqref{eq:path_flow_probability} we have
\begin{equation}
    \mathbb{P}(\gamma)=f(\gamma(K),\gamma(K+1))=\nu_0/3,
\end{equation}
if $\gamma \in \mathcal{I}_{\tilde S}$. Otherwise, if $\gamma \in \mathcal{I}_{\tilde D}$, we have
\begin{equation}
   \mathbb{P}(\gamma)=f(\gamma(K+1),\gamma(K+2))=\nu_0/6.
\end{equation}
Hence, it follows
\begin{equation}
    \text{CAP}(\underline{+1},\underline{-1}) \geq 
    \frac{e^{-\beta\Gamma^{Hex}}}{Z} 5(l-1)(1+o(1)).
\end{equation}
that matches the upper bound~\eqref{eq:upper_bound_capacity_small_delta}.

If $\delta \in (1/2,1)$, it is possible to construct the same flow described above.
In this case, the incomplete bar of length 3 can be attached
on one of the two longest sides of $E_{B_1}(r^*+1)$. Thus, the number of the paths in both $\mathcal{I}_{\tilde S}$ and $\mathcal{I}_{\tilde D}$ is equal to $12|\Lambda|(l-1)$. 
The same arguments allow us to write
\begin{equation}
    \text{CAP}(\underline{+1},\underline{-1}) \geq 
    \frac{e^{-\beta\Gamma^{Hex}}}{Z} 10(l-1) (1+o(1)),
\end{equation}
which matches the upper bound~\eqref{eq:upper_bound_capacity_big_delta}.

\section{Polyiamonds with minimum edge-perimeter and maximal area}\label{geom2}
The problem of finding the shape minimizing the perimeter of a polyiamond given its area is relevant in a handful of fields and a vast literature on the topic has flourished in several communities (see for instance \cite{FuSie, grussien2012iso, nagy2013isoperimetrically, davoli2017sharp, schmidt2013ground}).
However, depending on the application,
different definitions of perimeter may be taken into account.
For instance, one may wish to consider, for  the perimeter of a polyiamond, the number of neighboring vertices or, as we do above, the number of boundary edges.
The problem of minimizing the perimeter might, therefore,
be different.

In this Section we leverage on the results in \cite{FuSie} where the properties of the \emph{site-perimeter} of polyiamonds are studied. More formally,
\begin{definition}\label{def:site_perimeter}
 	Given a polyiamond P, its site-perimeter s(P) is the number of empty triangular units 
    sharing at least one edge with the polyiamond.
\end{definition}

In \cite{FuSie} both polyiamonds with fixed site-perimeter and maximal area and polyiamonds with fixed area and minimal site-perimeter are identified. In particular, they prove that those polyiamonds referred to as quasi-regular hexagons in the previous Sections have both maximal area for fixed site-perimeter and minimal site-perimeter for fixed area. Moreover, they show that 
those polyiamonds that here we called \emph{standard} have minimal site-perimeter and provide an explicit formula for its value.
Here we will show, on one hand, that quasi-regular hexagons are the \emph{only} polyiamonds of maximal area for fixed site-perimeter and, on the other hand, that standard polyiamonds not only minimize the site-perimeter for fixed area, but they also minimize the edge-perimeter establishing Theorem~\ref{thm:optimality_of_standard_polyiamonds}. Moreover, we show that quasi-regular hexagons maximize the area for fixed edge-perimeter as well.

In the remainder of this Section we first recall the definitions and the results of 
\cite{FuSie} (see Subsection~\ref{sec:polyiamonds_known_results}) and then we show how these
results can be extended as mentioned above (see Subsection~\ref{sec:polyiamonds_new_results}). 
Finally we give the proof of Theorem~\ref{thm:optimality_of_standard_polyiamonds}.

\subsection{Site-perimeter of polyiamonds: known results}\label{sec:polyiamonds_known_results}
In \cite{FuSie} hexagons, living on the triangular lattice, are identified starting from an equilateral
triangle and ``cutting the corners'' (see \cite[Definition~2.1]{FuSie}).
In particular hexagons are parametrized by quadruples $(a, b, c, d)$ and $T^{d}_{a,b,c}$ denotes a hexagon obtained from an equilateral
triangle with side length $d$ and removing from its corners the equilateral
triangles of side lengths $a, b$ and $c$  (see Figure~\ref{fig:side_lengths_generic_T_notation_hexagon}).
This parametrization allows to express in a straightforward manner the area and the site-perimeter of any hexagon as follows:
\begin{align}\label{eq:polyiamonds_area_and_perimeter}
\begin{aligned}
    s(T_{a,b,c}^{d}) & = 3d - a - b - c \\
    \area{T_{a,b,c}^{d}} & = d^2 - a^2 - b^2 - c^2
\end{aligned}
\end{align}
Note that degenerate hexagons are included in this definition. If this is the case the ``hexagon'' may, indeed, be a triangle, a quadrilateral or a pentagon (see Figure~\ref{angles13}).

\begin{figure}[htb!]
    \centering
    \includegraphics[width=0.7\textwidth]{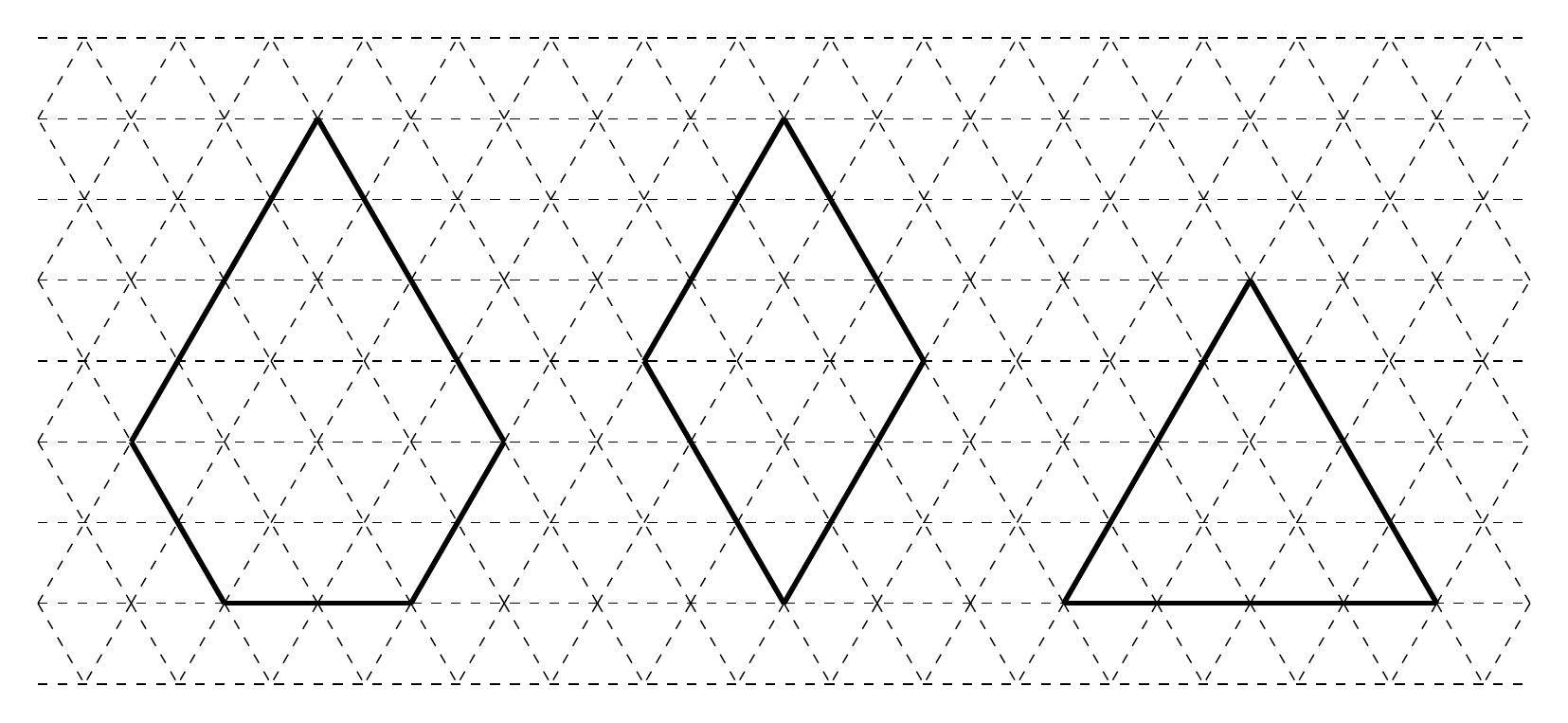}
    \caption{Examples of degenerate hexagons.}
    \label{angles13}
\end{figure}

\begin{remark}
    In general there exist two possible parametrizations identifying the same hexagon. 
\end{remark}

\begin{figure}
    \centering
    \includegraphics[width=0.5\textwidth]{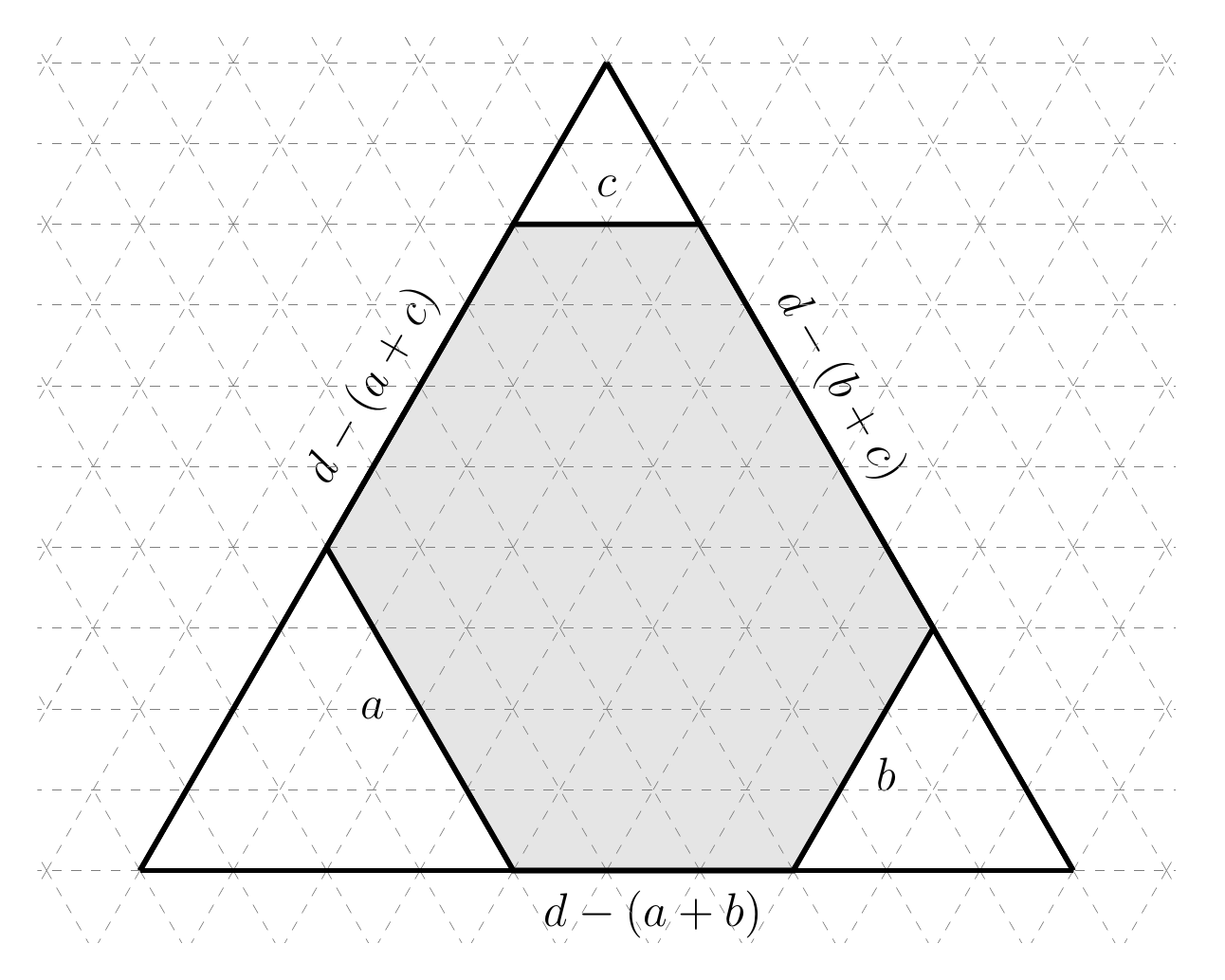}
    \caption{The side lengths of the hexagon obtained from the equilateral
            triangle of side length $d$ cutting the equilateral triangles
            of side lengths $a$, $b$ and $c$.}
    \label{fig:side_lengths_generic_T_notation_hexagon}
\end{figure}

In \cite[Proposition~3.4]{FuSie} the shape of those polyiamonds 
maximizing the area for a given site-perimeter is identified. 
In the following we show that these shapes are those that in this paper we called quasi-regular hexagons. 

Order the set of quasi-regular hexagons by increasing values of their area
and note that going from one quasi-regular hexagon to the next in this
sequence, the perimeter (both edge and site) increases by exactly one unit.
Consequently each quasi-regular hexagon can be identified univocally by its
perimeter. Write all possible values of the site-perimeter of a quasi-regular hexagon as
\begin{align}\label{eq:quasi_hexagon_perimeter_function_radius_and_numbars}
    s = 6r + i, \qquad r \ge 1, \qquad i \in \{0, 1, 2, 3, 4, 5\}.
\end{align}
Note that here $s(P)$ denotes the site-perimeter of polyiamond $P$ whereas in 
\cite{FuSie} the same notation identifies the area of $P$.

Since $s=6r$ is the site-perimeter of the regular hexagon of 
radius/side length $r$, $s = 6r + i$ is the site-perimeter of the quasi-regular
hexagon obtained by adding $i$ bars to the regular hexagon of radius $r$.

By constructing, explicitly, the shapes referred to in \cite[Proposition~3.4]{FuSie}, 
it is straightforward to check that these are the quasi-regular hexagons (see Fig.~\ref{fig:theorem_hexagons_examples}).
In particular there is the following correspondence between the notation used in \cite{FuSie} 
and the notation used in the previous Sections to denote quasi-regular hexagons
\begin{align}\label{eq:quasi_regular_hexagons_correspondence}
\begin{aligned}
    T^{\lfloor \frac{s}{2}\rfloor}_{r,r,r}, i = 0 & \quad \mbox{corresponds to} \quad E(r),\\
    T^{\lfloor \frac{s}{2}\rfloor}_{r-1,r,r}, i = 1 & \quad \mbox{corresponds to} \quad E_{B_1}(r),\\
    T^{\lfloor \frac{s}{2}\rfloor}_{r,r,r+1}, i = 2 & \quad \mbox{corresponds to} \quad E_{B_2}(r),\\
    T^{\lfloor \frac{s}{2}\rfloor}_{r,r,r}, i = 3 & \quad \mbox{corresponds to} \quad E_{B_3}(r),\\
    T^{\lfloor \frac{s}{2}\rfloor}_{r,r+1,r+1}, i = 4 & \quad \mbox{corresponds to} \quad E_{B_4}(r),\\
    T^{\lfloor \frac{s}{2}\rfloor}_{r,r,r+1}, i = 5 & \quad \mbox{corresponds to} \quad E_{B_5}(r).
\end{aligned}
\end{align}
Since there is only a quasi-regular hexagon for each value of $s$, 
the Proposition amounts to saying that quasi-regular hexagons maximize the area 
for a given site-perimeter.

As a consequence, \cite[Proposition~4.5]{FuSie} states, in our notation, 
that the minimal site-perimeter for a polyiamond of area $A$ 
is the site-perimeter of the smallest quasi-regular hexagon of area at least $A$.

\begin{figure}
    \centering
    \includegraphics[width=0.9\textwidth]{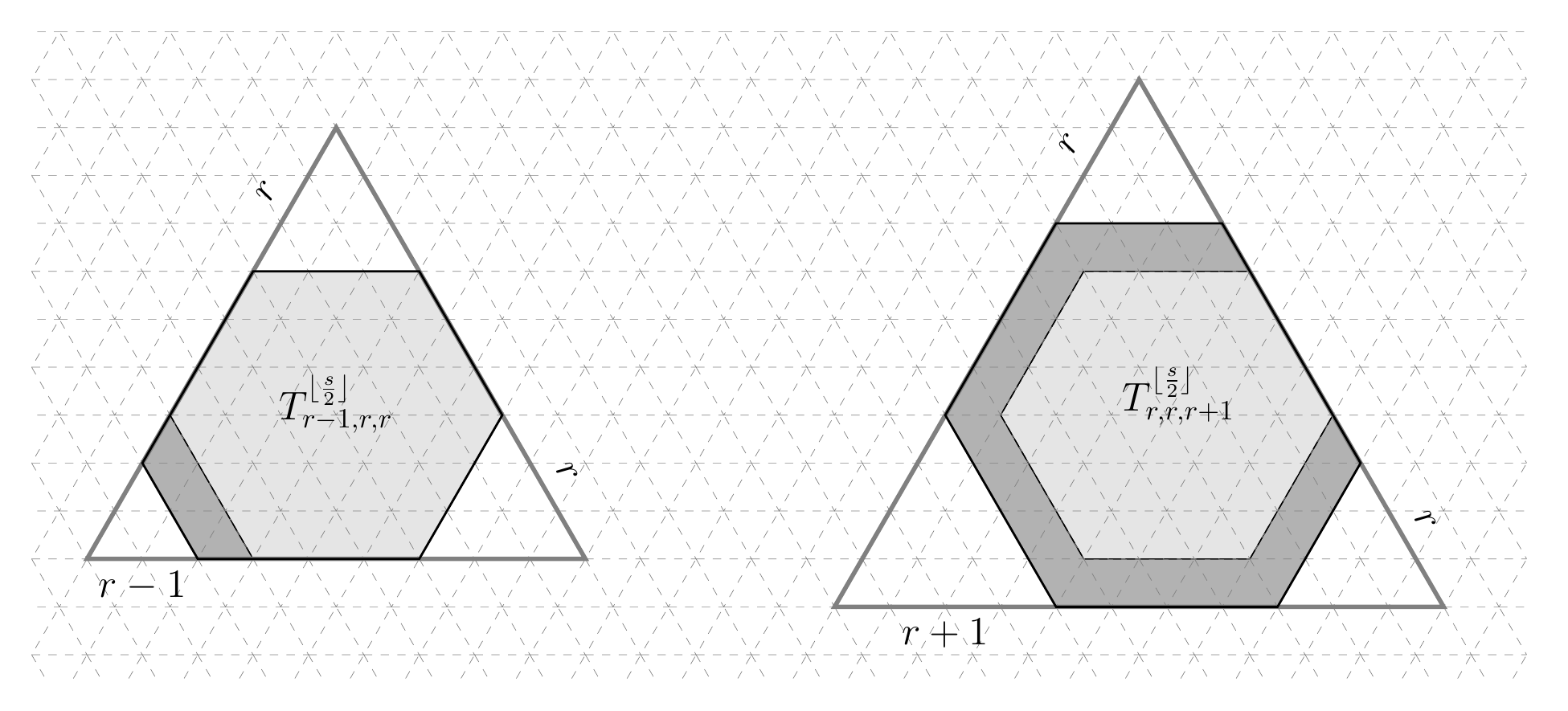}
    \caption{Two examples of the correspondences in~$\eqref{eq:quasi_regular_hexagons_correspondence}$. On the left the correspondence between 
    $T_{r-1, r, r}^{\lfloor \frac{s}{2} \rfloor}$ and $E_{B_1}(r)$. On the right the correspondence between
    $T^{\lfloor \frac{s}{2}\rfloor}_{r,r,r+1}$ and $E_{B_5}(r)$.}
    \label{fig:theorem_hexagons_examples}
\end{figure}

\subsection{Site-perimeter of polyiamonds: further results}\label{sec:polyiamonds_new_results}
We extend the result of \cite[Proposition~3.4]{FuSie} as follows
\begin{proposition}\label{thm:quasi_regular_hex_unique_area_maximizers_for_fixed_site_perimeter}
	Quasi-regular hexagons are the unique polyiamonds of maximal area for fixed site-perimeter.
\end{proposition}
\begin{proof}
    Let $s = 6r +i$ be the site-perimeter of a hexagon $T_{a, b, c}^{d}$ and consider the function
    \mbox{$M = s^{2} - 6 \area{T_{a,b,c}^{d}}$}.
    If the site-perimeter is fixed, polyiamonds of maximal area are those that minimize $M$.
	By~\eqref{eq:polyiamonds_area_and_perimeter} one can compute $M$ in terms of the parameters identifying the hexagon as
    $M = 3(d-a-b-c)^2 + 2((a-b)^2 + (a-c)^2 + (b-c)^2)$.
    In \cite{FuSie} it is shown that the minimum of $M$ depends on the value of the remainder $i$
    (modulo $6$). Calling $M^\star$ the minimum of $M$, by the proof of \cite[Proposition~3.4]{FuSie} we know that 
    \begin{align}\label{eq:optimal_values_Mstar}
    	M^\star = \begin{cases}
    							0 & \text{if } i \in \{0\}\\
                                3 & \text{if } i \in \{3\}\\
                                4 & \text{if } i \in \{2, 4\}\\
                                7 & \text{if } i \in \{1, 5\}
    			  \end{cases}
    \end{align}

    Call $\alpha = d - a - b - c$ and 
    $\beta = (a-b)^2 + (a-c)^2 + (b-c)^2$. 
    This implies that $M$ can be written as $M = 3\alpha^2 + 2\beta$ where $\alpha$ and $\beta$ are integers. 
    Therefore for $i \in \{0, \ldots, 5\}$ there is a unique pair of integers 
    $(\alpha^2, \beta)$ for which the optimal value of $M$ is attained.
    
    If $i = 0$, we see from \eqref{eq:optimal_values_Mstar} that $M^\star = 0$ and, hence, it must be $\alpha = 0$ and $\beta = 0$. This, in turn, implies $a = b = c  = v$ for some positive integer value $v$ and $d = a + b + c= 3v$.
    Hence, recalling that $s=3d-a-b-c$, we must have $0 = \alpha = d - a - b - c = \frac{1}{3}(s - 2(a + b + c)) = \frac{1}{3}(s - 6v)$. The unique solution of this equation is $v = r$ . This yields $d = 3r = \frac{s}{2} \in \mathbb{N}$ and, therefore, the unique quadruple  $(a, b, c, d)$  parametrizing a hexagon for which the minimum of $M$ is attained is $(r, r, r, \frac{s}{2})$. 
    
    If $i = 1$, from \eqref{eq:optimal_values_Mstar} we have $M^\star = 7$  amounting to saying $\alpha^2 = 1$ and $\beta = 2$. Assuming, without loss of generality, $a \le b \le c$, the latter implies that, for some $v$ it must be either
    $a = v - 1; \, b = c = v$ or 
    $a = b = v-1; \, c = v$.
    
    Consider, first, the case, $a = v - 1; \, b = c = v$.
    We have
    $\alpha = \frac{1}{3}(s - 2(a + b + c))
            = \frac{1}{3}(6(r-v) + 3)$.
    If $\alpha = +1$, then the solution of the equation is $v=r$ implying $d=3r = \lfloor \frac{s}{2} \rfloor$.
    This solution corresponds to the quadruple $(r-1, r, r, \lfloor \frac{s}{2} \rfloor)$.
    If $\alpha = -1$, then it must be $v = r+1$ and the associated quadruple is
    $(r, r+1, r+1, \lceil \frac{s}{2} \rceil)$.

    If we consider the case $a = b = v-1; \, c = v$, then
    $\alpha = \frac{1}{3}(6(r-v) + 5)$. However no acceptable quadruple can be obtained
    in this case since $\alpha$ must be integer and $\frac{1}{3}(6(r-v) + 5) \notin \mathbb{N}$.
    
    Arguing in the same manner for $i = 2, 3, 4, 5$ it is possible to determine all quadruples for which
    the minimum of $M$ is attained. In particular we have:
    
    if $i = 2$ the only quadruple minimizing $M$ is $(r, r, r+1, \frac{s}{2})$;
    
    if $i = 3$ the two quadruples minimizing $M$ are
    $(r, r, r, \lfloor\frac{s}{2}\rfloor)$ and
    $(r+1, r+1, r+1, \lceil\frac{s}{2}\rceil)$;
    
    if $i = 4$, $M$ is minimized only by the quadruple $(r, r+1, r+1, \frac{s}{2})$;
    
    if $i = 5$ the two quadruples minimizing $M$ are
    $(r, r, r+1, \lfloor\frac{s}{2}\rfloor)$ and
    $(r+1, r+1, r+2, \lceil\frac{s}{2}\rceil)$.
    
    Note that, for $i$ even, there is only one quadruple minimizing $M$ and, therefore, in these
    cases the hexagon maximizing the area for fixed site-perimeter is obviously unique.
    
    On the other hand, for $i$ odd, there are two quadruples minimizing $M$. However, these two quadruples
    identify the same hexagon. To see this we argue as follows.
    The parameters $a, b$ and $c$ are the lengths of three non consecutive sides of the hexagon.
    Therefore there is another parametrization $T_{a^\prime, b^\prime, c^\prime}^{d^\prime}$, in principle different from $T_{a, b, c}^{d}$, for a hexagon with site-perimeter $s$, given in terms of the lengths
    of the other three non consecutive sides for a suitable value $d^\prime$.
    Since $i$ is odd, $\frac{s}{2} \notin \mathbb{N}$ thus it is not possible to have $a = a^\prime$, $b = b^\prime$ and $c = c^\prime$ because $s = a + a^\prime + b + b^\prime + c + c^\prime$ would be even.
    Thus for each hexagon with odd site-perimeter there are, indeed, two quadruples. Since for $i$ odd the minimum of $M$ is always attained only in two quadruples they must identify the same hexagon (if the two quadruples identify different hexagons, then there should exist four quadruples minimizing $M$).
\end{proof}

\subsubsection{Proof of Theorem~\ref{thm:optimality_of_standard_polyiamonds}}

Edge-perimeter and site-perimeter are closely related. It is straightforward to check that
\linebreak
\mbox{$s(P) \le p(P)$}. Indeed, as shown in Figure~\ref{fig:site_edge_perimeter_relation}, an empty triangular unit, giving unitary contribution to the site-perimeter, may share $1$, $2$ or $3$ edges with the polyiamond each giving a unitary contribution to the edge-perimeter. More precisely the following proposition holds:
\begin{proposition}\label{prop:site_edge_perimeter_relation}
    Let $\nu(P)$ be the number of 
    $\frac{5}{3}\pi$ internal angles that are not part of an elementary hole
    and $e(P)$ the number of elementary holes in $P$. Then
    \begin{equation}\label{eq:site_edge_perimeter_relation}
        p(P) = s(P) + \nu(P) + 2e(P)
    \end{equation}
\end{proposition}

Proposition~\ref{prop:site_edge_perimeter_relation} immediately implies
\begin{proposition}\label{prop:minimal_edge_perimeter_for_minimal_site_perimeter}
    If $s(P)$ is minimal (that is there is no polyiamond with the same area and a smaller site-perimeter) and $p(P) = s(P)$, then $p(P)$ is minimal as well (that is there is no polyiamond with the same area and a smaller edge-perimeter). 
\end{proposition}
The proof of this statement is straightforward,
indeed, call
$\alpha(P) = \nu(P) + 2e(P)$ and note that $\alpha(P) \ge 0$.
Let $\bar{P}$ be a minimizer of $s$ such that
$\alpha(\bar{P}) = 0$ (it, obviously, exists). Assume there is a polyiamond
$\tilde{P}$ such that 
$p(\tilde{P}) < p(\bar{P})$. This is equivalent to saying
$s(\tilde{P}) + \alpha(\tilde{P}) < 
s(\bar{P}) + \alpha(\bar{P}) 
= s(\bar{P})$ and this is clearly a contradiction since 
$\alpha \geq 0$.
    
\begin{figure}
    \centering
    \includegraphics[width=0.7\textwidth]{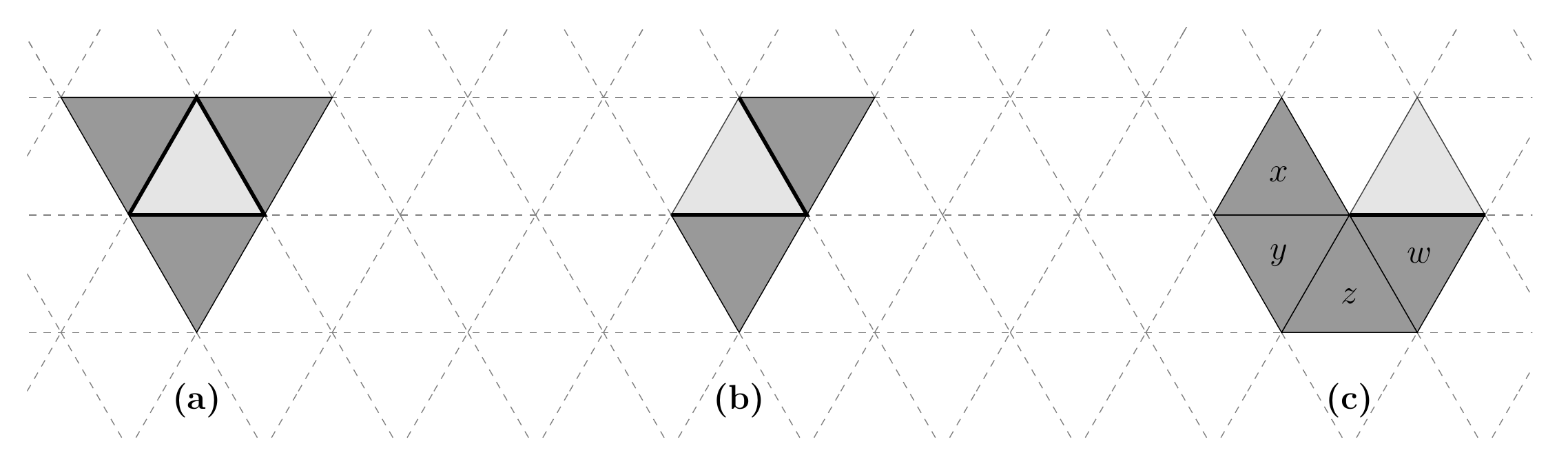}
    \caption{The three possible cases for the number of edges shared by
    an empty triangular unit and the polyiamond. If the number
    of shared edges is 3 (case (a)) the empty triangular unit is an elementary
    hole of the polyiamond. The empty triangular unit shares 2 edges with the polyiamond if and only if the two edges identify a $\frac{5}{3}\pi$ internal angle (case (b)).
    The case where the empty triangular unit and the polyiamond share a single edge is represented in (c). 
    Note that this case may correspond to different values of the internal angle (that is an angle of $\frac{\pi}{3}$ when the polyiamond contains only the triangular unit \emph{w}; an angle of $\frac{2 \pi}{3}$ when the polyiamond contains the triangular units \emph{w}, \emph{z}; an angle $\pi$ when the polyiamond contains the triangular units \emph{w}, \emph{z}, \emph{y}; an angle of $\frac{4 \pi}{3}$ when the polyiamond contains the triangular units \emph{w}, \emph{z}, \emph{y} and \emph{x}).
    }
    \label{fig:site_edge_perimeter_relation}
\end{figure}
    
\begin{proof}[Proof of Proposition~\ref{prop:site_edge_perimeter_relation}]
    For the proof we refer to Figure~\ref{fig:site_edge_perimeter_relation}.
    Note that in case (c) the contribution of 
    the empty triangular unit to the site-perimeter of the polyiamond
    is the same (one unit) that the shared edge gives to the
    edge-perimeter. In case (a) the three edges on the boundary of the polyiamond are adjacent to
    the same empty triangular unit. Therefore, for each elementary hole the edge-perimeter of the polyiamond increases by two extra units with respect to the site-perimeter.
    Finally, in case (b) the two edges on the boundary are adjacent to the same empty triangular unit. Hence, for each $\frac{5}{3}\pi$ angle the edge-perimeter of the
    polyiamond increases by one extra unit with respect to the site-perimeter.
\end{proof}

We have already seen that all quasi-regular hexagons have minimal site-perimeter for fixed area and maximal area for fixed site perimeter. Thanks to Proposition~\ref{prop:site_edge_perimeter_relation} it is possible to show that these two properties holds also for the edge-perimeter. Moreover, we show that quasi-regular hexagons are the unique polyiamonds of
maximal area for fixed edge-perimeter.
Denote by $\quasiregularset$ the set of all quasi-regular hexagons. More formally we have:
\begin{lemma}\label{lemma:quasi_regular_hex_unique_edge_perimeter_minimizers}
   Let $E$ be a quasi-regular hexagon and $P \notin \quasiregularset$ a polyiamond such that $||P|| \geq ||E||$. 
   Then \mbox{$p(P) > p(E)$} and \mbox{$s(P) > s(E)$}.
\end{lemma}
\begin{proof}[Proof of Lemma~\ref{lemma:quasi_regular_hex_unique_edge_perimeter_minimizers}]
   We give the proof for $p(P)$. The proof for $s(P)$ is analogous.
   Let $A$ be the area of the quasi-regular hexagon of edge-perimeter $p^{\star}$. We prove the equivalent statement: if $p(P) \le p^\star$, then $\area{P} < A$ for all $P \notin \quasiregularset$.
   
    Denote by
    $s(P)$ the site-perimeter of $P$ and by
    $\areamax{s}$ the largest possible value of the area 
    for a polyiamond of site-perimeter $s$. We consider the cases \mbox{$p(P) < p^\star$} and 
    \mbox{$p(P) = p^\star$} 
    separately. 
    
    Let $p(P) < p^\star$. By Proposition~\ref{eq:site_edge_perimeter_relation}, we have
    $s(P) \le p(P) <p^{\star}$. Then
    $\area{P} \le \areamax{s(P)} < \areamax{p^\star} = A$.
    The last inequality follows from \cite[Proposition~3.6]{FuSie} stating that
    $\areamax{\cdot}$ is a strictly increasing function.
    
    Let $p(P) = p^\star$, then $s(P) \le p^\star$ and
    $\area{P} \le \areamax{s(P)}\le \areamax{p^\star} = A$
    for all $P \notin \quasiregularset$ by
    Lemma~\ref{thm:quasi_regular_hex_unique_area_maximizers_for_fixed_site_perimeter} and, again, by noting that $\areamax{\cdot}$
    is increasing.
\end{proof}

The previous results serve as building blocks to show that, fixing the area, also all standard polyiamonds, other than quasi-regular hexagons, have minimal edge-perimeter establishing Theorem~\ref{thm:optimality_of_standard_polyiamonds}. 

\begin{proof}[Proof of Theorem~\ref{thm:optimality_of_standard_polyiamonds}]
Let $E$ be a quasi-regular hexagon of edge-perimeter $p(E)$ and area $\area{E}$.
Consider, at first, the standard polyiamonds obtained
by adding an incomplete bar 
with an odd number of triangular units
to $E$. These polyiamonds
have edge-perimeter $p(E) + 1$ and area 
strictly larger than $\area{E}$ and, therefore,
their edge-perimeter is minimal by Lemma~\ref{lemma:quasi_regular_hex_unique_edge_perimeter_minimizers}. 
It remains to show that also standard polyiamonds obtained
by adding an incomplete bar 
with an even number of triangular units
to a quasi-regular hexagon have minimal edge-perimeter.

Write $A = A_0 + \ell$ where $A_0$ is the area of the greatest quasi-regular hexagon $R'_{A}$ containing,
at most, $A_0$ triangular units (that can be obtained
via Algorithm~\ref{algorithm}) and $\ell \ge 2$ is the (even) number
of triangular units in the incomplete bar.

Let $p^{\star} = p(R'_{A})$ and consider
a standard polyiamond $\tilde{P}$ with area
$A$ and perimeter $p(\tilde{P} )= p^{\star} + 2$.
We will show that there is no polyiamond $P$ such that
$\area{P} = A$ and $p(P) < p(\tilde{P}) = p^{\star} + 2$.
    
If a $P$ as such existed, then it would be immediate to check that
it could not have neither $\frac{5}{3}\pi$ internal angles nor elementary
holes. Indeed, the polyiamond $P_{+}$ obtained by adding a triangular unit in the ``corner'' or in the elementary hole would have perimeter
$p(P_{+}) < p^{\star} + 1$ and area 
$\area{P_{+}} = A_0 + \ell + 1 > A_0$. 
This would contradict 
Lemma~\ref{lemma:quasi_regular_hex_unique_edge_perimeter_minimizers}.
Similarly, it can be seen that $P$ can not have ``protuberances'' 
($\frac{1}{3}\pi$ internal angles). Indeed, the polyiamond
$P_{-}$ obtained from $P$ by removing the protuberance would
have perimeter $p(P_{-}) < p^{\star}+1$ and area
$\area{P_{-}} = A_0 + \ell - 1 > A_0$, since $\ell \geq 2$
and, also in this case,
Lemma~\ref{lemma:quasi_regular_hex_unique_edge_perimeter_minimizers} would be
contradicted.
Therefore $P$ can only have $\frac{2}{3}\pi$ and $\frac{4}{3}\pi$ internal
angles.
Consider the sequence of polyiamonds 
$P =: P_{0}, P_{1}, \ldots P_{m}$ where each $P_{i}$ 
is obtained from $P_{i-1}$ by adding an elementary rhombus
to a corner corresponding to a $\frac{4}{3}\pi$ internal angle
until no $\frac{4}{3}\pi$ internal angle is present.
Then $p(P_i) = p(P_{i-1})$ for all $i$ and
$\area{P_{i}} = \area{P_{i-1}} + 2 = \area{P} + 2i$.
Note that, if some of the $P_{i}$ had elementary holes, $\frac{5}{3}\pi$ or
$\frac{1}{3}\pi$ internal angles we could argue as above and obtain a contradiction.
Then $P_{m}$ should be, necessarily, a (non degenerate) hexagon.

To conclude, we rely on the following
\begin{lemma}\label{lemma:parity_area_perimeter}
   Area and perimeter (both site and edge) of every hexagon have the same parity.
\end{lemma}
\begin{proof}[Proof of Lemma~\ref{lemma:parity_area_perimeter}]
    As already mentioned above, the area and both the site-perimeter
    and edge-perimeter of
    a hexagon $E$ are computed to be $\area{E} = d^2 - (a^2 + b^2 + c^2)$
    and $p(E) = 3d - (a + b + c)$ respectively, where $a, b, c, d$ are
    parameters identifying the hexagon. The conclusion follows by observing
    that $d$ has the same parity of $d^2$ and $(a + b + c)$ has the same
    parity of $(a^2 + b^2 + c^2)$.
\end{proof}
Observe that $\area{P_m}$ has the same parity of $\area{R'_{A}}$ 
(both $\ell$ and $2i$ are, indeed, even) whereas
the edge-perimeters of these two polyiamonds differ by one and, hence, have different parities.
Since $P_m$ is a proper hexagon, this contradicts Lemma~\ref{lemma:parity_area_perimeter} completing the proof.
\end{proof}

We conclude this section by providing a lower bound for the 
edge-perimeter of \emph{non standard} polyiamonds with area equal to the area of a quasi-regular hexagon.

\begin{lemma}\label{lemma:lower_bound_edge_perimeter_non_standard_polyiamonds}
    Let $\bar{p}$ be the edge-perimeter of a quasi-regular hexagon
    and let $\bar{A}$ be its area. Then, for all $P \notin \quasiregularset$ such that $\area{P} = \bar{A}$,
    $p(P) \ge \bar{p} + 2$.
\end{lemma}
\begin{proof}
    The proof can be done following the same strategy of the
    proof of 
    Theorem~$\ref{thm:optimality_of_standard_polyiamonds}$.
    Let $P\notin \quasiregularset$ 
    be a polyiamond of area $\bar{A}$ where
    $\bar{A}$ is the area of a quasi-regular hexagon.
    By Lemma~$\ref{lemma:quasi_regular_hex_unique_edge_perimeter_minimizers}$ it follows that $p(P) \ge \bar{p} + 1$. 
    We will show that $p(P) = \bar{p} + 1$
    can not hold.
    
    To this end, suppose $p(P) = \bar{p} + 1$. 
    Then $P$ can not have neither $\frac{5}{3}\pi$ internal angles
    nor elementary holes. Indeed, if it were the case, the polyiamond
    obtained by filling the angle or the hole would have area
    $\bar{A} + 1$ and edge-perimeter at most $\bar{p}$
    contradicting Proposition~\ref{thm:quasi_regular_hex_unique_area_maximizers_for_fixed_site_perimeter}.
    
    Consider, then, the sequence of polyiamonds $P =: P_{0}, P_{1}, \ldots P_{m}$ where each $P_{i}$
    is obtained from $P_{i-1}$ by adding an elementary rhombus
    to a corner corresponding to a $\frac{4}{3}\pi$ internal angle
    until no $\frac{4}{3}\pi$ internal angle is present.
    Then $p(P_i) = p(P_{i-1})$ for all $i$ and
    $\area{P_{i}} = \area{P_{i-1}} + 2 = \area{P} + 2i$.
    Note that, if some of the $P_{i}$ had either elementary holes or $\frac{5}{3}\pi$ internal angles we could argue as above and obtain a contradiction.
    Then $P_{m}$ must be, necessarily, a hexagon, possibly degenerate, and, therefore, 
    $p(P_m)$ and $\area{P_m}$ must have the same parity by
    Lemma~\ref{lemma:parity_area_perimeter}.
    By construction, $\area{P_m}$ has the same parity of $\bar{A}$
    and $p(P_m)$ has the same parity of $\bar{p} + 1$.
    Since $\bar{p}$ is the edge-perimeter of
    a quasi-regular hexagon of area $\bar{A}$, 
    $\bar{p} + 1$ and $\bar{A}$ have different parities
    contradicting the hypothesis that $P_m$ is a hexagon.
\end{proof}

\begin{remark}
    Note that, in the case of site-perimeter, the analogue of the property of the previous lemma does not hold. 
\end{remark}
Indeed, a counterexample is given by the polyiamond obtained by
removing an elementary rhombus from one corner of the hexagon
and moving it on top of a side of the hexagon. The polyiamond obtained
in this way has site-perimeter $\bar{p} + 1$.

\section*{Acknowledgments}
Unfortunately Francesca passed away on 21 October 2021, before the completion of the review process of the manuscript. She kept working passionately on this paper until the end and every page is full of her love for mathematics and for life.
Francesca, we miss you.

The Authors are grateful to 
Cristian Spitoni 
for valuable suggestions.
F.R.N. was partially supported by 
the Netherlands Organisation for Scientific Research (NWO) [Gravitation 
Grant number 024.002.003--NETWORKS].  A.T. has been supported by the H2020 Project Stable and Chaotic 
Motions in the Planetary Problem (Grant 677793 StableChaoticPlanetM of the European Research Council). 
V.J. and F.R.N. are grateful to INDAM-GNAMPA.

\providecommand{\bysame}{\leavevmode\hbox to3em{\hrulefill}\thinspace}
\providecommand{\MR}{\relax\ifhmode\unskip\space\fi MR }
\providecommand{\MRhref}[2]{%
  \href{http://www.ams.org/mathscinet-getitem?mr=#1}{#2}
}
\providecommand{\href}[2]{#2}


\begin{thebibliography}{10}

\bibitem{apollonio2019shaken}
Valentina Apollonio, Roberto D'Autilia, Benedetto Scoppola, Elisabetta
  Scoppola, and Alessio Troiani, \emph{Shaken dynamics for 2d {Ising} models},
  arXiv preprint arXiv:1904.06257 (2019).

\bibitem{apollonio2019criticality}
Valentina Apollonio, Roberto D’Autilia, Benedetto Scoppola, Elisabetta
  Scoppola, and Alessio Troiani, \emph{Criticality of measures on 2-d {Ising}
  configurations: from square to hexagonal graphs}, Journal of Statistical
  Physics \textbf{177} (2019), no.~5, 1009--1021. \MR{4031905}

\bibitem{bashiri2017note}
K~Bashiri, \emph{On the metastability in three modifications of the ising
  model}, MARKOV PROCESSES AND RELATED FIELDS \textbf{25} (2019), no.~3,
  483--532. \MR{3966349}

\bibitem{beltran2012tunneling}
J~Beltr{\'a}n and C~Landim, \emph{Tunneling and metastability of continuous
  time {Markov} chains ii, the nonreversible case}, Journal of Statistical
  Physics \textbf{149} (2012), no.~4, 598--618. \MR{2998592}

\bibitem{beltran2010tunneling}
Johel Beltran and Claudio Landim, \emph{Tunneling and metastability of
  continuous time {Markov} chains}, Journal of Statistical Physics \textbf{140}
  (2010), no.~6, 1065--1114. \MR{2684500}

\bibitem{arous1996metastability}
G{\'e}rard Ben~Arous and Rapha{\"e}l Cerf, \emph{Metastability of the three
  dimensional {Ising} model on a torus at very low temperatures}, Electronic
  Journal of Probability \textbf{1} (1996). \MR{1423463}

\bibitem{bet2020effect}
Gianmarco Bet, Vanessa Jacquier, and Francesca~R Nardi, \emph{Effect of energy
  degeneracy on the transition time for a series of metastable states:
  application to probabilistic cellular automata}, arXiv preprint
  arXiv:2007.08342 (2020). \MR{4280486}

\bibitem{bianchi2009sharp}
Alessandra Bianchi, Anton Bovier, Dmitry Ioffe, et~al., \emph{Sharp asymptotics
  for metastability in the random field curie-weiss model}, Electronic Journal
  of Probability \textbf{14} (2009), 1541--1603. \MR{2525104}

\bibitem{bianchi2016metastable}
Alessandra Bianchi and Alexandre Gaudilliere, \emph{Metastable states,
  quasi-stationary distributions and soft measures}, Stochastic Processes and
  their Applications \textbf{126} (2016), no.~6, 1622--1680. \MR{3483732}

\bibitem{bovier2016metastability}
Anton Bovier and Frank {Den}~Hollander, \emph{Metastability: a
  potential-theoretic approach}, vol. 351, Springer, 2016. \MR{3445787}

\bibitem{bovier2010homogeneous}
Anton Bovier, Frank {Den}~Hollander, Cristian Spitoni, et~al.,
  \emph{Homogeneous nucleation for {Glauber} and {Kawasaki} dynamics in large
  volumes at low temperatures}, The Annals of Probability \textbf{38} (2010),
  no.~2, 661--713. \MR{2642889}

\bibitem{boviermanzo2002metastability}
Anton Bovier and Francesco Manzo, \emph{Metastability in {Glauber} dynamics in
  the low-temperature limit: beyond exponential asymptotics}, Journal of
  Statistical Physics \textbf{107} (2002), no.~3-4, 757--779. \MR{1898856}

\bibitem{cerf2013nucleation}
Rapha{\"e}l Cerf and Francesco Manzo, \emph{Nucleation and growth for the
  {Ising} model in $ d $ dimensions at very low temperatures}, The Annals of
  Probability \textbf{41} (2013), no.~6, 3697--3785. \MR{3161463}

\bibitem{cirillo2013relaxation}
Emilio N.~M. Cirillo and Francesca~R. Nardi, \emph{Relaxation height in energy
  landscapes: an application to multiple metastable states}, Journal of
  Statistical Physics \textbf{150} (2013), no.~6, 1080--1114. \MR{3038678}

\bibitem{cirillo2008competitive}
Emilio N.~M. Cirillo, Francesca~R. Nardi, and Cristian Spitoni,
  \emph{Competitive nucleation in reversible {Probabilistic} {Cellular}
  {Automata}}, Physical Review E \textbf{78} (2008), no.~4, 040601.
  \MR{2529569}

\bibitem{cirillo2002note}
Emilio~NM Cirillo, \emph{A note on the metastability of the {Ising} model: the
  alternate updating case}, Journal of statistical physics \textbf{106} (2002),
  no.~1, 385--390. \MR{1881734}

\bibitem{cirillo2003metastability}
Emilio~NM Cirillo and Francesca~R Nardi, \emph{Metastability for a stochastic
  dynamics with a parallel heat bath updating rule}, Journal of statistical
  physics \textbf{110} (2003), no.~1, 183--217. \MR{1966327}

\bibitem{CNrelax13}
\bysame, \emph{Relaxation height in energy landscapes: an application to
  multiple metastable states}, Journal of Statistical Physics \textbf{150}
  (2013), no.~6, 1080--1114. \MR{3038678}

\bibitem{cirillo2008metastability}
Emilio~NM Cirillo, Francesca~R Nardi, and Cristian Spitoni, \emph{Metastability
  for reversible probabilistic cellular automata with self-interaction},
  Journal of Statistical Physics \textbf{132} (2008), no.~3, 431--471.
  \MR{2415112}

\bibitem{davoli2017sharp}
Elisa Davoli, Paolo Piovano, and Ulisse Stefanelli, \emph{Sharp $n^{3/4}$ law
  for the minimizers of the edge-isoperimetric problem on the triangular
  lattice}, Journal of Nonlinear Science \textbf{27} (2017), no.~2, 627--660.
  \MR{3614767}

\bibitem{dehghanpour1997metropolis}
Pouria Dehghanpour and Roberto~H Schonmann, \emph{Metropolis dynamics
  relaxation via nucleation and growth}, Communications in mathematical physics
  \textbf{188} (1997), no.~1, 89--119. \MR{1471333}

\bibitem{FuSie}
G{\'a}bor F{\"u}lep and N{\'a}ndor Sieben, \emph{Polyiamonds and polyhexes with
  minimum site-perimeter and achievement games.}, The Electronic Journal of
  Combinatorics [electronic only] \textbf{17} (2010), no.~1, v17i1r65--pdf.
  \MR{2644851}

\bibitem{gaudillierelandim2014}
Alexandre Gaudilliere and Claudio Landim, \emph{{A Dirichlet principle for non
  reversible {Markov} chains and some recurrence theorems}}, {Probability
  Theory and Related Fields} \textbf{158} (2014), 55--89. \MR{3152780}

\bibitem{gaudilliere2020asymptotic}
Alexandre Gaudilli{\`e}re, Paolo Milanesi, and Maria~Eul{\'a}lia Vares,
  \emph{Asymptotic exponential law for the transition time to equilibrium of
  the metastable kinetic {Ising} model with vanishing magnetic field}, Journal
  of Statistical Physics (2020), 1--46. \MR{4091560}

\bibitem{grussien2012iso}
Berit Gru{\ss}ien, \emph{Isoperimetric inequalities on hexagonal grids}, arXiv
  preprint arXiv:1201.0697 (2012).

\bibitem{jovanovski2017metastability}
Oliver Jovanovski, \emph{Metastability for the {Ising} model on the hypercube},
  Journal of Statistical Physics \textbf{167} (2017), no.~1, 135--159.
  \MR{3619543}

\bibitem{kotecky1994shapes}
R~Kotecky` and E~Olivieri, \emph{Shapes of growing droplets—a model of escape
  from a metastable phase}, Journal of statistical physics \textbf{75} (1994),
  no.~3, 409--506. \MR{1279759}

\bibitem{kotecky1993droplet}
Roman Kotecky` and Enzo Olivieri, \emph{Droplet dynamics for asymmetric {Ising}
  model}, Journal of statistical physics \textbf{70} (1993), no.~5, 1121--1148.
  \MR{1208633}

\bibitem{manzo1998relaxation}
F~Manzo and E~Olivieri, \emph{Relaxation patterns for competing metastable
  states: a nucleation and growth model}, {Markov} Proc. Relat. Fields, vol.~4,
  1998, pp.~549--570. \MR{1677059}

\bibitem{manzo2001dynamical}
\bysame, \emph{Dynamical {Blume}--{Capel} model: competing metastable states at
  infinite volume}, Journal of Statistical Physics \textbf{104} (2001),
  no.~5-6, 1029--1090. \MR{1858997}

\bibitem{manzo2004essential}
Francesco Manzo, Francesca~R. Nardi, Enzo Olivieri, and Elisabetta Scoppola,
  \emph{On the essential features of metastability: tunnelling time and
  critical configurations}, Journal of Statistical Physics \textbf{115} (2004),
  no.~1-2, 591--642. \MR{2070109}

\bibitem{MNOS04}
Francesco Manzo, Francesca~R Nardi, Enzo Olivieri, and Elisabetta Scoppola,
  \emph{On the essential features of metastability: tunnelling time and
  critical configurations}, Journal of Statistical Physics \textbf{115} (2004),
  no.~1-2, 591--642. \MR{2070109}

\bibitem{nagy2013isoperimetrically}
Benedek Nagy and Krisztina Barczi, \emph{Isoperimetrically optimal polygons in
  the triangular grid with jordan-type neighbourhood on the boundary},
  International Journal of Computer Mathematics \textbf{90} (2013), no.~8,
  1629--1652. \MR{2813684}

\bibitem{nardi2012sharp}
FR~Nardi and C~Spitoni, \emph{Sharp asymptotics for stochastic dynamics with
  parallel updating rule}, Journal of statistical physics \textbf{146} (2012),
  no.~4, 701--718. \MR{2916093}

\bibitem{nardi1996low}
Francesca~R. Nardi and Enzo Olivieri, \emph{Low temperature stochastic dynamics
  for an {Ising} model with alternating field}, {Markov} Proc. Relat. Fields,
  vol.~2, 1996, pp.~117--166. \MR{1418410}

\bibitem{nardi2016hitting}
Francesca~R Nardi, Alessandro Zocca, and Sem~C Borst, \emph{Hitting time
  asymptotics for hard-core interactions on grids}, Journal of Statistical
  Physics \textbf{162} (2016), no.~2, 522--576. \MR{3441372}

\bibitem{neves1991critical}
E~Jordao Neves and Roberto~H Schonmann, \emph{Critical droplets and
  metastability for a {Glauber} dynamics at very low temperatures},
  Communications in Mathematical Physics \textbf{137} (1991), no.~2, 209--230.
  \MR{1101685}

\bibitem{neves1992behavior}
E~Jord{\~a}o Neves and Roberto~H Schonmann, \emph{Behavior of droplets for a
  class of {Glauber} dynamics at very low temperature}, Probability theory and
  related fields \textbf{91} (1992), no.~3-4, 331--354. \MR{1151800}

\bibitem{olivieri2005large}
Enzo Olivieri and Maria~Eul{\'a}lia Vares, \emph{Large deviations and
  metastability}, vol. 100, Cambridge University Press, 2005. \MR{2123364}

\bibitem{schmidt2013ground}
Bernd Schmidt, \emph{Ground states of the 2d sticky disc model: fine properties
  and n 3/4 law for the deviation from the asymptotic wulff shape}, Journal of
  Statistical Physics \textbf{153} (2013), no.~4, 727--738. \MR{3117623}

\bibitem{schonmann1994slow}
Roberto~H Schonmann, \emph{Slow droplet-driven relaxation of stochastic {Ising}
  models in the vicinity of the phase coexistence region}, Communications in
  Mathematical Physics \textbf{161} (1994), no.~1, 1--49. \MR{1266068}

\bibitem{schonmann1998wulff}
Roberto~H Schonmann and Senya~B Shlosman, \emph{Wulff droplets and the
  metastable relaxation of kinetic {Ising} models}, Communications in
  mathematical physics \textbf{194} (1998), no.~2, 389--462. \MR{1627669}

\end{thebibliography}
\end{document}